\documentclass[11pt]{article}

	\usepackage{mathtools}
	\usepackage{amssymb}
	\usepackage{amsthm}
	\usepackage{amsmath}
	\usepackage{mathrsfs}
	\usepackage{tikz}
	\usepackage{tikz-cd}
		\usetikzlibrary{decorations.pathmorphing}
	\usepackage{xcolor} 
	\usepackage{stmaryrd}
	\usepackage{bbm}
	\usepackage{adjustbox}	
	\usepackage{setspace}
	\usepackage{geometry}
	\usepackage{array} 
		\newlength\mylen
		\newcolumntype{C}{>{\hfil$}p{\mylen}<{$\hfil}}
	\usepackage{graphicx}
		\graphicspath{ {./Images/} }
	\usepackage{caption}
	\usepackage{subcaption}
	\usepackage{longtable}

	\newlength{\cellsize} 
	\newsavebox{\cell}
	\sbox{\cell}{\begin{picture}(18,18)
	\put(0,0){\line(1,0){18}}
	\put(0,0){\line(0,1){18}}
	\put(18,0){\line(0,1){18}}
	\put(0,18){\line(1,0){18}}
	\end{picture}}
	\newcommand\cellify[1]{\def\thearg{#1}\def\nothing{}%
	\ifx\thearg\nothing
	\vrule width0pt height\cellsize depth0pt\else
	\hbox to 0pt{\usebox{\cell} \hss}\fi%
	\vbox to \cellsize{
	\vss
	\hbox to \cellsize{\hss$#1$\hss}
	\vss}}
	\newcommand\tableau[1]{\vtop{\let\\\cr
	\baselineskip -16000pt \lineskiplimit 16000pt \lineskip 0pt
	\ialign{&\cellify{##}\cr#1\crcr}}}

	\newcommand{\red}[1]{\textcolor{red}{\textit{#1}}}
	\newcommand{\blue}[1]{\textcolor{blue}{\textit{#1}}}
	\newcommand{\green}[1]{\textcolor{green}{\textit{#1}}}

\newcommand{\pair}[1]{\left\langle#1\right\rangle}

\newcommand{\ldef}[1]{\textcolor{purple}{\textit{#1}}}

\DeclareMathOperator{\Ind}{Ind}
\DeclareMathOperator{\Endo}{End}

\DeclareMathOperator{\triv}{triv}
\DeclareMathOperator{\sign}{sign}

\DeclareMathOperator{\Tilt}{Tilt}

\DeclareMathOperator{\Parity}{Parity}
\DeclareMathOperator{\mix}{mix}
\DeclareMathOperator{\BS}{BS}

\DeclareMathOperator{\Pmix}{P^{\mix}}
\DeclareMathOperator{\Pmixw}{P^{\mix}_{Wh, \textit{I}}}
\DeclareMathOperator{\DBE}{\sD_{\text{BE}}}
\DeclareMathOperator{\DRE}{\sD_{\text{RE}}}
\DeclareMathOperator{\DRA}{\sD_{\text{RA}}}

\DeclareMathOperator{\C}{\mathbb{C}}
\DeclareMathOperator{\Q}{\mathbb{Q}}
\DeclareMathOperator{\Z}{\mathbb{Z}}
\DeclareMathOperator{\bD}{\mathbb{D}}

\DeclareMathOperator{\bF}{\mathbb{F}}
\DeclareMathOperator{\bK}{\mathbb{K}}
\DeclareMathOperator{\bO}{\mathbb{O}}
\DeclareMathOperator{\bP}{\mathbb{P}}

\DeclareMathOperator{\h}{\mathfrak{h}}

\DeclareMathOperator{\cE}{\mathcal{E}}

\DeclareMathOperator{\cP}{\mathcal{P}}
\DeclareMathOperator{\cT}{\mathcal{T}}

\DeclareMathOperator{\sD}{\mathscr{D}}

\DeclareMathOperator{\tA}{\textbf{A}}
\DeclareMathOperator{\tB}{\textbf{B}}
\DeclareMathOperator{\tC}{\textbf{C}}
\DeclareMathOperator{\tD}{\textbf{D}}
\DeclareMathOperator{\tE}{\textbf{E}}
\DeclareMathOperator{\tF}{\textbf{F}}
\DeclareMathOperator{\tG}{\textbf{G}}

\DeclareMathOperator{\tI}{\textbf{I}}
\DeclareMathOperator{\tR}{\textbf{R}}

\DeclareMathOperator{\ux}{\underline{\textit{x}}}
\DeclareMathOperator{\uy}{\underline{\textit{y}}}
\DeclareMathOperator{\uz}{\underline{\textit{z}}}
\DeclareMathOperator{\ue}{\underline{\textit{e}}}

\DeclareMathOperator{\uw}{\underline{\textit{w}}}
\DeclareMathOperator{\uu}{\underline{\textit{u}}}
\DeclareMathOperator{\ur}{\underline{\textit{r}}}
\DeclareMathOperator{\ustar}{\underline{\star}}

\DeclareMathOperator{\Hom}{Hom}
\DeclareMathOperator{\rank}{rank}
\DeclareMathOperator{\grrk}{rk^{\bullet}}
\DeclareMathOperator{\id}{id}
\DeclareMathOperator{\ch}{ch}
\DeclareMathOperator{\Sym}{Sym}
\DeclareMathOperator{\alphac}{\check{\alpha}}
\DeclareMathOperator{\For}{For}

\DeclareMathOperator{\ua}{\uparrow}
\DeclareMathOperator{\da}{\downarrow}

\newtheorem{thm}{Theorem}[section]

\newtheorem{lem}[thm]{Lemma}
\newtheorem{prop}[thm]{Proposition}
\newtheorem{conj}[thm]{Conjecture}

\theoremstyle{remark}
\newtheorem{rem}[thm]{Remark}
\newtheorem{exmp}[thm]{Example}

\title{(Co)Minuscule Hecke categories}
\author{J. Baine}
\date{}
\begin{document}

	\maketitle
	\begin{abstract} 
	
	We determine the $p$-Kazhdan-Lusztig bases for antispherical (co)minuscule Hecke categories in all characteristics, and for spherical (co)minuscule Hecke categories in good characteristic. 
	This is achieved using geometric and diagrammatic methods. 
	The $2$-Kazhdan-Lusztig bases of antispherical cominuscule Hecke categories exhibit extremely pathological behaviour. 
	The notions of $p$-small resolutions and $p$-tight elements are introduced and conjecturally explain this behaviour.
	\end{abstract}


\section{Introduction}
\subsection{Parabolic $p$-Kazhdan-Lusztig theory}

	Categories whose Grothendieck groups are isomorphic to Hecke algebras are ubiquitous in mathematics. 
	They arise in geometry as perverse sheaves on flag varieties \cite{KL80}, in Lie theory as blocks of BGG category $\mathcal{O}$ \cite{BGG76}, and in algebra as Soergel bimodules \cite{Soe07}.
	Such categories are known as Hecke categories.
	Kazhdan-Lusztig theory encodes deep, structural properties of Hecke categories in Kazhdan-Lusztig bases of the corresponding Hecke algebra \cite{KL79}.  
	\par 
	Deodhar \cite{Deo87} introduced parabolic Kazhdan-Lusztig theory, where one instead considers modules for the Hecke algebra obtained by inducing the trivial or sign representations of parabolic subalgebras; these are known as spherical and antispherical modules respectively. 
	For Weyl groups, these modules are categorified by certain classes of `mixed' sheaves on the corresponding partial flag variety. 
	The category of parity sheaves is known as the spherical category as it categorifies the spherical module; the category of mixed tilting sheaves is known as the antispherical Hecke category as it categorifies the antispherical module. 
	\par 
	When these categories are $\Bbbk$-linear over a field of characteristic 0, the indecomposable sheaves categorify the Kazhdan-Lusztig bases of their respective modules. 
	If $\Bbbk$ is a field of characteristic $p>0$, the indecomposable sheaves give rise to the celebrated $p$-Kazhdan-Lusztig bases of their respective modules. 
	$p$-Kazhdan-Lusztig bases are fundamental objects in modular representation theory; they encode the characters of simple and tilting modules for algebraic groups \cite{AMRW19, RW21} and the decomposition numbers of simple modules for symmetric groups \cite{Erd94}.
	\par 
	Despite their wide-ranging importance, categories with well-understood Kazhdan-Lusztig theory are rare.
	Categories with well-understood $p$-Kazhdan-Lusztig theory are rarer.
	At present, we only `understand' the $p$-Kazhdan-Lusztig theory arising from dihedral groups, the closely related universal Coxeter groups \cite{JW17, EL17}, and Hermitian symmetric pairs \cite{BDHN22}.
	In this paper we determine the parabolic $p$-Kazhdan-Lusztig bases arising from (co)minuscule partial flag varieties.  
		
\subsection{Main results}
		
	(Co)Minuscule partial flag varieties are among the best understood projective varieties. 
	They include projective spaces, (isotropic) Grassmannians, quadric hypersurfaces and certain exceptional partial flag varieties. 
	A complete list is given in Table \ref{Table: classification}. 
	We call a category a (co)minuscule Hecke category if it is equivalent to either the category of parity sheaves or the category of mixed tilting sheaves on a (co)minuscule partial flag variety. 
	\par
	The first main result of this paper is the determination of the $p$-Kazhdan-Lusztig bases of antispherical cominuscule Hecke categories. 
		
		\begin{thm}
		\label{Thm: antispherical cominuscule pKL basis}
			For any antispherical cominuscule Hecke category the $p$-Kazhdan-Lusztig basis is the: Bott-Samelson basis when $p=2$; and, the Kazhdan-Lusztig basis when $p \neq 2$. 	
		\end{thm}
		
		The Bott-Samelson basis is introduced in Section \ref{Ssec: Hecke algebras}; it is the crudest upper-bound for the $p$-Kazhdan-Lusztig basis. 
		Theorem \ref{Thm: antispherical cominuscule pKL basis} can be understood as saying the $2$-Kazhdan-Lusztig theory of antispherical cominuscule Hecke categories is as complicated as it can be, whilst the $p$-Kazhdan-Lusztig theory (when $p\neq 2$) is as simple as it can be. 
		The dramatic jump in complexity is illustrated for the Lagrangian Grassmannian in Figures \ref{fig:LG-AS-p0} and \ref{fig:LG-AS-p2}.\footnote{See Section \ref{Ssec: Figures} for details on how to interpret Figures \ref{fig:LG-AS-p0}  and \ref{fig:LG-AS-p2}.} 
		\par 
		We provide three proofs of Theorem \ref{Thm: antispherical cominuscule pKL basis}. 
		When $p>2$, we provide a geometric proof using modular Koszul duality in Section \ref{Ssec: AS Comin Geo}.
		The second proof, which holds for all $p$, explicitly calculates the relevant local intersection forms. 
		The third, which holds for all $p$, explicitly calculates the relevant idempotents projecting from the Bott-Samelson object onto the indecomposable object in the diagrammatic antispherical category. 
		\par 
		In Section \ref{Ssec: p-tightness and p-small}, we introduce the geometric notion of $p$-small resolutions of singularities, which generalises the classical notion of small resolutions.
		We conjecture a geometric explanation for the characteristic 2 phenomena in terms of $2$-small resolutions of Schubert varieties in the Langlands dual, full flag variety. 
		A related notion of $p$-tight elements is also introduced. 
		We suspect these notions will be of independent interest. 
		\par 
		We also determine the $p$-Kazhdan-Lusztig bases of antispherical minuscule Hecke categories. 
		
		\begin{thm}
		\label{Thm: antispherical minuscule pKL basis}
			For any antispherical minuscule Hecke category the $p$-Kazhdan-Lusztig basis is the Kazhdan-Lusztig basis for all $p$.
		\end{thm}
		
		This re-proves the main result of \cite{BDHN22}. 
		Again, we give three proofs: one geometric, one based on local intersections forms, and one by explicitly calculating idempotents. 
		The geometric proof in Section \ref{Ssec: AS Min Geo} is extremely short, utilising Koszul duality and certain classically known isomorphisms of minuscule partial flag varieties.  
		\par 
		Classical antispherical Kazhdan-Lusztig polynomials for (co)minuscule partial flag varieties satisfy the remarkable `monomial property:' they are always either $0$ or a monic monomial \cite{Bre02,Bre09}. 
		It follows immediately from Theorems \ref{Thm: antispherical cominuscule pKL basis} and \ref{Thm: antispherical minuscule pKL basis} and  \cite{BDFHN23} that this property continues to hold in the modular setting.

		\begin{thm}
		\label{Thm: Monomial property}
			For any antispherical (co)minuscule Hecke category and any $p \geq 0$, every $p$-Kazhdan-Lusztig polynomial is either 0 or a monic monomial. 
		\end{thm}
		
		Indeed, Theorem \ref{Thm: antispherical cominuscule pKL basis} shows that the monomial property is really a property of the Bott-Samelson basis that is inherited by the ($p$-)Kazhdan-Lusztig basis.
		A geometric consequence of this property is each indecomposable Whittaker parity sheaf in a (co)minuscule Hecke category is perverse. 
		\par 
		The final main result of this work is the determination of the $p$-Kazhdan-Lusztig basis of any spherical (co)minuscule Hecke category in good characteristic. 
		
		\begin{thm}
		\label{Thm: Spherical pKL bases}
			For any spherical (co)minuscule Hecke category the $p$-Kazhdan-Lusztig basis is the Kazhdan-Lusztig basis when $p$ is a good prime.
		\end{thm}
		
		This is shown using using categorical arguments. 
		In particular, we construct mixed, parabolic Ringel duality functors in Section \ref{Ssec: Ringel duality}. 
		We then related the graded decomposition numbers of mixed tilting, projective and intersection cohomology sheaves, and parity sheaves on the Langlands dual partial flag variety, generalising \cite[\S 2.7]{AR16a}. 
		Finally we show that, under mild conditions on $p$, if the antispherical $p$-Kazhdan-Lusztig basis is the antispherical Kazhdan-Lusztig basis, then the spherical $p$-Kazhdan-Lusztig basis is the spherical Kazhdan-Lusztig basis. 
		We expect the results of this section will be of independent interest. 
		
		\begin{figure}[!htb]
  			\centering
  			\includegraphics[width=0.675\linewidth]{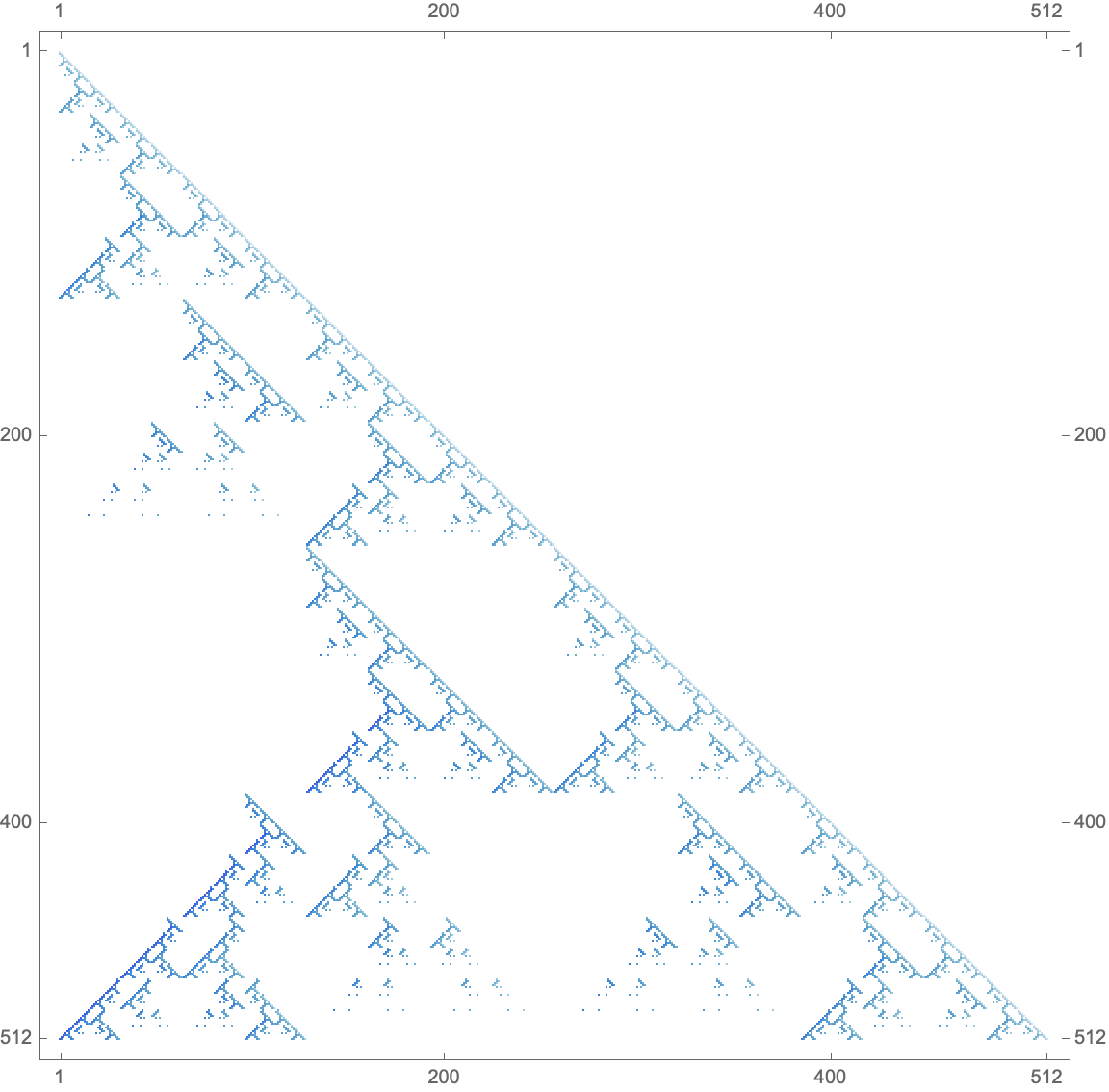}
  			\caption{Antispherical Kazhdan-Lusztig polynomials for the Lagrangian Grassmannian.}
  			\label{fig:LG-AS-p0}
		\end{figure}
		\begin{figure}[!htb]
			\centering
 		 	\includegraphics[width=0.675\linewidth]{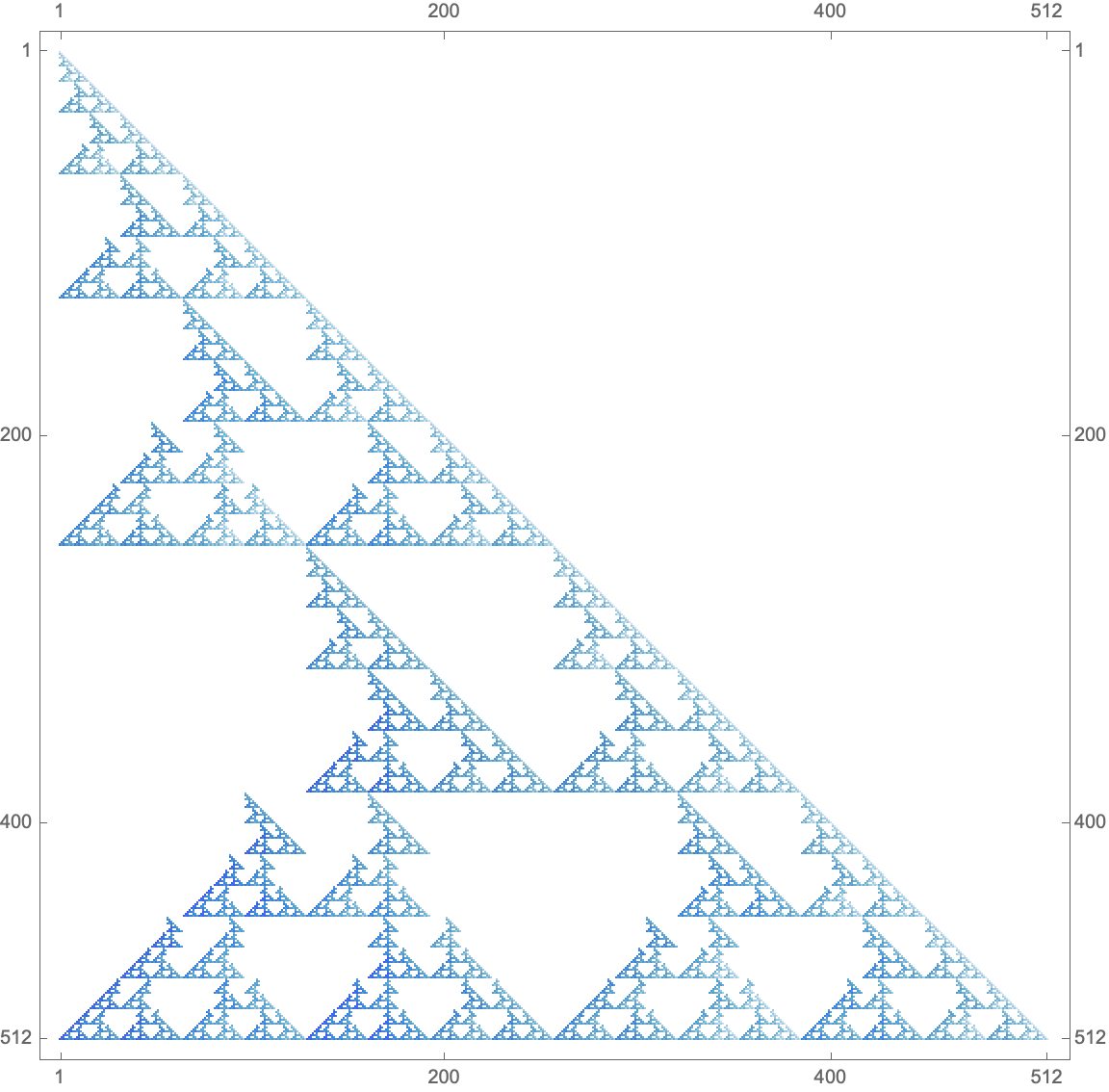}
 			\caption{Antispherical $2$-Kazhdan-Lusztig polynomials for the Lagrangian Grassmannian.}
  			\label{fig:LG-AS-p2}
		\end{figure}
		
\subsection{The Lagrangian Grassmannian}

		The uniformity in the statement of Theorem \ref{Thm: antispherical cominuscule pKL basis} does a disservice to complexity that arises in antispherical cominuscule Hecke categories in characteristic $2$.
		The most interesting behaviour occurs when the partial flag variety $G/P$ is the Lagrangian Grassmannian. 
		That is, when $G= Sp_{2n}$ and $P$ is a parabolic of type $\tA_{n-1}$. 
		\par 
		In this case, the indecomposable objects are parameterised by subsets of $x \subseteq \{1, \dots, n\}$. 
		To each subset $x$, we associate a combinatorially defined subset $E(x) \subset x$ in Section \ref{Ssec: Defect 0}. 
		
		\begin{prop}
			For each $x \subseteq \{ 1, \dots , n\}$, the antispherical $2$-Kazhdan-Lusztig basis element ${}^2 d_x$ decomposes as the following sum of antispherical Kazhdan-Lusztig basis elements $d_y$:
			\begin{align*}
				{}^2 d_x  = \sum_{y \subseteq E(x)} d_{x \backslash y}
			\end{align*}
			where $x \backslash y$ denotes the complement of $y$ in $x$. 
		\end{prop}
		
		When $x = w_0 w_I$ we have $\ell(x) = n(n+1)/2$ and $|E(x)| = \lfloor n/2\rfloor$. 
		Varying $n$ gives the first family of elements in finite Weyl groups where the number of non-zero terms in the $p$-Kazhdan-Lusztig basis (when expressed in the Kazhdan-Lusztig basis) grows exponentially relative to the length of the element. 
		The jump in complexity becomes more apparent when one considers the change-of-basis matrices between the $p$-Kazhdan-Lusztig basis and the standard basis. 
		This is depicted in Figures \ref{fig:LG-AS-p0} and \ref{fig:LG-AS-p2}; see Section \ref{Ssec: Figures} for further details on interpreting these figures.
		\par 
		In Section \ref{Ssec: Defect 0} we explicitly describe all degree-$0$ endomorphisms of indecomposable objects in the antispherical category. 
		Consequently, determine a presentation for the degree-0 endomorphism algebra. 
		
		\begin{prop}
		\label{Prop: Endo alg}
			If the antispherical Hecke category is $\Bbbk$-linear and $2 \notin \Bbbk^{\times}$, then, for any $x \in W^I$, the degree-0 endomorphism algebra of the corresponding indecomposable object is 
			\begin{align*}
				\frac{\Bbbk[\varphi_E \vert E \subseteq E(x)]}
				{\left(\varphi_E \varphi_{E'} - 2^{|E \cap E'|} (-1)^{\ell(E \cap E')/2 } \varphi_{E \cup E'} \right)}.
			\end{align*} 
		\end{prop}
		
		More generally, for any complete local ring $\Bbbk$, Proposition \ref{Prop: Endo alg} gives a presentation for the degree-0 endomorphism algebra of any Bott-Samelson object indexed by a reduced expression for $x$. 
		Using this presentation we can explicitly describe the projector onto the maximal indecomposable summand. 
		
		\begin{prop}
			If $2 \in \Bbbk^{\times}$, then for any $x$ and any Bott-Samelson indexed by a reduced expression of $x$, the idempotent corresponding to the projector onto the maximal indecomposable summand is given by:
			\begin{align*}
				e_x 
				= 
				\sum_{E \subseteq E(x)} 
				\frac{(-1)^{\ell(E)/2} }{(-2)^{|E|}} \varphi_E
				=
				\prod_{t \in E(x)} \left( 1 - \frac{(-1)^{t/2}}{2} \varphi_{\{ t\}} \right).
			\end{align*}
		\end{prop}
	   
\subsection{Structure}

	We begin by fixing notation for Coxeter systems and Hecke algebras in Section \ref{Sec: Coxeter systems}. 
	Minuscule and cominuscule partial flag varieties are introduced in Section \ref{Sec: CoMin partial flags}.
	In Section \ref{Sec: Geo Hecke cats} we recall various geometric incarnations of Hecke categories. 
	The geometric proofs of Theorems \ref{Thm: antispherical cominuscule pKL basis} and \ref{Thm: antispherical minuscule pKL basis} are given in Section \ref{Sec: Geo anti}.
	Diagrammatic Hecke categories are introduced in Section \ref{Sec: Diagrammatic Hecke Cats} and used to give the remaining proofs of Theorems \ref{Thm: antispherical cominuscule pKL basis} and \ref{Thm: antispherical minuscule pKL basis} in Section \ref{Sec: Dia Cominuscule}. 
	In Section \ref{Sec: Decomposition numbers} we study decomposition numbers for various classes of objects under the (derived) extension of scalars and modular reduction functors. 
	In Section \ref{Sec: Spherical} we deduce Theorem \ref{Thm: Spherical pKL bases} as an application of the results in Section \ref{Sec: Decomposition numbers}. 
	Finally, we introduce the notions of $p$-tight elements and $p$-small resolutions in Section \ref{Sec: p-tightness}, and conjecture a geometric explanation for the characteristic 2 phenomena occurring in antispherical cominuscule Hecke categories.

\section{Coxeter systems and Hecke algebras}
\label{Sec: Coxeter systems}

\subsection{Coxeter systems}

	We fix notation relating to Coxeter systems. 
	A standard reference is \cite{BB08}. 
	\par 
	Let $(W,S)$ be a Coxeter system with length function $\ell$ and Bruhat order $\leq$. 
	We always assume $W$ is a finite Weyl group with longest element $w_0$.  
	For any $I \subseteq S$ we have the standard\footnote{As we exclusively consider standard parabolic subgroups, we will omit the word `standard'. } parabolic subgroup $W_I$ with longest element $w_I$. 
	We denote by ${}^IW$ (resp. $W^I$) the minimal length coset representatives for $W_I \backslash W$ (resp. $W/W_I$). 
	\par 
	An \ldef{expression} is a sequence $ \ux = (s_1 , \dots, s_k)$ where $s_i \in S$. 
	Expressions are always denoted by underlines, i.e. $\ux$.
	The length of an expression $\ell(\ux)$ is its length as a sequence. 
	An expression is said to be reduced if $\ell(\ux)= \ell(\ux_{\bullet})$, where $(-)_{\bullet}$ is the map into $W$ induced by multiplication. 
	The empty expression $\underline{\emptyset}$ is, by definition, a reduced expression for the identity element $\id \in W$.
	\par 
	Later we will consider light-leaves for the antispherical Hecke category.
	These are constructed from the data of an $I$-parabolic Bruhat strolls associated to a subexpression.
	We recall the following notions from \cite{Deo90,EW16,LW22}.
	\par 
	A \ldef{subexpression} $\uy \subseteq \ux$ is a subsequence $\uy = (s_{i_1}, \dots , s_{i_j})$ of $\ux$. 
	Alternatively, we identify a subexpression $\uy \subseteq \ux$ with the decorated sequence $\ux^e$ where each term is decorated with either 1 (indicating inclusion) or 0 (indicating omission). 
	When explicitly dealing with subexpressions, we write
	\begin{align*}
			\uy 
			= 
			\ux ^{e}
			= 
			\begin{smallmatrix}
				e_1		&	e_2		&	\dots	&	e_k
			\\
				s_{1}	&	s_{2}	&	\dots 	&	s_{k}
			\end{smallmatrix}
			. 
	\end{align*}
	Each subexpression has an associated \ldef{Bruhat stroll}, which is a decoration of $\ux^e$ defined as follows. 
	Set $\ux_{0} := \underline{\emptyset}$, $\ux_{i} := (s_1 , s_2 , \dots , s_i)$, $e_{i}  := (e_1 , e_2 , \dots , e_i)$, and $x_{i} := (\ux_{i}^{e_i})_{\bullet}$ where $i \leq \ell(\ux)$.
	That is, $\ux_{i}$ is $\ux^e$ truncated to the first $i$ terms. 
	The $i$-th term in the Bruhat stroll is defined to be
	\begin{align*}
		&\ua \text{ if } x_{i-1}  s_i > x_{i-1} \text{ and,}
		\\
		&\da \text{ if } x_{i-1}  s_i < x_{i-1}.
	\end{align*}
	The Bruhat stroll of a subexpression will typically be written as
	\begin{align*}
			\uy 
			=
			\ux^e 
			=
			\begin{smallmatrix}
				\ua		&	\ua		&	\da	&	\dots 	&	\ua
			\\
				0			&	1			&	1			&	\dots	&	e_k
			\\
				s_1			&	s_2			&	s_3			&	\dots 	&	s_k
			\end{smallmatrix}
			. 
	\end{align*}
	A Bruhat stroll is \ldef{$I$-parabolic} if $x_{i-1}s_i \in {}^I W$ for each $1 \leq i \leq \ell(\ux)$.
	The \ldef{defect} $\text{df}$ of a subexpression $\ux^e$ is defined as
	\begin{align*}
			\text{df}(\ux^e) = 
			\left(
			\# \begin{smallmatrix}\ua \\ 0	\end{smallmatrix}
			\text{ in the Bruhat stroll of } \ux^e 
			\right)
			- 
			\left(
			\# \begin{smallmatrix}\da \\ 1	\end{smallmatrix}
			\text{ in the Bruhat stroll of } \ux^e 
			\right)
			.
	\end{align*}
	
	Finally, recall that $x \in W$ is said to be \ldef{fully commutative}, in the sense of \cite{Ste96}, if no reduced expression of $x$ contains a consecutive subexpression that is a reduced expression for the longest element of a dihedral group of type $\tI_{2}(m)$ where $(m\geq 3)$. 

\subsection{Hecke algebras}
\label{Ssec: Hecke algebras}
 
 	We now fix notation for Hecke algebras. 
 	We use the normalisation in \cite{Soe97}, and the notation of \cite{Wil18}. 
	\par 
	The Hecke algebra $H$ associated to the Coxeter system $(W,S)$ is the unital, associative $\Z[v,v^{-1}]$-algebra on the symbols $\{ \delta_x  \,|\, x \in W \}$ subject to the relations:
		\begin{align*}
			\delta_x \delta_{y} = &\delta_{xy} 
			&& \text{if } \ell(x) + \ell(y) = \ell(xy) \text{, and}
			\\
			(\delta_s + v)(\delta_s -v^{-1}) = &0
			&& \text{for all } s \in S.
		\end{align*}	
	It follows from \cite{Tits69} that $\{ \delta_x \,|\, x \in W\}$ is a basis of $H$, known as the standard basis. 
	For any $I \subseteq S$, we denote by $H_I$ the corresponding Hecke subalgebra of $H$. 
	\par 
	Fix $I \subseteq S$. 
	Denote by $\triv$ (resp. $\sign$) the rank 1 $H_I$-module where $\delta_s$ acts by $v^{-1}$ (resp. $-v$) for each $s \in I$. 
	Induction produces the \ldef{spherical module} ${}^I M := \Ind_{H_I}^{H} \triv$ and \ldef{antispherical module} ${}^I N := \Ind_{H_I}^{H} \sign$ which are right $H$-modules.
	Both modules have a standard basis which we abusively denote by $\{ \delta_x^I  ~\vert~ x \in {}^I W \}$, where $\delta_x^I := 1 \otimes \delta_x$.

\subsection{Kazhdan-Lusztig bases}

	The Kazhdan-Lusztig involution is the $\Z$-linear involution uniquely characterised by the properties $\overline{v} = v^{-1}$ and $\overline{\delta_x} = \delta_{x^{-1}}^{-1}$ for each $x \in W$. 
	It induces an involution on ${}^I M$ and ${}^I N$. 	
	Each module has \ldef{Kazhdan-Lusztig basis} 
	\begin{align*}
		\{ b_x  ~\vert~ x \in  W \} \subset H, 
		&&
		\{ c_x  ~\vert~ x \in {}^I W \} \subset {}^I M,
		&&
		\{ d_x  ~\vert~ x \in {}^I W \} \subset {}^I N,
	\end{align*}
	
	\noindent
	uniquely characterised by being: self-dual $\overline{a_x}=a_x$; and positively graded $a_x \in \delta_x + \sum_{y<x} v\Z[v]\delta_y$, where $a \in \{b,c,d\}$ and $x \in W$ or $x \in {}^IW$ as appropriate, \cite{KL79,Deo87}.  
	The corresponding \ldef{Kazhdan-Lusztig polynomials} $h_{y,x}$, $m_{y,x}$, and $n_{y,x}$ are defined by the equalities:
	\begin{align*}
		b_x = \sum_{y} h_{y,x} \delta_y
		&&
		c_x = \sum_{y} m_{y,x} \delta^{I}_y 
		&&
		d_x = \sum_{y} n_{y,x} \delta^{I}_y. 
	\end{align*}

\subsection{Bott-Samelson bases }

	For any expression $\ux = (s_1, \dots , s_k)$ define $b_{\ux} := b_{s_1} \dots b_{s_k}$. 
	Similarly define $c_{\ux} = c_{\id} \cdot b_{\ux}$ and $d_{\ux} = d_{\id} \cdot b_{\ux}$. 
	If a reduced expression $\ux$ of $x$ is fixed for each $x \in W$, then a standard uni-triangularity argument shows 
	\begin{align*}
		\{ b_{\ux}  ~\vert~ x \in  W \} \subset H, 
		&&
		\{ c_{\ux}  ~\vert~ x \in {}^I W \} \subset {}^I M,
		&&
		\{ d_{\ux}  ~\vert~ x \in {}^I W \} \subset {}^I N,
	\end{align*}
	are each bases of their respective modules. 
	Each will be called a \ldef{Bott-Samelson basis} of its module. 
	Typically the choice of reduced expression is extremely important; $b_{\ux}$ is independent of the choice of reduced expression if and only if $x$ is fully commutative. 
	Fortunately, we only ever consider fully commutative elements, so we can safely ignore this subtlety. 
	\par 
	Deodhar's defect formula expresses a Bott-Samelson element in terms of the standard basis \cite{Deo90, LW22}. 
	The versions we require are:
	\begin{align*}
			b_{\ux} 
			= 
			\sum_{\ux^e \subseteq \ux} v^{\text{df}(\ux^e)} 
				\delta_{\ux^{e}_{\bullet}}
			&&
			\text{ and }
			&&
			 d_{\ux} 
			 = 
			 \sum_{\substack{\ux^e \subseteq \ux \\ \ux^e \text{is } I\text{-antispherical}}} 
			 	v^{\text{df}(\ux^e)} \delta_{\ux^{e}_{\bullet}}^{I}
	\end{align*}
	where $\text{df}(\ux^e)$ denotes the defect of the subexpression. 

\subsection{$p$-Kazhdan-Lusztig bases} 
	
	The $p$-Kazhdan-Lusztig bases were introduced in \cite{Wil15, JW17, RW18}.
	Unlike the preceding bases they are not intrinsic to the Hecke algebra.
	Rather, the \ldef{$p$-Kazhdan-Lusztig bases}
	\begin{align*}
		\{ {}^p b_x  ~\vert~ x \in  W \} \subset H, 
		&&
		\{ {}^p c_x  ~\vert~ x \in {}^I W \} \subset {}^I M,
		&&
		\{{}^p d_x  ~\vert~ x \in {}^I W \} \subset {}^I N,
	\end{align*}
	arise as the images under the character map of the indecomposable parity sheaves in the Hecke categories introduced in Section \ref{Sec: Geo Hecke cats}.
	The corresponding \ldef{$p$-Kazhdan-Lusztig polynomials} ${}^ph_{y,x}$, ${}^pm_{y,x}$, and ${}^pn_{y,x}$ are defined by the equalities:
	\begin{align*}
		{}^p b_x = \sum_{y} {}^ph_{y,x} \delta_y
		&&
		{}^pc_x = \sum_{y} {}^pm_{y,x} \delta^{I}_y 
		&&
		{}^pd_x = \sum_{y} {}^pn_{y,x} \delta^{I}_y. 
	\end{align*}
	The following basic properties of $p$-Kazhdan-Lusztig polynomials are shown in \cite[\S4]{JW17}. 
	Let $x \in W$ be arbitrary, and fix a reduced expression $\ux$ of $x$, then
	
	\begin{tabular}{cl}
		(a) & we have ${}^p b_x \in b_x + \sum_{y<x} \Z_{\geq 0} [v,v^{-1}] b_y$, and 
		\\
		(b) & we have $b_{\ux} \in {}^p b_x + \sum_{y<x} \Z_{\geq 0} [v,v^{-1}] {}^p b_y$.
	\end{tabular}
	
	\noindent
	Analogous statements hold for ${}^p c_x$ and ${}^p d_x$. 
	These allow us to obtain crude bounds for $p$-Kazhdan-Lusztig basis purely using the Hecke algebra. 
	For any two $h,h' \in H$, we write $h \preceq h'$ if $h'-h \in \sum_{y} \Z_{\geq 0}[v,v^{-1}] b_y$.

	\begin{lem}
	\label{Lem: pKL bound}
		For any $x \in W$ and any reduced expression $\ux$ of $x$, we have $b_{x} \preceq {}^p b_x  \preceq b_{\ux}$ for all $p$.
		In particular, if $b_x = b_{\ux}$ then $b_x = {}^p b_x  = b_{\ux}$.
	\end{lem}
	\begin{proof}
		This is immediate from properties (a) and (b) listed above. 
	\end{proof}
	
	The obvious analogue of Lemma \ref{Lem: pKL bound} holds for ${}^p c_{x}$ and ${}^p d_x$. 

\section{(Co)Minuscule partial flag varieties}
\label{Sec: CoMin partial flags}

	We begin by fixing notation for algebraic groups, then we recall the definition of (co)minuscule partial flag varieties and their classification.  
	Finally we note various properties of these varieties that will be exploited in later sections. 
	\par 
	For a thorough account of the geometry of (co)minuscule partial flag varieties see  \cite[\S 9]{BL00}.

\subsection{Algebraic groups and flag varieties}
\label{Ssec: Alg groups}

	Fix a simple root datum $(R \subset X, \check{R} \subset \check{X})$ and a set of positive roots $R_+ \subset R$. 
	Associated to this data is the Chevalley group scheme $G_{\Z}$, Borel subgroup scheme $B_{\Z}$ and maximal torus subgroup scheme $T_{\Z}$. 
	Let $U_{\Z} \subset B_{\Z}$ be the unipotent radical subgroup scheme of $B_{\Z}$.
	The choice of positive roots endows the Weyl group $W$ of $G$ with a set of simple reflections. 
	For any $I \subseteq S$, there is an associated parabolic subgroup $P_{I,\Z} $.
	When $I$ is clear, we write $P_{\Z}$. 
	Denote by $G$, $B$, $T$, $U$, $P$ the corresponding schemes over $\C$ obtained via extension of scalars.
	By our assumptions, $G$ is a simple algebraic group. 
	\par 
	Given $G$, the Langlands dual group $\check{G}$ is the group associated to the root datum $( \check{R} \subset \check{X}, R \subset X)$. 
	We denote the corresponding subgroup schemes of $\check{G}$ by $\check{B}$, $\check{T}$, $\check{U}$, $\check{P}$, etc. 
	\par 
	Finally, the partial flag variety $G/P$ admits a Bruhat stratification, i.e. a stratification $B$-orbits:
	\begin{align*}
			G/P = \bigsqcup_{x \in W^I} B \cdot xP/P.
	\end{align*} 
	The closure of each strata $X_x := \overline{B \cdot xP/P} = \sqcup_{y\leq x } B \cdot yP/P$	is a Schubert variety. 
		
\subsection{Classification of (co)minuscule flag varieties}	
\label{Ssec: classification}
 	
 	We now recall the classification of (co)minuscule flag varieties. 
 	\par 
 	Fix a simple algebraic group $G$.  
 	A fundamental weight $\varpi \in X$ for $G$ is said to be: \ldef{minuscule} if $\pair{\alphac, \varpi} \leq 1$ for all $\alpha \in R_+$; and \ldef{cominuscule} if $\pair{\alphac_0 , \varpi} =1$, where $\alpha_0$ is the longest root.
 	Various equivalent definitions are noted in \cite[\S 2.11.15]{BL00}.   
 	A maximal parabolic subgroup  $P_I$ is said to be \ldef{(co)minuscule} when $I = S\backslash \{ s_{\varpi} \}$, where $s_{\varpi}$ is the simple reflection corresponding to a (co)minuscule weight.  
 	The variety $G/P_I$ is a \ldef{(co)minuscule flag variety} if $P_I$ is a (co)minuscule parabolic subgroup. 
 	Note that minuscule and cominuscule flag varieties are related by Langlands duality: if $G/P$ is minuscule then $\check{G}/\check{P}$ is cominuscule and vice versa. 
 	\par 
 	(Co)Minuscule partial flag varieties arise from seven families and two exceptional cases, which are listed in Table \ref{Table: classification}.

		\begin{table}[ht!]
		\begin{center}
		\bgroup
		\def\arraystretch{1.5}
		\begin{tabular}{| c c c c |} 
		\hline
				$G$
			& 
				$P$
			& 
				Classification
			& 
				Description
		\\ 
		\hline\hline
				$A_{n}$
			& 
				
				\begin{tikzpicture}
				\draw[fill=black] 
						(0,0) circle [radius=.1]  --
						(0.5,0) circle [radius=.1];
					\draw[fill=black] (0.5,0) -- (1,0);
					\draw[fill=black] (1,0) -- (1.5,0);
					\draw[fill=white]	
						(1,0) circle [radius=.1] ;
					\draw[fill=black]	
						(1.5,0) circle [radius=.1] 	--
						(2,0) circle [radius=.1] ; 
					\draw[fill=white]
						(2,0.25) circle [radius=.0001];
				\end{tikzpicture}
			& 
				Both
			&
				Grassmannian
		\\ 
		\hline
				$B_{n}$
			& 
				
				\begin{tikzpicture}
					\draw[fill=black]
				 		(0,0) circle [radius=.1]; 
					\draw (0,0.05) -- (0.5,0.05);
					\draw (0,-0.05) -- (0.5,-0.05);
					\draw (0.2,0.1) -- (0.3,0);
					\draw (0.2,-0.1) -- (0.3,0);
					\draw[fill=black] 
						(0.5,0) circle [radius=.1] --
						(1,0) circle [radius=.1];
					\draw (1,0) -- (1.5,0);
					\draw[fill=black]
						(1.5,0) circle [radius=.1] --
						(2,0) circle [radius=.1];
					\draw[fill=white]
						(0,0) circle [radius=.1];
					\draw[fill=white]
						(2,0.25) circle [radius=.0001]; 
				\end{tikzpicture}
			& 
				minuscule
			&
				Odd orthogonal  Grassmannian
			\\
		\hline
				$B_{n}$
			& 
				
				\begin{tikzpicture}
					\draw[fill=black]
				 		(0,0) circle [radius=.1]; 
					\draw (0,0.05) -- (0.5,0.05);
					\draw (0,-0.05) -- (0.5,-0.05);
					\draw (0.2,0.1) -- (0.3,0);
					\draw (0.2,-0.1) -- (0.3,0);
					\draw[fill=black] 
						(0.5,0) circle [radius=.1] --
						(1,0) circle [radius=.1];
					\draw (1,0) -- (1.5,0);
					\draw[fill=black]
						(1.5,0) circle [radius=.1] --
						(2,0) circle [radius=.1];
					\draw[fill=white]
						(2,0) circle [radius=.1];
					\draw[fill=white]
						(2,0.25) circle [radius=.0001]; 
				\end{tikzpicture}
			& 
				cominuscule
			&
				Odd dimensional quadric 
		\\
		\hline
				$C_{n \geq 3}$
			& 
				
				\begin{tikzpicture}
					\draw[fill=black]
				 		(0,0) circle [radius=.1]; 
					\draw (0,0.05) -- (0.5,0.05);
					\draw (0,-0.05) -- (0.5,-0.05);
					\draw (0.3,0.1) -- (0.2,0);
					\draw (0.3,-0.1) -- (0.2,0);
					\draw[fill=black] 
						(0.5,0) circle [radius=.1] --
						(1,0) circle [radius=.1];
					\draw (1,0) -- (1.5,0);
					\draw[fill=black]
						(1.5,0) circle [radius=.1] --
						(2,0) circle [radius=.1];
					\draw[fill=white]
						(0,0) circle [radius=.1];
					\draw[fill=white]
						(2,0.25) circle [radius=.0001]; 
				\end{tikzpicture}
			& 
				cominuscule
			&
				Lagrangian Grassmannian
		\\
		\hline
				$C_{n \geq 3}$
			& 
				
				\begin{tikzpicture}
					\draw[fill=black]
				 		(0,0) circle [radius=.1]; 
					\draw (0,0.05) -- (0.5,0.05);
					\draw (0,-0.05) -- (0.5,-0.05);
					\draw (0.3,0.1) -- (0.2,0);
					\draw (0.3,-0.1) -- (0.2,0);
					\draw[fill=black] 
						(0.5,0) circle [radius=.1] --
						(1,0) circle [radius=.1];
					\draw (1,0) -- (1.5,0);
					\draw[fill=black]
						(1.5,0) circle [radius=.1]--
						(2,0) circle [radius=.1];
					\draw[fill=white]
						(2,0) circle [radius=.1] ;
					\draw[fill=white]
						(2,0.25) circle [radius=.0001]; 
				\end{tikzpicture}

			& 
				minuscule
			&
				Projective space
		\\ 
		\hline
				$D_{n \geq 4}$
			& 
				
				\begin{tikzpicture}
					\draw[fill=black]
						(0,0.25) circle [radius=.1]    --++ (-25:0.55)
						(0,-0.25) circle [radius=.1]  --++ (25:0.55); 
					\draw[fill=white]
						(0,0.25) circle [radius=.1]; 
					\draw[fill=black] 
						(0.5,0) circle [radius=.1]  --
						(1,0) circle [radius=.1];
					\draw[fill=black] (1,0) -- (1.5,0);
					\draw[fill=black]
						(1.5,0) circle [radius=.1] --
						(2,0) circle [radius=.1]; 
				\end{tikzpicture} 
			& 
				both
			&
				Spinor variety
		\\
		\hline
				$D_{n \geq 4}$
			& 
				
				\begin{tikzpicture}
					\draw[fill=black]
						(0,0.25) circle [radius=.1]  --++ (-25:0.55)
						(0,-0.25) circle [radius=.1] --++ (25:0.55); 
					\draw[fill=black] 
						(0.5,0) circle [radius=.1] --
						(1,0) circle [radius=.1];
					\draw[fill=black] (1,0) -- (1.5,0);
					\draw[fill=black]
						(1.5,0) circle [radius=.1] --
						(2,0) circle [radius=.1];
					\draw[fill=white]
						(2,0) circle [radius=.1];  
				\end{tikzpicture} 
			& 
				both
			&
				Even dimensional  quadric
		\\ 
		\hline
				$E_6$
			& 
				
				\begin{tikzpicture}
					\draw[fill=black] 
						(0.5,0) circle [radius=.1]  --
						(1,0) circle [radius=.1] ;
					\draw[fill=black] 
						(1.5,0) circle [radius=.1]  --
						(1.5,0.5) circle [radius=.1];
					\draw[fill=black] (1,0) -- (1.5,0);
					\draw[fill=black] (1.5,0) -- (2,0);
					\draw[fill=black]	
						(2,0) circle [radius=.1]	--
						(2.5,0) circle [radius=.1]; 
					\draw[fill=white] 
						(0.5,0) circle [radius=.1];
				\end{tikzpicture}
			& 
				both
			&
				Caley plane
				\\ 
		\hline
				$E_7$
			& 
				
				\begin{tikzpicture}
					\draw[fill=black] 
						(0,0) circle [radius=.1] --
						(0.4,0) circle [radius=.1];
					\draw[fill=black] (0.4,0) -- (0.8,0);
					\draw[fill=black] 
						(0.8,0) circle [radius=.1]  --
						(1.2,0) circle [radius=.1];
					\draw[fill=black] (1.2,0) -- (1.6,0);
					\draw[fill=black]	
						(1.6,0) circle [radius=.1]	--
						(2,0) circle [radius=.1]; 
					\draw[fill=black]	
						(1.2,0) circle [radius=.1]	--
						(1.2,0.5) circle [radius=.1];
					\draw[fill=white] 
						(0,0) circle [radius=.1];
				\end{tikzpicture}
			& 
				both
			&
				Exceptional Freudenthal  variety
				\\ 
		\hline
		\end{tabular}
		\egroup
		\end{center}
		\caption{Minuscule and cominuscule partial flag varieties.}
		\label{Table: classification}
		\end{table}

\subsection{Non-simply laced minuscule flag varieties.}

	It is well-known that when studying Schubert varieties in minuscule flag varieties, it suffices to consider only the those Schubert varieties arising from simply-laced minuscule flag varieties (i.e. when $G$ is simply-laced) by the following Lemma.  
	
	\begin{lem}
	\label{Lem: reduce to simply laced vars}
		For any non-simply-laced minuscule flag variety $G/P$, there is a simply laced minuscule flag variety $G'/P'$ and an isomorphism $f: G/P \rightarrow G'/P'$, which respects the Bruhat stratification of each partial flag variety. 
	\end{lem}
	
	\begin{proof}
		This is \cite[\S 3.1]{BP99}. 
	\end{proof}

	\par
	These isomorphisms are quite explicit. 
	When $G$ is type $\tC_n$ and $P$ is type $\tC_{n-1}$, the variety $G/P$ is complex projective space $\bP^{2n-1}$ and the corresponding $G'$ is always type $\tA_{2n-1}$ with $P'$ type $\tA_{2n-2}$. 
	When $G$ is type $\tB_n$ and $P$ is type $\tA_{n-1}$, the variety $G/P$ is the odd orthogonal Grassmannian, which can be identified with the spinor variety, and the corresponding $G'$ is always type $\tD_{n+1}$ with $P'$ type $\tA_{n}$. 

	\par 
	By an abuse of notation we write $f: W^I \longrightarrow W'^{I'}$ for the map induced on minimal coset representatives; it is described in \cite{BM09}.

\section{Geometric Hecke categories}
\label{Sec: Geo Hecke cats}

	We now introduce various geometric categories which categorify either the spherical or antispherical module. 
	Our treatment will be quite terse; the interested reader should consult \cite{JMW14,AR16a,AR16b,RW18,AMRW19} for more details.  

\subsection{Parity sheaves and the spherical Hecke category}
\label{Ssec: parity sheaves}

	As noted in the introduction, the $p$-Kazhdan-Lusztig basis of the spherical module $M^{I}$ will be categorified by parity sheaves on the partial flag variety $G/P_I$. 
	\par  
	We continue the notation for algebraic groups from Section \ref{Ssec: Alg groups}. 	
	For any of the categories of sheaves considered below, when the ring of coefficients $\Bbbk$ is clear, we omit it from the notation. 
	\par 
	Fix $\Bbbk$, either a complete local ring or a field. We denote by 
	\begin{align*}
		D^b(U \backslash G / P,\Bbbk)
	\end{align*}
	the bounded, derived category of sheaves of $\Bbbk$-modules on $G/P$ which are constructible with respect to the stratification by $B$-orbits.
	It is triangulated, and we denote the shift functor by $(1)$.
	The `stacky' notation is motivated by the following equivalences which are explained in \cite[\S5.3]{AR16a}: 
	(a) the forgetful functor from the $U$-equivariant derived category of constructible sheaves of $\Bbbk$-modules on $G/P$ to $D^b(U \backslash G / P)$ is an equivalence; and 
	(b) the aforementioned $U$-equivariant category is naturally equivalent to the $P^{\text{op}}$-equivariant derived category of constructible sheaves of $\Bbbk$-modules on $U \backslash G$, where the action is induced by right multiplication. 
	\par 
	Parity sheaves were introduced in \cite{JMW14}, they are certain classes of sheaves satisfying cohomological parity vanishing properties; stalks and costalks of indecomposable parity sheaves are concentrated in either even or odd degrees. 
	We denote by
	\begin{align*}
		\Parity(U \backslash G /P, \Bbbk)	
	\end{align*}
	the full subcategory of \ldef{parity sheaves} in $D^b(U \backslash G / P,\Bbbk)$. 
	It is stable under the shift functor $(1)$.
	Since we assume $\Bbbk$ is a complete local ring, $\Parity(U \backslash G /P, \Bbbk)$ is Krull-Schmidt, and isomorphism classes of indecomposable parity sheaves $\cE_x(n)$ are parameterised by $(x,n) \in W^I \times \Z$. 
	We write $\cE_x$ for $\cE_x(0)$.
	Each $\cE_x$ is characterised by the properties:
	
	\begin{tabular}{rl}
		(a)& $\cE_x$ is supported on the Schubert variety $X_x$;\\
		(b)& $\cE_x$ restricted to $B \cdot xP/P$ is isomorphic to $\underline{\Bbbk}_{B\cdot xP/P}(\ell(x))$; and \\
		(c)& the stalks and costalks of $\cE_x$ vanish in degree $d$ where $d \not\equiv \ell(x) \mod 2$.
	\end{tabular}
	
	\noindent
	We note, but will not use, that $\Parity(U \backslash G /P, \Bbbk)$ is a right-module category over the analogously defined $\Parity(P \backslash G /P, \Bbbk)$, where the monoidal action is induced by convolution of sheaves. 
	\par 
	The category $\Parity(U \backslash G /P, \Bbbk)$ categorifies the spherical module.\footnote{More accurately $\Parity(B \backslash G /P, \Bbbk)$ categorifies the spherical $H$-module $M^I$, as convolution endows $\Parity(B \backslash G /P, \Bbbk)$ with the structure of a left-module category over $\Parity(B \backslash G /B, \Bbbk)$. Then $\Parity(U \backslash G /P, \Bbbk)$ categorifies the $\Z[v,v^{-1}]$-module obtained by forgetting the left $H$-action on $M^I$.} 
	In particular, we have an isomorphism of $\Z[v,v^{-1}]$-modules:
	\begin{center}
	\begin{tabular}{cccc}
		$\ch:$
		&$[\Parity(U \backslash G /P , \Bbbk)]$ 
		& $\tilde{\longrightarrow}$ 
		& $M^I$
	\\
		&$[\mathcal{F}]$
		& $\longmapsto$
		& $\displaystyle \sum_{x \in W^I} \sum_{i \in \Z}  \rank_{\Bbbk} H^i (\mathcal{F} _{x P/P}) v^{-\ell(x) - i} \delta_x^I $ 
	\end{tabular}
	\end{center}
	where $\rank_{\Bbbk} H^i (\mathcal{F} _{x P/P})$ denotes the rank of the cohomology of the stalk of $\mathcal{F}$ at the point $xP/P$. 
	Note that if $\Bbbk$ is a field of characteristic $p$ or a complete local ring with residue field of characteristic $p$, then, by definition, we have $\ch [\cE_x] = {}^p c_x$.
	When $I = \emptyset$ then $M^I \cong H$ and $\ch [\cE_x] = {}^p b_x$.
	Finally, if $p=0$ then indecomposable parity sheaves are simple intersection cohomology sheaves, so we obtain the classical Kazhdan-Lusztig basis by \cite{KL80, Deo87}.

\subsection{Whittaker sheaves and the antispherical Hecke category}
\label{Ssec: Whittaker}

	Fix an algebraically closed field $\mathbb{L}$ of characteristic $\ell \neq p$ such that there is a non-trivial additive character $\psi: \Z/\ell\Z \rightarrow \Bbbk^{\times}$. 
	In a deviation from the notation in Section \ref{Ssec: Alg groups}, we have group schemes $G$, $B$, $T$, $U$, and $P_I$, but they are now considered as schemes over $\mathbb{L}$ via extension of scalars.
	In this context, $D^b(H \backslash G / K,\Bbbk)$ now denotes the analogously defined category of sheaves in the \'etale topology.  
	\par 
	Whittaker sheaves arise from an alternate stratification of the \textit{full} flag variety $B\backslash G$. 
	We first need to introduce two more groups: $U^I$ and $U_I^{-}$. 
	The group $U^I$ is the unipotent radical of $P_I$. 
	If $L_I \subset P_I$ denotes the Levi factor, then $U_I^{-}$ is defined to be the unipotent radical of the  Borel subgroup of $L_I$ that is opposite $L_I \cap B$. 
	The desired stratification of $B\backslash G$ is by $ U_I^{-} U^I$-orbits:
	\begin{align*}
			B \backslash G 
			= 
			\bigsqcup_{x \in W}  B \backslash B x \cdot  U_I^- U^I .
	\end{align*} 
	When $x \in W^I$, the stratum $B \backslash B x \cdot  U_I^- U^I$ has dimension $\ell(xw_I)$.
	\par 
	If we fix an isomorphism $U_{s}^- \rightarrow \mathbb{G}_a$ for each $s \in I$, then we obtain a morphism of algebraic groups:
	\begin{align*}
		\chi: 
		 U_I^- U^I
		\longrightarrow 
		U_I^-
		\longrightarrow
		U_I^- / [U_I^-,U_I^-]
		\tilde{\longrightarrow}
		\prod_{s \in I} U_s^-
		\tilde{\longrightarrow}
		(\mathbb{G}_a)^I
		\longrightarrow
		\mathbb{G}_a
	\end{align*} 
	where the third map is the inverse of the natural embedding, and the final map is induced by addition. 
	Fix a non-trivial additive character $\psi: \Z/\ell\Z \rightarrow \Bbbk^{\times}$.
	The Artin-Schreier sheaf $\mathcal{L}_{\psi}$ is the rank-one local system on $\mathbb{G}_a$ which arises as the $\psi$-isotypic component of $\alpha_{*} \underline{\Bbbk}_{\mathbb{G}_a}$, where $\alpha:\mathbb{G}_a \rightarrow \mathbb{G}_a $ is the Artin-Schreier map $\alpha(x)= x^\ell - x $. 
	The sheaf $\chi_I^* \mathcal{L}_{\psi}$ is multiplicative, in the sense of \cite[\S A.1]{AR16a}, so one can consider twisted $(U^IU_I^- ,\chi_I^* \mathcal{L}_{\psi})$-equivariant complexes on $B \backslash G$. 
	\par 
	The triangulated category of twisted $(U_I^- U^I ,\chi_I^* \mathcal{L}_{\psi})$-equivariant complexes on $B \backslash G$ is denoted by 
	\begin{align*}
		D_{Wh,I}(B \backslash G , \Bbbk).
	\end{align*}
	The definition of parity objects carries over to $D_{Wh,I}(B \backslash G , \Bbbk)$, and we denote by 
	\begin{align*}
		\Parity_{Wh,I}(B \backslash G ,\Bbbk)
	\end{align*} 
	the full subcategory of parity objects in $D_{Wh,I}(B \backslash G)$. 
	Since $\Bbbk$ is a complete local ring or a field, the category $\Parity_{Wh,I}(B \backslash G)$ is Krull-Schmidt and isomorphism classes of indecomposable parity sheaves $\cE_x(n)$ are parameterised by $(x,n) \in  W^I \times \Z$ and admit an analogous characterisation to that in Section \ref{Ssec: parity sheaves}. 
	\par 
	The category $\Parity_{Wh,I}(B \backslash G)$ categorifies the antispherical module $N^I$.\footnote{$\Parity_{Wh,I}(B \backslash G)$ can be endowed with the structure of a left-module category over $\Parity(B\backslash G/B)$ where the action is via convolution; see \cite[\S11.1]{RW18}.}
	In particular, there is an isomorphism of $H$-modules given by 
	\begin{center}
	\begin{tabular}{cccc}
		$\ch:$
		&$[\Parity_{Wh,I}(B \backslash G,\Bbbk)]$ 
		& $\tilde{\longrightarrow}$ 
		& $N^I$
	\\
		&$[\mathcal{F}]$
		& $\longmapsto$
		& $\displaystyle \sum_{x \in W } \sum_{i \in \Z}  \rank_{\Bbbk} H^i (\mathcal{F}_{B\backslash Bx}) v^{-\ell(x) - i} \delta_x^I $ 
	\end{tabular}
	\end{center} 
	where $\rank_{\Bbbk} H^i (\mathcal{F}_{B\backslash Bx}) $ denotes the rank of the cohomology of the stalk of $\mathcal{F}$ at the point $B\backslash Bx$.  
	If $\Bbbk$ is a field of characteristic $p$ or a complete local ring with residue field of characteristic $p$, then, by definition, we have $\ch [\cE_x] = {}^p d_x$. 
	When $p=0$ we have $\ch [\cE_x] = d_x$ by \cite{AB09}.

\subsection{A quotient construction of antispherical Hecke categories}
\label{Ssec: Geo Quotient Construction}

	The antispherical module $N^I$ can be constructed as a quotient of the Hecke algebra
	\begin{align*}
		N^I
		\cong 
		H~/~( b_{z} ~|~ z \notin W^I ),
	\end{align*}
	where the image of $b_x$ identifies with $d_x$ when $x \in W^I$, see \cite[\S 1.3]{LW22}. 
	Completely analogously, we can consider the category 
	\begin{align*}
		\Parity(B \backslash G / U) ~/~ \langle \cE_x ~|~ x \notin W^I \rangle_{\oplus, (1)}
	\end{align*} 
	where $\langle \cE_x \,|\, x \notin W^I \rangle_{\oplus, (1)}$ denotes the subcategory of $\Parity(B \backslash G / U)$ additively generated by the $\cE_x$, and closed under the shift functor $(1)$. 
	Morphisms in this category are morphisms in $\Parity(B \backslash G / U)$ modulo those factoring through any object in  $\langle \cE_x \,|\, x \notin W^I \rangle_{\oplus, (1)}$. 
	By \cite[\S 11.5]{RW18} and \cite[\S 6.2]{AMRW19} we have an equivalence 
	\begin{align*}
		Av: \Parity(B \backslash G / U) ~/~ \langle \cE_x ~|~ x \notin W^I \rangle_{\oplus, (1)}
		\longrightarrow
		\Parity_{Wh,I}(B \backslash G,\Bbbk)
	\end{align*}
	satisfying $Av(\cE_x) \cong \cE_x$ for each $x \in W$.

\subsection{Mixed Hecke categories}

	In this Section we recall mixed Hecke categories. 
	These are modular analogues of the category of mixed $\Q_{p}$-sheaves on flag varieties that were studied in \cite{BGS96}. 
	This approach to mixed modular sheaves stems from \cite{AR16a, AMRW19}. 	
	One of the main reasons for considering these categories is that they are related by Koszul duality. 
	\par 
	The \ldef{mixed derived spherical Hecke category} $D^{\mix}(U \backslash G / P)$ is defined as
	\begin{align*}
		D^{\mix}(U \backslash G / P) &:= K^b(\Parity(U \backslash G / P)).
	\end{align*} 
	Similarly, the \ldef{mixed derived antispherical Hecke category} $D^{\mix}_{Wh, I}( G / B)$ is defined as 
	\begin{align*}
		D^{\mix}_{Wh, I}( B \backslash G) &:= K^b(\Parity_{Wh,I}( B \backslash G))
	\end{align*} 
	where, in both cases, $K^b$ denotes the bounded homotopy category. 
	Each category is endowed with an internal shift functor $(1)$ and a cohomological shift functor $[1]$. 
	We define the tate twist $\langle 1 \rangle$ as $(-1)[1]$.
	\par
	Each of the mixed Hecke categories admits a canonical perverse $t$-structure. 
	We will not require the precise definition, instead referring the interested reader to \cite[\S 3]{AR16b} and \cite[\S 6]{AMRW19} in the case of mixed spherical Hecke categories, and \cite[\S 6]{AMRW19} for mixed antispherical Hecke categories. 
	The \ldef{mixed perverse categories}, denoted  
	\begin{align*}
		\Pmix (U \backslash G / P) && \Pmixw (B \backslash G),
	\end{align*}
	arise as the hearts of the perverse $t$-structures on $D^{\mix}(U \backslash G/P)$ and $D^{\mix}_{Wh,I}(B \backslash G)$ respectively. 
	\par 
	We now write $\Pmix$ to mean either of the mixed perverse Hecke categories introduced above, as each of the following statements applies to each category.
	The tate twist $\langle 1 \rangle$ is $t$-exact. 
	The category $\Pmix$ has the structure of a graded highest weight category, in the sense of \cite[\S A]{AR16b},\footnote{Where they are instead called quasihereditary.} where $\langle 1 \rangle$ is the shift functor. 
	The graded highest weight structure implies isomorphism classes of indecomposable tilting $T_x$, standard $\Delta_x$, costandard $\nabla_x$, and simple $L_x$ objects are parameterised (up to shift) by $x \in W^I$.
	Since $W$ is finite $\Pmix$ has enough projectives; we denote a projective cover of $L_x$ by $P_x$.  
	We denote by 
	\begin{align*}
		\Tilt^{\mix}(U \backslash G / P) && \Tilt_{Wh, I}^{\mix}(B \backslash G ) 
	\end{align*}
	the full subcategories of tilting objects in the relevant perverse Hecke category. 
	\par 
	Suppose $\Bbbk$ is a field or complete local ring.
	Modular \ldef{Koszul duality} \cite{AMRW19,RV23} establishes an equivalence of triangulated categories
	\begin{align*}
		\kappa: D^{\mix}(U \backslash G / B) ~~ \tilde{\longrightarrow} ~~ D^{\mix}(\check{B} \backslash \check{G} / \check{U})
	\end{align*}
	satisfying $\kappa \circ \langle 1 \rangle \cong (1) \circ \kappa$ and 
	\begin{align*}
		\kappa(T_x) \cong \check{\cE}_x, 
		&&
		\kappa(\cE_x) \cong \check{T}_x,
		&&
		\kappa(\Delta_x ) \cong \check{\Delta}_x 
		&&
		\kappa(\nabla_x) \cong \check{\nabla}_x. 
	\end{align*}
	Note that the groups $ \check{G} , \check{B}$ and $\check{U}$ are the Langlands dual groups to  $G,B$ and $U$ respectively. 
	Parabolic Koszul duality was also considered in \cite{AMRW19}, where it was shown that if $2 \in \Bbbk^{\times}$, there is an equivalence of triangulated categories
	\begin{align*}
		\kappa: D^{\mix}(U \backslash G / P) ~~ \tilde{\longrightarrow} ~~ D^{\mix}_{Wh,I}(\check{B} \backslash \check{G})
	\end{align*}
	again satisfying $\kappa \circ \langle 1 \rangle \cong (1) \circ \kappa$ and 
	\begin{align*}
		\kappa(T_x) \cong \check{\cE}_x, 
		&&
		\kappa(\cE_x) \cong \check{T}_x,
		&&
		\kappa(\Delta_x ) \cong \check{\Delta}_x, 
		&&
		\kappa(\nabla_x) \cong \check{\nabla}_x. 
	\end{align*}
	We also denote the map induced on Grothendieck groups by $\kappa$. 
	
\section{Geometric antispherical (co)minuscule Hecke categories}
\label{Sec: Geo anti}

	If $\Bbbk$ is a field or complete local ring such that $2 \in \Bbbk^{\times}$, parabolic Koszul duality induces an equivalence 
	\begin{align*}
		\kappa: \Tilt^{\mix}(U \backslash G / P, \Bbbk) ~~ \tilde{\longrightarrow} ~~ \Parity_{Wh,I}(\check{B} \backslash \check{G}, \Bbbk).
	\end{align*}
	For this reason, we call $\Parity_{Wh,I}(\check{B} \backslash \check{G}, \Bbbk)$ an antispherical (co)minuscule Hecke category when $G/P$ is a (co)minuscule flag variety. 
	We continue nomenclature even when $2 \notin \Bbbk$. 
	\par 
	We now prove, using geometric techniques, that the antispherical $p$-Kazhdan-Lusztig basis for antispherical (co)minuscule Hecke categories is the Kazhdan-Lusztig basis whenever $p>2$.

\subsection{Antispherical minuscule Hecke categories}
\label{Ssec: AS Min Geo}

	We begin by showing that we can reduce to the simply-laced case. 
	Throughout this section we always take $\Bbbk$ to be a field or a complete local ring. 
	\begin{lem}
	\label{Lem: reduce to simply laced cats}
		Suppose $2 \in \Bbbk^{\times}$. 
		For any non-simply-laced $\check{G}$ such that $\Parity_{Wh,I}(\check{B} \backslash \check{G}, \Bbbk)$ is minuscule, there is a simply-laced $\check{G}'$ such that $\Parity_{Wh,I'}(\check{B}' \backslash \check{G}', \Bbbk)$ is minuscule and  an equivalence
		\begin{align*}
			\check{f}_* :\Parity_{Wh,I}(\check{B} \backslash \check{G}, \Bbbk) ~~ \longrightarrow ~~ \Parity_{Wh,I'}(\check{B}' \backslash \check{G}', \Bbbk)
		\end{align*}
		satisfying $\check{f}_*(\cE_x) \cong \cE_{f(x)}'$ for each $x \in W^I$.
	\end{lem}
	
	\begin{proof}
		By assumption on $\check{G}$ and Lemma \ref{Lem: reduce to simply laced vars}, there is a simply-laced $G'$, a minuscule flag variety $G'/P'$, and isomorphism $f: G/P \rightarrow G'/P'$ satisfying $B\cdot x P/P \cong B' \cdot f(x) P'/P'$ for each $x \in W^I$. 
		Hence $f$ is a stratified isomorphism, in the sense of \cite[\S2.4]{JMW14}, which induces an equivalence of (constructible) derived categories
		\begin{align*}
			f_* :  
			D_{(B)}^b(G/P) 
			~~ \longrightarrow ~~
			D_{(B')}^b(G'/P').
		\end{align*}
		After identifying $D^b(U \backslash G/P, \Bbbk)$ with $D_{(B)}^b(G/P, \Bbbk)$ via the forgetful functor, this restricts to an equivalence 
		\begin{align*}
			f_* :  
			\Parity(U \backslash G/P) 
			~~ \longrightarrow ~~
			\Parity(U' \backslash G'/P').
		\end{align*}
		The equivalence is extremely naive: the identification of $B\cdot x P/P $ with $ B' \cdot f(x) P'/P'$ induces an isomorphism $f_*( \cE_x) \cong \cE_{f(x)}'$.
		Passing to bounded homotopy categories, and then restricting to the full subcategories of tilting sheaves gives an equivalence
		\begin{align*}
			f_* :  
			\Tilt^{\mix}(U \backslash G/P) 
			~~ \longrightarrow ~~
			\Tilt^{\mix}(U' \backslash G'/P') 
		\end{align*}
		satisfying $f_*(T_x) \cong T_{f(x)}'$. 
		The desired functor is then obtained as the composition 
		\begin{align*}
			\check{f}_*:
			\Parity_{Wh,I}(\check{B} \backslash \check{G}) 
			~ \xrightarrow{\kappa} ~
			\Tilt^{\mix}(U \backslash G / P)
			~ \xrightarrow{f_*} ~
			\Tilt^{\mix}(U' \backslash G'/P')
			~ \xrightarrow{\kappa'} ~
			\Parity_{Wh,I}(\check{B}' \backslash \check{G}') .
		\end{align*}
		which evidently satisfies $\check{f}_*(\cE_x) \cong \cE_{f(x)}'$.
	\end{proof}

	\begin{thm}
	\label{Thm: Min pKL Geom}
		Let $2 \in \Bbbk^{\times}$ and  $\Parity_{Wh,I}(\check{B} \backslash \check{G}, \Bbbk)$ be an antispherical minuscule Hecke category. 
		Then ${}^p d_x = d_x$ for all $x \in W^I$. 
	\end{thm}
	
	\begin{proof}
		By Lemma \ref{Lem: reduce to simply laced cats}, it suffices to consider only the simply-laced cases. 
		However, in each simply-laced case it is known that for each reduced expression $\ux$ of $x \in W^I$ we have $d_x = d_{\ux}$, so we must have $d_x = {}^p d_x = d_{\ux}$ by Lemma \ref{Lem: pKL bound}. 
		For proofs of $d_x = d_{\ux}$, see \cite[Theorem 5.1]{Bre02} for Grassmannians and \cite[Theorem 4.1]{Bre09} for Spinor varieties and even dimensional quadrics, the exceptional cases are easily checked by hand. 
	\end{proof}
 	
 	\begin{rem}
		The condition $p>2$ in Lemma \ref{Lem: reduce to simply laced cats} is a consequence of a technical difficulty in the construction of Koszul duality (and consequently parabolic Koszul duality) in \cite{AMRW19}. 
		In \cite{RV23} Riche and Vay overcome this technical condition, constructing Koszul duality for all $p$.
		However, they do not consider parabolic Koszul duality.  
		If parabolic Koszul duality is constructed for all $p$, the geometric proof of Theorem \ref{Thm: Min pKL Geom} will hold for all $p$.
	\end{rem}
 	
\subsection{Antispherical cominuscule Hecke categories}
\label{Ssec: AS Comin Geo}

	We now determine the $p$-Kazhdan-Lusztig bases for antispherical cominuscule Hecke categories. 
	By Theorem \ref{Thm: Min pKL Geom}, it suffices to consider only the non-simply-laced cases, which are the Lagrangian Grassmannian and odd dimensional quadrics. 
	\par 
	First, recall that if $G$ is a simple algebraic group with root system $R$, then $p$ is a good prime for $G$ if $p$ does not divide any coefficient of the highest root $\alpha \in R$ when expressed as a sum of simple roots. 
	A prime $p$ is said to be very good if either: $G$ is not type $\tA$ and $p$ is good; or $G$ is type $\tA_{n}$ and $p \not| n+1$. 
	See Table \ref{Table: primes} for (very) good primes in each type.

		\begin{table}[h!]
		\begin{center}
		\bgroup
		\def\arraystretch{1.5}
		\begin{tabular}{| c c c c c|} 
		\hline
				Type
			& 
				$\tA_n$
			& 
				$\tB_n$, $\tC_n$ or $\tD_n$
			& 
				$\textbf{G}_2$, $\tF_4$, $\tE_6$ or  $\tE_7$
			& 
				$\tE_8$
		\\ 
		\hline
				Good primes
			& 
				$p>0$
			& 
				$p>2$
			& 
				$p>3$
			& 
				$p>5$
		\\  
		\hline
				Very good primes
			& 
				$p ~\not\vert ~ n+1 $
			& 
				$p>2$
			& 
				$p>3$
			& 
				$p>5$
		\\  
		\hline
		\end{tabular}
		\egroup
		\caption{Good and very good primes for simple algebraic groups.}
		\label{Table: primes}
		\end{center}
		\end{table}
		
	We note the following equivalence, from which we will deduce our result. 
	
	\begin{lem}
	\label{Lem: Other Koszul duality}
		Let $G$ be a simply-connected, simple algebraic group, and $p \in \Bbbk^{\times}$ for each $p$ which is not a very good prime for $G$.
		Then there is an equivalence of triangulated categories
		\begin{align*}
			\varkappa:
			D^{\mix} (U \backslash G / B) 
			\longrightarrow
			D^{\mix} (\check{U} \backslash \check{G} / \check{B}) 
		\end{align*}
		satisfying $\varkappa(\cE_x) \cong \check{\cE}_x$ for all $x \in W$. 
	\end{lem}
	
	\begin{proof}
		The desired equivalence is obtained by composing the two forms of Koszul duality constructed in \cite{AMRW19}. 
		Namely, those in Theorems 5.6 and 7.2.
	\end{proof}
	
	\begin{lem}
	\label{Lem: Other Koszul duality Whit version}
		Let $G$ be simply-connected and non-simply-laced, $\Parity_{Wh,I}(B \backslash G )$ cominuscule, and $2 \in \Bbbk^{\times}$. 
		There is an equivalence
		\begin{align*}
			\varkappa^{Av}: \Parity_{Wh,I}(B \backslash G, \Bbbk ) \longrightarrow \Parity_{Wh,I}(\check{B} \backslash \check{G} , \Bbbk)
		\end{align*}
		satisfying $\varkappa^{Av} (\cE_x) \cong \cE_x$ for each $x \in W$.
	\end{lem}
	
	\begin{proof}
		By the classification of non-simply-laced cominuscule flag varieties, it suffices to ensure $2 \in \Bbbk^{\times}$ to apply Lemma \ref{Lem: Other Koszul duality}. 
		Observe that $\varkappa$ restricts to an equivalence of the graded, additive subcategories
		\begin{align*}
			\varkappa:
			\langle \cE_x ~|~ x \notin W^I \rangle_{\oplus, (1)}
			\longrightarrow
			\langle \check{\cE}_x ~|~ x \notin W^I \rangle_{\oplus, (1)}
		\end{align*}
		Hence $\varkappa$ induces an equivalence of the quotient categories
		\begin{align*}
			\varkappa:
			\Parity (U \backslash G / B) ~/~  \langle \cE_x ~|~ x \notin W^I \rangle_{\oplus, (1)}
			\longrightarrow
			\Parity (\check{U} \backslash \check{G} / \check{B})  ~/~ \langle \check{\cE}_x ~|~ x \notin W^I \rangle_{\oplus, (1)}.
		\end{align*}
		Now recall the equivalence $Av$ from Section \ref{Ssec: Geo Quotient Construction}, and denote by $Av^{-1}$ the quasi-inverse functor. 
		The composition $\check{A}v^{-1} \circ \varkappa \circ Av$ gives the desired equivalence $\varkappa^{Av}$, it is easy to check it satisfies $\varkappa^{Av}(\cE_x) \cong \cE_x$ for all $x \in W$. 
	\end{proof}
	
	We can now deduce the main result of this Section. 
	
	\begin{thm}
	\label{Thm: Comin pKL Geom}
		Let $2 \in \Bbbk^{\times}$ and $\Parity_{Wh,I}(\check{B} \backslash \check{G}, \Bbbk)$ be an antispherical cominuscule Hecke category. 
		Then ${}^p d_x = d_x$ for all $x \in W^I$. 
	\end{thm}
	
	\begin{proof}
		Lemma \ref{Lem: Other Koszul duality Whit version} implies the remaining non-simply-laced cases are equivalent to their Langlands dual Hecke categories (i.e. antispherical minuscule Hecke categories) when $2 \in \Bbbk^{\times}$. 
		Since each simply-laced cominuscule flag variety is also minuscule, the claim follows from Theorem \ref{Thm: Min pKL Geom}.
	\end{proof}
	
	\begin{rem}
		Unlike Theorem \ref{Thm: Min pKL Geom}, parabolic Koszul duality in characteristic 2 would not imply ${}^2 d_x  = d_x$ for the cominuscule case. 
		As we will see, the 2-Kazhdan-Lusztig basis is considerably more complicated. 
		A geometric explanation for the characteristic 2 phenomena is conjectured in Section \ref{Ssec: 2-tightness}.
	\end{rem}

\section{Diagrammatic Hecke categories}	
\label{Sec: Diagrammatic Hecke Cats}

	We now recall the various diagrammatic categories we will use to determine the $p$-Kazhdan-Lusztig bases for the non-simply-laced, antispherical (co)minuscule Hecke categories. 

\subsection{Realisations}
  
	Recall the notion of a realisation $(\{ \alphac_s \,|\, s \in S\} \subset \h, \{ \alpha_s \,|\, s \in S\} \subset \h^* )$ of $W$ as introduced in \cite{EW16} and subsequently refined in \cite{Haz23}. 
	\par 
	Define $R$ to be the graded ring $\Sym(\h^*)$ where $\h^*$ is in degree $2$. The Coxeter group $W$ acts on $R$ via the contragradient representation. The Demazure operators $\partial_s : R \rightarrow R$ are given by
	\begin{align*}
		\partial_s (f) := \frac{f-s(f)}{\alpha_s} ~~~\text{for all } f \in R. 
	\end{align*}
	Further, we consider $\Bbbk$ as an $R$-module via the isomorphism $\Bbbk \cong R/ R_+$ where $R_+$ now denotes the ideal generated by polynomials of strictly positive degree. 
	\par 
	The root datum of $G$ gives rise to a realisation over $\Z$, known as the Cartan realisation. 
	We always take the realisation $\h$ to be that obtained from the Cartan realisation by extending scalars to $\Bbbk$. 
	In particular, it can be of adjoint, simply-connected or universal type, in the sense of \cite[\S2.2.2]{Ric}.
	Cartan realisations of universal or simply-connected type satisfy the `parabolic property' of \cite[\S2.3]{LW22}, and Cartan realisations of universal or adjoint type always satisfy Demazure surjectivity. 

\subsection{The Elias-Williamson category}	
	
	We now give an extremely terse account of the Elias-Williamson diagrammatic categories. 
	The reader not already familiar with these categories should consult \cite{EMTW20}.
	\par 
	Let $\sD_{\BS}(\h)$ denote the diagrammatic \ldef{Bott-Samelson category} associated to the realisation $\h$, as introduced in \cite{EW16}. 
	It is a $\Bbbk$-linear, strict, monoidal category enriched over $R$-modules. 
	Objects $B_{\ux}$ in $\sD_{\BS}(\h)$ are indexed by (not necessarily reduced) expressions $\ux$. 
	The monoidal product is written as juxtaposition and defined by $B_{\ux} B_{\uy} = B_{\ux \uy}$ where $\ux\uy$ is the concatenation of $\ux$ and $\uy$. 
	Morphisms in this category are given by planar diagrams, known as Soergel diagrams. 
	These diagrams are generated by vertices of the form
	\[
		\settowidth\mylen{200000000000000000}
		\begin{array}{CCC|C}
			\begin{array}{c}
			\textit{univalent}
			\\
			\textit{vertex}
			\end{array}
		& 
			\begin{array}{c}
			\textit{trivalent}
			\\
			\textit{vertex}
			\end{array}
		&
			\begin{array}{c}
			2m_{st}\textit{-valent} 
			\\
			\textit{vertex} 
			\end{array}
		&
			\begin{array}{c}
			\textit{homogeneous}
			\\
			\textit{polynomial}
			\end{array}
		\\
			\tikz[scale=0.7]
			{
			\draw[color=blue] (3,-1) to (3,0);
			\node[circle,fill,draw,inner sep=0mm,minimum size=1mm,color=blue] at (3,0) {};
			\draw[dotted] (3,0) circle (1cm);
			}
		&
			\tikz[scale=0.7]
			{
			\draw[color=blue] (-30:1cm) -- (0,0) -- (90:1cm);
			\draw[color=blue] (-150:1cm) -- (0,0);
			\draw[dotted] (0,0) circle (1cm);
			}
		&
			\tikz[scale=0.7]
			{
			\draw[color=blue] (0,0) -- (22.5:1cm);
			\draw[color=red] (0,0) -- (0:1cm);
			\draw[color=blue] (0,0) -- (67.5:1cm);
			\draw[color=red] (0,0) -- (45:1cm);
			\draw[color=blue] (0,0) -- (112.5:1cm);
			\draw[color=red] (0,0) -- (90:1cm);
			\draw[color=blue] (0,0) -- (157.5:1cm);
			\draw[color=red] (0,0) -- (135:1cm);			
			\draw[color=blue] (0,0) -- (-22.5:1cm);
			\draw[color=red] (0,0) -- (180:1cm);			
			\draw[color=blue] (0,0) -- (-67.5:1cm);
			\draw[color=red] (0,0) -- (-45:1cm);			
			\draw[color=blue] (0,0) -- (-112.5:1cm);
			\draw[color=red] (0,0) -- (-90:1cm);			
			\draw[color=blue] (0,0) -- (-157.5:1cm);
			\draw[color=red] (0,0) -- (-135:1cm);
			\draw[dotted] (0,0) circle (1cm);
			}
		&
			\tikz[scale=0.7]{
			\draw[dotted] (0,0) circle (1cm);
			\node at (0,0) {$f$};
			}
		\\
			\deg 1
		&
			\deg -1
		&
			\deg 0
		&
			\deg f
		\end{array}
		\]
	and subject to various local relations.
	We will only recall the relations we require; the rest may be found in \cite{EW16}.
	We require both polynomial relations
		\[
		\settowidth\mylen{2000000000000000000000000000000000}
		\begin{array}{CC}
			\textit{Barbell relation}
			&
			\textit{Polynomial forcing}
		\\
			\begin{array}{c}
			\tikz[scale=0.7]{
			\draw[dotted] (0,0) circle (1cm);
			\draw[color=blue] (0,0.5) -- (0,-0.5);
			\node[draw, fill, circle, inner sep=0mm, minimum size=1mm, color=blue] at (0,0.5) {};
			\node[draw, fill, circle, inner sep=0mm, minimum size=1mm, color=blue] at (0,-0.5) {};
			} 
			\end{array}
			=
			\begin{array}{c}
			\tikz[scale=0.7]{
			\draw[dotted] (0,0) circle (1cm);
			\node at (0,0) {$\alpha_{{\color{blue}s}}$};
			} 
			\end{array}
		&
			\begin{array}{c}
			\tikz[scale=0.7]{
			\draw[dotted] (0,0) circle (1cm);
			\draw[color=blue] (0,-1) to (0,1);
			\node at (-0.5,0) {$f$};
			} \end{array}
			=
			\begin{array}{c}
			\tikz[scale=0.7]{
			\draw[dotted] (0,0) circle (1cm);
			\draw[color=blue] (0,-1) to (0,1);
			\node at (0.5,0) {${\color{blue}s}f$};
			} 
			\end{array}
			+
			\begin{array}{c}
			\tikz[scale=0.7]{
			\draw[dotted] (0,0) circle (1cm);
			\draw[color=blue] (0,-1) to (0,-0.5);
			\draw[color=blue] (0,1) to (0,0.5);
			\node[color=blue, draw, fill, circle, inner sep=0mm, minimum size=1mm] at (0,-0.5) {};
			\node[color=blue, draw, fill, circle, inner sep=0mm, minimum size=1mm] at
			(0,0.5) {};
			\node at (0,0) {$\partial_{\color{blue}s} f$};
			} \end{array}
		\end{array}
		\]
	the one-colour relations
		\[
		\settowidth\mylen{20000000000000000000000000}
		\begin{array}{CCC}
			\textit{Frobenius unit}
		&
			\textit{Frobenius associativity}
		&
			\textit{Needle relation}
		\\
			\begin{array}{c}
    		\tikz[scale=0.7]{\draw[dotted] (0,0) circle (1cm);
			\draw[color=blue] (-90:1cm) -- (0,0) -- (90:1cm);
			\draw[color=blue] (-0:0.5cm) -- (0,0);
			\node[circle, fill, draw, inner sep=0mm, minimum size=1mm, color=blue] at (-0:0.5cm) {};
			}
  			\end{array}
			=
  			\begin{array}{c}
    		\tikz[scale=0.7]{\draw[dotted] (0,0) circle (1cm);
			\draw[color=blue] (-90:1cm) -- (0,0) -- (90:1cm);}
  			\end{array}
  		&
 			\begin{array}{c}
    		\tikz[scale=0.7]{\draw[dotted] (0,0) circle (1cm);
			\draw[color=blue] (-45:1cm) -- (0.3,0) -- (45:1cm);
			\draw[color=blue] (135:1cm) -- (-0.3,0) -- (-135:1cm);
			\draw[color=blue] (-0.3,0) -- (0.3,0);
			}
			\end{array}
			=
			\begin{array}{c}
			\tikz[scale=0.7, rotate=90]{\draw[dotted] (0,0) circle (1cm);
			\draw[color=blue] (-45:1cm) -- (0.3,0) -- (45:1cm);
			\draw[color=blue] (135:1cm) -- (-0.3,0) -- (-135:1cm);
			\draw[color=blue] (-0.3,0) -- (0.3,0);
			}
			\end{array}
		&
			\begin{array}{c}
			\begin{tikzpicture}[scale=0.7]
			\draw[dotted] (0,0) circle (1cm);
			\draw[blue] (0,0) circle (0.6cm);
			\draw[blue] (0,-0.6) --(0,-1);
			\draw[blue] (0,0.6) --(0,1);
    		\end{tikzpicture}
 			\end{array}
 			= 
 			0
		\end{array}
		\]
	the two-colour relations for $m_{st}=2$, and the $\tA_1 \times \tA_1 \times \tA_1$ Zamolodchikov relation
		\[
		\settowidth\mylen{20000000000000000000000000}
		\begin{array}{CCC}
			\textit{Two-colour associativity}
		&
			\textit{Jones-Wenzl relation}
		&
			\textit{Zamolodchikov}
		\\
			\begin{array}{c}
    		\begin{tikzpicture}[scale=0.7]
      		\draw[dotted] (0,0) circle (1cm);
			\draw[red] (-45:1) -- (135:1);
			\draw[blue] (-135:1) -- (45:0.3);
			\draw[blue] (45:0.3) -- (20:1);
			\draw[blue] (45:0.3) -- (80:1); 
    		\end{tikzpicture}
  			\end{array}
			=
  			\begin{array}{c}
    		\begin{tikzpicture}[scale=0.7]
      		\draw[dotted] (0,0) circle (1cm);
			\draw[red] (-45:1) -- (135:1);
			\draw[blue] (-135:1) -- (-135:0.4);
			\draw[blue] (-135:0.4) -- (20:1);
			\draw[blue] (-135:0.4) -- (80:1); 
    		\end{tikzpicture}
  			\end{array}
		&
			\begin{array}{c}
    		\begin{tikzpicture}[scale=0.7]
      		\draw[dotted] (0,0) circle (1cm);
			\draw[red] (-45:1) -- (135:1);
			\draw[blue] (45:1) -- (-135:0.6);
			\node[circle, fill, draw, inner sep=0mm, minimum size=1mm, color=blue] at (-135:0.6) {}; 
    		\end{tikzpicture}
  			\end{array}
			=
			\begin{array}{c}
    		\begin{tikzpicture}[scale=0.7]
      		\draw[dotted] (0,0) circle (1cm);
			\draw[red] (-45:1) -- (135:1);
			\draw[blue] (45:1) -- (-135:-0.4);
			\node[circle, fill, draw, inner sep=0mm, minimum size=1mm, color=blue] at (-135:-0.4) {}; 
    		\end{tikzpicture}
  			\end{array}
  		&
  			\begin{array}{c}
    		\begin{tikzpicture}[scale=0.7]
      		\draw[dotted] (0,0) circle (1cm);
			\draw[red] (-45:1) -- (135:1);
			\draw[blue] (45:1) -- (-135:1); 
			\draw[color=green] (1,0) arc [start angle=0, end angle = 180, x radius = 1cm , y radius = 0.3cm]; 
    		\end{tikzpicture}
  			\end{array}
			=
			\begin{array}{c}
    		\begin{tikzpicture}[scale=0.7]
      		\draw[dotted] (0,0) circle (1cm);
			\draw[red] (-45:1) -- (135:1);
			\draw[blue] (45:1) -- (-135:1); 
			\draw[color=green] (1,0) arc [start angle=0, end angle = -180, x radius = 1cm , y radius = 0.3cm];  
    		\end{tikzpicture}
  			\end{array}.
		\end{array}
		\]
	Each relation is homogeneous with respect to the degree of the Soergel graph, so $\Hom_{\sD_{\BS}(\h)}(B_{\ux}, B_{\uy})$ is a $\Z$-graded $\Bbbk$-module. 
	Following \cite{ARV20}, there is a category associated to $\sD_{\BS}(\h)$ which has: objects $B_{\ux}(n)$, where $\ux$ is an expression and $n \in \Z$; morphisms from $B_{\ux}(n)$ to $B_{\uy}(m)$ are degree $m-n$ morphisms in $\Hom_{\sD_{\BS}(\h)}(B_{\ux}, B_{\uy})$; and, a shift functor $(1)$. 
	We will interchangeably identify these categories. 
	\par 
	The \ldef{(bi-equivariant) Elias-Williamson diagrammatic category} $\DBE(\h)$ is the graded, additive, Karoubian completion of $\sD_{\BS}(\h)$. 
	If we wish to emphasise the underlying ring $\Bbbk$, we write $\sD_{\text{BE}}(\h,\Bbbk)$.
	It inherits a monoidal structure from $\sD_{\BS}(\h)$, which is again denoted by juxtaposition. 
	When $\Bbbk$ is a complete local ring $\DBE(\h)$ is Krull-Schmidt, and the indecomposable objects are parameterised by $B_{x}(n)$ where $(x, n) \in W \times \Z$. 
	Moreover, each $B_x$ is uniquely characterised by the properties:
		
	\begin{tabular}{rl}
		(a) & if $\ux$ is a reduced expression of $x$, $B_x$ is a multiplicity-one direct summand of $B_{\ux}$;
		\\
		(b) & if $\ell(\uy) < \ell(\ux)$, $B_x$ is not a summand of $B_{\uy}$;
	\end{tabular}
		
	\noindent
	We define the graded Hom-space $\Hom^{\bullet}(B, B')$ as $\bigoplus_{n \in \Z} \Hom(B, B'(n))$.
	It is a free, finitely generated $R$-module. 
	Moreover, since we have taken $\h$ to be the Cartan realisation of $G$, we have an equivalence of monoidal categories
	\begin{align*}
		\mathcal{D} :  \Parity(B \backslash G / B) ~~\longrightarrow~~ \DBE(\h)
	\end{align*}
	satisfying $\mathcal{D}(\cE_x) \cong B_x$ for all $x \in W$, see \cite[\S10.3]{RW18}.

\subsection{The diagrammatic antispherical category}

	We now recall the diagrammatic categorification of the antispherical module introduced in \cite[\S 4]{RW18} and generalised in \cite{LW22}. 
	We proceed in two steps.  
	\par 
	First, consider the \ldef{(right-equivariant) Elias-Williamson diagrammatic category} $\DRE(\h)$, which is defined as the category with the same objects as $\DBE(\h)$, and (graded) Hom-spaces are
	\begin{align*}
		\Hom^{\bullet}_{\DRE(\h)} (B,B') := \Bbbk \otimes_{R} \Hom^{\bullet}_{\DBE(\h)}(B,B').
	\end{align*} 
	This category is additive, $\Bbbk$-linear and has shift functor $(1)$ induced by the shift functor on $\DBE(\h)$. 
	It is not monoidal, but is a module category over $\DBE(\h)$ with the obvious monoidal action. 
	If $\Bbbk$ is a complete local ring $\DRE(\h)$ is Krull-Schmit, with indecomposable objects $B_x(n)$ parameterised by $(x,n) \in W \times \Z$. 
	The forgetful functor $\For : \DBE(\h) \rightarrow \DRE(\h)$ is essentially surjective and satisfies $\For(B_x) \cong B_x$, hence our abuse of notation. 
	Since $\h$ is the Cartan realisation of $G$, we have an equivalence
	\begin{align*}
		\mathcal{D} : \Parity(U \backslash G / B ) ~~\longrightarrow~~ \DRE(\h)
	\end{align*}
	satisfying $\mathcal{D}(\cE_x) \cong B_x$ for each $x \in W$; this follows from combining the arguments in \cite[\S4.3]{AR16a}, \cite[\S1.7]{Ric19} and \cite[\S6.7]{EW16}.
	\par 
	The construction of the diagrammatic antispherical Hecke category is completely analogous to the quotient construction in Section \ref{Ssec: Geo Quotient Construction}. 
	We define the \ldef{diagrammatic antispherical category} $\DRA(\h)$ as the category 
	\begin{align*}
		\DRA(\h,I) := \DRE(\h) ~/~ \langle B_x  ~|~ x \notin {}^I W \rangle_{\oplus, (1)}
	\end{align*}
	where $\langle B_x  ~|~ x \notin {}^I W \rangle_{\oplus, (1)}$ denotes the full subcategory generated by the $B_x$ under direct sums and shifts.  
	The category $\DRA(\h,I)$ is graded, additive and a module category over $\DBE(\h)$. 
	When $\Bbbk$ is a complete local ring, $\DRA(\h,I)$ is Krull-Schmit with indecomposable objects $B_{x}(n)$ parameterised by $(x, n) \in {}^I W \times \Z$. 
	Moreover, the quotient functor $q : \DRE(\h) \rightarrow \DRA(\h,I)$  is essentially surjective and satisfies $q(B_{x}) \cong B_x$, hence the abuse of notation. 
	Since $\h$ is the Cartan realisation of $G$, the functors $\mathcal{D}$ and $Av^{-1}$ induce an equivalence 
	\begin{align*}
		\Parity_{Wh, I}( G/ B) ~~ \longrightarrow ~~ \DRA(\h,I)
	\end{align*}
	satisfying $\mathcal{D} \circ Av^{-1} (\cE_x) \cong B_x$ for all $x \in {}^I W$.
	\par
	We conclude by recalling that antispherical light-leaves were introduced in \cite[\S5]{LW22}. 
	These are morphisms in $\DRA(\h,I)$ which are inductively constructed from the data of an $I$-parabolic Bruhat stroll on a subexpression $\ux^e \subseteq \ux$.  
	We assume the reader is familiar with this construction.

\section{Diagrammatic antispherical (co)minuscule Hecke categories}
\label{Sec: Dia Cominuscule}

	We now determine the $p$-Kazhdan-Lusztig bases of antispherical (co)minuscule Hecke categories using diagrammatic methods. 
	These results hold in all characteristics, and facilitate greater insight into the structure of indecomposable objects in $\Parity_{Wh,I}(G/B)$.  
	\par  
	The vast majority of this Section is devoted the antispherical Hecke category of the Lagrangian Grassmannian.  
	In Section \ref{Ssec: odd quadrics} we consider odd dimensional quadric hypersurfaces.
	The Langlands dual minuscule cases are treated simultaneously. 
	\par 
	When we consider the diagrammatic antispherical Hecke category of the Lagrangian Grassmannian (resp. odd dimensional quadric hypersurfaces), $\h$ will be a Cartan realisation of type $\tB_n$ (resp. $\tC_n$) because Koszul duality requires we consider the Langlands dual realisation.  
	
\subsection{Minimal coset representatives}
\label{Ssec: min coset reps LG}
	
	In Sections \ref{Ssec: min coset reps LG}-\ref{Ssec: Endomorphism algebras}, the Coxeter system $(W,S)$ will always be type $\tB_n$ and $I \subset S$ will be type $\tA_{n-1}$.
	\par 
	We begin by establishing notation for ${}^I W$.
	Standard references are \cite{BM09} or \cite{BB08}. 
	Our simple reflections are labelled as follows: 
	\[
	\begin{tikzpicture}
		\node[circle,fill,draw,inner sep=0mm,minimum size=2mm,color=black, thick] at (-1.5,0) {}; 
		\node[circle,fill,draw,inner sep=0mm,minimum size=2mm,color=black, thick] at (-1,0) {};
		\node[circle,fill,draw,inner sep=0mm,minimum size=2mm,color=black, thick] at (-0.5,0) {};
		\node[circle,fill,draw,inner sep=0mm,minimum size=2mm,color=black, thick] at (0.5,0) {}; 
		\node[circle,fill,draw,inner sep=0mm,minimum size=2mm,color=black, thick] at (1,0) {};
		\node[circle,fill,draw,inner sep=0mm,minimum size=2mm,color=black, thick] at (1.5,0) {};  
		\draw (-1.5,0.05) -- (-1,0.05);
		\draw (-1.5,-0.05) -- (-1,-0.05);
		\draw (-0.5,0) -- (-1,0);
		\draw (0.5,0) -- (1,0);
		\draw (1,0) -- (1.5,0);
		\filldraw[black] (-1.5,-0.5) node[font = \small]{${}_0$}; 
		\filldraw[black] (-1,-0.5) node[font = \small]{${}_1$};
		\filldraw[black] (-0.5,-0.5) node[font = \small]{${}_2$}; 
		\filldraw[black] (1.5,-0.5) node[font = \small]{${}_{n-1}$};	
		\filldraw[black] (0,0) node[font = \small]{\dots};		 
		\end{tikzpicture}
		\]
		This ensures compatibility with the obvious inclusion $W_{\tB_n} \hookrightarrow W_{\tB_{n+1}}$.
		Recall $W$ can be realised as the signed permutations of $[\pm n] := \lbrace \pm 1, \dots , \pm n \rbrace$, and $W_I$ as the sign-symmetric permutations. 
		In terms of transpositions, this is
		\begin{align*}
			s_0 = (-1, 1) 
		&&
			\text{ and }
		&&
			s_i = (i, i+1)(-i-1, -i) 
			\text{ when } 0<i<n.
		\end{align*}
		Each minimal coset representative ${}^I W$ is fully commutative \cite[\S 6]{Ste96}. 
		Many combinatorial objects index ${}^I W$, we briefly recall how they can be identified. 
		\par 
		\textit{Subsets of $\{1, \dots ,n \}$}:
		Let $[n] := \{1, 2, \dots , n \}$ and $\cP([n])$ be the power-set of $[n]$.
		The group $W$ acts on $\cP([n])$ via the action
		\begin{align}
		\label{eqn: action on sets}
			x s_i 
		=
		 	\begin{cases}
		 		(x \backslash \lbrace i\rbrace ) \cup \lbrace i+1 \rbrace
		 	&
		 		\text{if }
		 		i \in x
		 		\text{ and }
			 		i+1 \notin x,
	 		\\
	 			x
	 		&
	 			\text{if }
	 			i, i+1 \in x 
	 			\text{ or }
	 			i, i+1 \notin x,
		 		\text{ and}
		 	\\
		 		(x \backslash \lbrace i+1 \rbrace) \cup \lbrace i \rbrace
		 	&
		 		\text{if }
		 		i+1 \in x
			 		\text{ and }
	 			i \notin x. 
		 	\end{cases}
		\end{align}
		where we identify $\{ 0 \}$ with $\emptyset$. 
		Clearly $W_I$ stabilises $\emptyset$, and it is easily checked if $j \in x \subseteq [n]$ is minimal and $i<j$, then $x(s_0 s_1 \dots s_{i-1}) = x \cup \{ i \}$.  
		Hence $W$ acts transitively, and we can identify  ${}^I W$ with $\cP([n])$. 
		Under this identification $\id \in {}^I W$ corresponds to $\emptyset$ and $w_I w_0 \in {}^I W$ corresponds to $[n]$.   
		\par 
		\textit{Strictly decreasing sequences}:
		Given $\{ t_1, t_2 , \dots , t_k \} \subseteq [n]$ we impose an order on the elements by requiring $t_1 > t_2 > \dots > t_k$. Then the sequence $(t_1 , t_s , \dots , t_k )$ is strictly decreasing. 
		\par 
		\textit{Partitions}: 
		Given a strictly decreasing sequence $(t_1, t_2 , \dots , t_k)$, we can associate the (left justified) partition with rows of lengths $t_1, t_2 , \dots , t_k$ from top-to-bottom. 
		\par 
		Partitions are endowed with a partial order called the \ldef{containment order}:
		a partition $x = (t_1 , \dots , t_k)$ contains $y = (t_1' , \dots , t_l')$ if $k\geq l$ and $t_i \geq t_i'$ for each $1 \leq i \leq l$.
		A partition $\lambda = (t_1 , \dots , t_k)$ is called \ldef{strict} if $t_1 > \dots > t_k$.   
		Figure \ref{fig:Bruhat graph} depicts the containment order on strict partitions with longest part $5$. 
		The following fact is shown in \cite[Lemma 44]{BM09}. 
		\begin{lem}
			The containment order on strict partitions with longest part $n$ is isomorphic, as a poset, to the Bruhat order on ${}^I W$. 
		\end{lem}
		  
		\begin{figure}[p]
			\begin{center}
			\begin{tikzpicture}[scale=.6]
 				\node (54321) at (18,15) {$\{5,4,3,2,1\}$};
 				
 				\node (5432) at (18,13) {$\{5,4,3,2\}$};
 				
 				\node (5431) at (18,11) {$\{5,4,3,1\}$};
 				
 				\node (5421) at (15,9) {$\{5,4,2,1\}$};
 				\node (543) at (21,9) {$\{5,4,3\}$};
 				
 				\node (5321) at (15,7) {$\{5,3,2,1\}$};
 				\node (542) at (21,7) {$\{5,4,2\}$};
 				
 				\node (4321) at (12,5) {$\{4,3,2,1\}$};
 				\node (532) at (18,5) {$\{5,3,2\}$};
 				\node (541) at (24,5) {$\{5,4,1\}$};
 				
 				\node (432) at (12,3) {$\{4,3,2\}$};
 				\node (531) at (18,3) {$\{5,3,1\}$};
 				\node (54) at (24,3) {$\{5,4\}$};
 				
 				\node (431) at (12,1) {$\{4,3,1\}$};
 				\node (521) at (18,1) {$\{5,2,1\}$};
 				\node (53) at (24,1) {$\{5,3\}$};

 				\node (421) at (12,-1) {$\{4,2,1\}$};
 				\node (43) at (18,-1) {$\{4,3\}$};
 				\node (52) at (24,-1) {$\{5,2\}$};
 				
 				\node (321) at (12,-3) {$\{3,2,1\}$};
 				\node (42) at (18,-3) {$\{4,2\}$};
 				\node (51) at (24,-3) {$\{5,1\}$};
 				
 				\node (32) at (12,-5) {$\{3,2\}$};
 				\node (41) at (18,-5) {$\{4,1\}$};
 				\node (5) at (24,-5) {$\{5\}$};
 				
 				\node (31) at (15,-7) {$\{3,1\}$};
 				\node (4) at (21,-7) {$\{4\}$};
 				
 				\node (21) at (15,-9) {$\{2,1\}$};
 				\node (3) at (21,-9) {$\{3\}$};
 				
 				\node (2) at (18,-11) {$\{2\}$};
 				
 				\node (1) at (18,-13) {$\{1\}$};
 				
 				\node (0) at (18,-15) {$\emptyset$};
				\draw (0) -- (1);
				
				\draw (1) -- (2);
				
				\draw (2) -- (21);
				\draw (2) -- (3);
			
				\draw (21) -- (31);
				\draw (3) -- (31);
				\draw (3) -- (4);
				
				\draw (31) -- (32);
				\draw (31) -- (41);
				\draw (4) -- (41);
				\draw (4) -- (5);
				
				\draw (32) -- (321);
				\draw (32) -- (42);
				\draw (41) -- (42);
				\draw (41) -- (51);
				\draw (5) -- (51);
				
				\draw (321) -- (421);
				\draw (42) -- (421);
				\draw (42) -- (43);
				\draw (42) -- (52);
				\draw (51) -- (52);
				
				\draw (421) -- (431);
				\draw (421) -- (521);
				\draw (43) -- (431);
				\draw (43) -- (53);
				\draw (52) -- (521);
				\draw (52) -- (53);
				
				\draw (431) -- (432);
				\draw (431) -- (531);
				\draw (521) -- (531);
				\draw (53) -- (531);
				\draw (53) -- (54);
				
				\draw (432) -- (4321);
				\draw (432) -- (532);
				\draw (531) -- (532);
				\draw (531) -- (541);
				\draw (54) -- (541);
				
				\draw (4321) -- (5321);
				\draw (532) -- (5321);
				\draw (532) -- (542);
				\draw (541) -- (542);
				
				\draw (5321) -- (5421);
				\draw (542) -- (5421);
				\draw (542) -- (543);
				
				\draw (5421) -- (5431);
				\draw (543) -- (5431);

				\draw (5431) -- (5432);
				
				\draw (5432) -- (54321);
			\end{tikzpicture}
			\end{center}
			\caption{The Bruhat poset in types $\tA_{4} \backslash\tB_5$}
			\label{fig:Bruhat graph}
			\end{figure}	
		 
		For $x \in {}^I W$ we consider the associated partition.
		Filling each box with the number of boxes to the left in that row\footnote{
		This is sometimes referred to as the right-arm length of the box.}
		defines a tableau. 
		Reading the entries of the tableau left-to-right, top-to-bottom gives a reduced expression for $x \in {}^I W$, which we call the \ldef{row reduced expression}.
		This is a special case of \cite[Theorem 8]{BM09}. 
		\par 
		Given $x \in {}^I W$ and the tableau constructed above, we construct a \ldef{shifted tableau} by shifting the $i^{\text{th}}$ row $2(i-1)$ boxes to the right. 
		The \ldef{column reduced expression} is obtained by reading the entries of the tableau top-to-bottom, left-to-right. 
		\par 
		One sees the column reduced expression is reduced, as it can be obtained from the row reduced expression by repeatedly applying the commutative braid relation $st=ts$ to each $s_i$ until $s_i$ is either: 
		on the far left (in which case $i=0$);
		or adjacent to either $s_0$, $s_1$ or $s_{i+2}$. 
		
		\begin{exmp}
		The subset $\{6,4,2,1\}$ has tableau and shifted tableau
		\[
		\begin{array}{c}
		\tableau{
			0 	& 	1 	& 	2 	& 	3 	& 	4 	& 	5 	 \\
			0 	& 	1 	& 	2 	& 	3 	& 	 	& 	   \\
			0 	& 	1 	& 	 	& 	 	& 	 	& 	  \\
			0 	& 	 	& 	 	& 	 	& 	 	& 	   
		}
		\end{array}
		~~~~~~~~~~~~~~
		\begin{array}{c}
		\tableau{
			0&1&2&3&4&5 	 \\
			 & &0&1&2&3  \\
			 & & & &0&1  \\
			 & & & & & &	0   
		}
		\end{array}
		\]
		respectively. The row reduced expression is $\ux = s_0 s_1 s_2 s_3 s_4 s_5 ~ s_0 s_1 s_2 s_3 ~ s_0 s_1  ~ s_0$, while the column reduced expression is $\ux = s_0 ~ s_1 ~ s_2 s_0 ~ s_3 s_1 ~ s_4 s_2 s_0 ~ s_5 s_3 s_1 ~ s_0$.
		\end{exmp} 
		
		Our rationale for considering these reduced expressions is as follows:
		the column reduced expression is well-suited for studying subexpressions, whereas the row reduced expression is better-suited for inductive arguments on the length of elements.  	
		\par 
		We freely and implicitly identify elements of $ {}^I W$ with subsets of  $[n]$, and the corresponding partitions and tableaux. 
		Given $x = \{ t_1  > \dots > t_k \}$ we can consider its:
		\begin{itemize}
			\item \textit{Length}: $\ell(x) = \sum _{i} t_i$, which is the usual length function on ${}^I W$, and
			\item \textit{Cardinality}: $|x| = k$, which is the cardinality of the subset corresponding to $x$. 
		\end{itemize}
		Finally, note the row (resp. column) reduced expression $\ux$ arises as the concatenation of the expressions for each row (resp. column) in the tableau (resp. shifted tableau).  

\subsection{Defect-0 subexpressions}
\label{Ssec: Defect 0}
	
		We now identify a set of defect-0 subexpressions of $\ux$ for the row and column reduced expression of $x$. 
		In Section \ref{Ssec: bs basis} it will be shown these are \textit{all} defect-0 subexpressions of $\ux$. 
		\par 
		Fix $x \subseteq [n]$ and define $E(x) \subset x$ as
		\begin{align*}
			E(x)
		:=
			\lbrace
				t \in x 
			~\vert~	
				t 
				\text{ is even and }
				| \lbrace t' \in x ~\vert~ t > t' \geq k \rbrace |
			>
				| \lbrace t' \notin x ~\vert~ t > t' \geq k \rbrace |
				\text{ for all }
				t > k \geq 1
			\rbrace.
		\end{align*}
		Ordering the elements of $x = \{ t_1 > \dots > t_k \}$ shows
		\begin{align}
		\label{eqn: row length condition}
			E(x)
		=
			\lbrace
				t_i \in x 
			~\vert~	
				t_i 
				\text{ is even and }
				t_i - t_{i+j}
				\leq 
				2j-1
				\text{ for all }
				1 \leq  j \leq t_i /2
			\rbrace.
		\end{align}
		\noindent 
		If $t_i \in E(x)$, the non-shifted teableau for $x$ must have the rightmost box of row $i$ marked $t_i -1$, row $i+j$ must contain boxes marked $t_i - 2j-1$ and $t_i - 2j$ when $0<j<t_i/2$, and row $t_{i + t_i/2}$ must contain a box marked $0$. 
		We call the skew-partition formed by such boxes a  \ldef{length-$t_i$ 1-2-1-tableau} $\underline{121}_{t_i}$. 
		Note $\underline{121}_{t_i}$ only depends on $t_i$, not $i$; $i$ determines how $\underline{121}_{t_i}$ is contained inside $x$.   
		For example, $\underline{121}_6$ is 
		\[
		\tableau{
			&&&&&5\\
			&&&3&4\\
			&1&2\\
			0
		}
		\]
		Since the partition $x$ is strict, there is no box to the right of $t_i -1$ in row $i$, or $t_i-2$ in row $i+1$. 
		\par 
		Recall the shifted tableau for $x$ is obtained from the non-shifted tableau by shifting row $i$ by $2(i-1)$ boxes to the right. 
		After shifting, the skew-tableau $\underline{121}_t$ becomes a two-column, skew tableau $\uu_t$, where the left column consists of boxes marked with $t-1$, $t-3$, $\dots$ $1$ read top-to-bottom, while the right column consists of boxes marked with $t-2$, $t-4$, $\dots$, $0$ read top-to-bottom, and the right column is shifted-down by a box.
 		For example $\uu_6$ is 
		\[
		\tableau{
			5\\
			3&4\\
			1&2\\
			&0
		}
		\]
		Reading the entries of $u_t$ top-to-bottom, left-to-right defines an expression 
		\begin{align*}
			\uu_t 
			= 
			s_{t -1} s_{t -3}  \dots s_3 s_1 ~~
			s_{t -2} s_{t -4}  \dots s_2 s_0 
		\end{align*} 
		which is reduced as each simple reflection appears at most once. 
		We write $u_t$ for $\uu_{t \bullet}$. 
		\par 
		Examples are provided for the sanity of the reader. 
		
		\begin{exmp}
			Let $x =\{6,5,4,3,2,1\}$, then $E(x) = \{6,4,2 \}$. 
			The tableaux corresponding to $x$ are 
				\[
				\begin{array}{c}
					\tableau{
					0&1&2&3&4&\green{5}\\
					0&1&2&\green{3} &\green{4}&\\
					0&\green{1}   &\green{2}&\red{3}	\\
					\green{0}&\red{1}	&\red{2}&\\
					\red{0}	&\blue{1}\\
					\blue{0}
				}
				\end{array}
				\longleftrightarrow
				\begin{array}{c}
					\tableau{
					0&1&2&3&4&\green{5}\\
					&&0&1&2&\green{3} &\green{4}&\\
					&& & &0&\green{1}   &\green{2}&\red{3}	\\
					&& & & & 			&\green{0}&\red{1}	&\red{2}&\\
					&& & & & 			& 		  & 			&\red{0}	&\blue{1}\\
					&& & & & 			& 		  & 			& 		& 		&\blue{0}
				}
				\end{array}
				\]
			where we have marked $u_6$ in \green{green}, $u_4$ in \red{red}, and $u_2$ in \blue{blue}.
			Their reduced expressions are $\uu_6 = s_5 s_3 s_1 ~ s_4 s_2 s_0$, $\uu_4 = s_3 s_1 ~s_2 s_0$, and $\uu_2 = s_1~ s_0$ respectively. 
		\end{exmp}
		
		The next example demonstrates  boxes can appear above the left column of $\uu_i$. 
		\begin{exmp}
			Let $x =\{6,4,2,1\}$, then $E(x) = \{2 \}$.
			The tableaux corresponding to $x$ are 
			\[
			\begin{array}{c}
			\tableau{
				0&1&2&3&4&5 	 		\\
				0&1&2&3  		\\
				0&\red{1}	\\
				\red{0}   
			}
			\end{array}
			\longleftrightarrow
			\begin{array}{c}
			\tableau{
				0&1&2&3&4&5 	 		\\
				 & &0&1&2&3  		\\
				 & & & &0&\red{1}	\\
				 & & & & & &\red{0}   
			}
			\end{array}
			\]
			with $u_2$ marked in \red{red}.
			The reduced expression is $\uu_2 = s_1 ~s_0$.
		\end{exmp}
		
		The final example demonstrates that $x$ can have boxes to the right of some $u_i$ which are not part of any $u_j$.  
		
		\begin{exmp}
		\label{exmp: tableau 121 removal}
			Let $x =\{10,9,8,7,3,1\}$, then $E(x) = \{10\}$. The tableaux corresponding to $x$ are 
			\[
			\begin{array}{c}
			\tableau{
				 0&1&2&3&4&5&6&7&8&\red{9}\\
				 0&1&2&3&4&5&6&\red{7}&\red{8} \\
				 0&1&2&3&4&\red{5}&\red{6}&7\\
				 0&1&2&\red{3}&\red{4}&5&6	\\
				 0&\red{1}&\red{2}\\
				 \red{0}
			}
			\end{array}
			\longleftrightarrow
			\begin{array}{c}
			\tableau{
				 0&1&2&3&4&5&6&7&8&\red{9}\\
				  & &0&1&2&3&4&5&6&\red{7}&\red{8} \\
				  & & & &0&1&2&3&4&\red{5}&\red{6}&7\\
				  & & & & & &0&1&2&\red{3}&\red{4}&5&6	\\
				  & & & & & & & &0&\red{1}&\red{2}\\
				  & & & & & & & & & 		  &\red{0}
			}
			\end{array}
			\]
			with $u_{10}$ marked in {\color{red}red}. 
			The reduced expression is $\uu_{10} = s_9 s_7 s_5 s_3 s_1 ~s_8 s_6 s_4 s_2 s_0$. 
		\end{exmp}
		
		Now fix $x = \{ t_1 > \dots > t_l \}$ and $E = \{ t_{i_1} > \dots > t_{i_k}  \}  \subseteq E(x)$. 
		The subset $E$ determines a sequence of disjoint skew-partitions $u_{t_i}$ in the shifted partition of $x$, each with a reduced expression $\uu_{t_i}$.
		This allows us to write the column reduced expression $\ux$ of $x$ as the concatenation of a sequence of expressions
		\begin{align*}
			\ux
			=
			\ur_{1} \uu_{t_{i_1}} \ur_{2} \uu_{t_{i_2}} 
			\dots 
			\ur_{k} \uu_{t_{i_k}} \ur_{k+1}
		\end{align*}
		where $\ur_1$ is the expression formed by taking all terms of $\ux$ prior to the commencement of $\uu_{t_{i_1}}$, $\ur_2$ is the expression formed by taking all terms of $\ux$ proceeding $\uu_{t_{i_1}}$ and preceding $\uu_{t_{i_2}}$, and so on. 
		Note that each $r_i$ depends on both $x$ and $E$.
		
		\begin{lem}
		\label{lem: subexpression 121 removal}
			Fix $x \in {}^I W$ and $E= \{ t_i \} \subseteq E(x)$. 
			Then $\ur_1 \ur_2$ is a reduced expression for $x \backslash \{ t_i \}$. 
		\end{lem}

		The following example is more enlightening than the proof. 
		
		\begin{exmp}
			Let $x =\{10,9,8,7,3,1\}$, as in Example \ref{exmp: tableau 121 removal}, and $E= \{10\}$. 
			The expression $\ur_1 \ur_2$ corresponds to deleting the skew-partition $u_{10}$ from $x$, then shifting the tableau $r_2$ up until it joins $r_1$.  
			\[
			\begin{array}{c}
			\tableau{
				 0&1&2&3&4&5&6&7&8&\red{9}\\
				 0&1&2&3&4&5&6&\red{7}&\red{8} \\
				 0&1&2&3&4&\red{5}&\red{6}&7\\
				 0&1&2&\red{3}&\red{4}&5&6	\\
				 0&\red{1}&\red{2}\\
				 \red{0}
			}
			\end{array}
			\longrightarrow
			\begin{array}{c}
			\tableau{
				 0&1&2&3&4&5&6&7&8&\\
				 0&1&2&3&4&5&6& & \\
				 0&1&2&3&4& & &7\\
				 0&1&2& & &5&6	\\
				 0\\
			}
			\end{array}
			\longrightarrow
			\begin{array}{c}
			\tableau{
				 0&1&2&3&4&5&6&7&8\\
				 0&1&2&3&4&5&6&7 \\
				 0&1&2&3&4&5&6\\
				 0&1&2&	\\
				 0&\\
			}
			\end{array}
			\]
		\end{exmp}
 
 		Before we prove the Lemma, we first introduce some notation. 
 		For non-negative integers $a\leq b$ define
 			\begin{align*}
 				\uw_a 
 			= 
 				s_0 s_1 \dots s_{a-1},
 			&&
 				\uw_{a,b}
 			= 
 				s_a s_{a+1} \dots  s_{b-2} s_{b-1}.
 			\end{align*}
 		Write $w_a$, $w_{a,b}$ for the corresponding elements of $W$. 
 		Note $\uw_a$ is the expression obtained by reading the terms in a length-$a$ row of any tableau.

 		\begin{proof}
 			Fix $x = \{ t_1 > \dots > t_k \}$.
 			Recall $t_i$ is even; set $j = t_i/2$.
  			By definition of the 1-2-1 tableau we have  
 			\begin{align*}
 				\ur_1 
 			&= 
 				\uw_{t_1} \dots \uw_{t_{i-1}} 
 				\uw_{t_i -1} \uw_{t_i -3} \dots \uw_3 \uw_1,
 				\\
 				\ur_2
			&=
				\uw_{t_{i} - 1, t_{i + 1}} \uw_{t_i - 3, t_{i + 2}}\dots 
				\uw_{1, t_{i + j}}
				\uw_{t_{ i + j  +1} } 
				\uw_{t_{i + j  +2} } \dots 
				\uw_{t_k}. 
 			\end{align*} 
 			Note $\uw_a \uw_{a,b} = \uw_b$ and $w_{a,b}$ commutes with $w_{c}$ for any $c < a$.  
 			Consequently 
 			\begin{align*}
 				(\ur_1\ur_2)_{\bullet} 
 			&= 
 				w_{t_1} \dots w_{t_{i-1}}
 				(w_{t_i -1} w_{t_i - 1 , t_{i+1}} )\dots (w_1 w_{1, t_{i + j}})
 				w_{t_{ i + j  +1} } 
				\dots 
				w_{t_k}
			\\
			&=
				w_{t_1} \dots w_{t_{i-1}}
 				w_{t_{i+1}} \dots w_{ t_{i + j}}
 				w_{t_{ i + j  +1} } 
				\dots 
				w_{t_k}
			\\
			&= 
				w_{t_1} \dots w_{t_{i-1}}
 				w_{t_{i+1}} \dots 
				w_{t_k}
 			\end{align*}
 			which is the element $x \backslash \{ t_i\}$. 
 			Note that $\ell(x \backslash \{ t_i\}) = \ell(x) - t_i$, which is the length of the expression $\ur_1 \ur_2$. 
 			Hence $\ur_1 \ur_2$ is reduced. 
 		\end{proof}
 
		Given an expression $\uz$, denote by $\uz^0$ (resp. $\uz^1$) the subexpression where each index is 0 (resp. $1$). 
		For each $E \subseteq E(x)$ we define an associated subexpression $\ux^{E} \subseteq \ux$ as
		\begin{align*}
			\ux^E
			= 
			\ur_{1}^1 \uu_{t_{i_1}}^0 \ur_{2}^1 \uu_{t_{i_2}}^0 
			\dots 
			\ur_{k}^1 \uu_{t_{i_k}}^0 \ur_{k+1}^1.
		\end{align*}
		
		\begin{prop}
		\label{prop: bruhat strolls}
			Fix $x \in {}^I W$ and $E \subseteq E(x)$. 
			The subexpression $\ux^{E} \subseteq \ux$ is $I$-parabolic, with defect $0$ and no $\begin{smallmatrix} \da		\\1		\end{smallmatrix}$'s occurring in the Bruhat stroll. 
			Moreover, $\left(\ux^{E} \right)_{\bullet} = x \backslash E$. 
		\end{prop}
		\begin{proof}
			Fix $x = \{ t_1 > \dots > t_l \}$. 
			We prove the claim by induction on $|E|$.
			If $|E|= 0$, the claim is trivial. 
			\par 
			Suppose $E= \{ t_i \}$. 
			Then $\ux = \ur_1 \uu_{t_i} \ur_2$ and $\ux^E = \ur_{1}^{1} \uu_{t_i}^{0} \ur_{2}^{1}$. 
			Since $\ux$ is reduced, $\ur_1$ must also be reduced.
			In particular every decoration of $\ur_1^1$ has the form $\begin{smallmatrix} \ua		\\1		\end{smallmatrix}$.
			It is clear $\ur_1$ is a reduced expression for 
			\begin{align*}
				r_1 
				= 
				\{ t_1 > \dots > t_{i-1} > t_i -1 > t_i -3 > \dots > 3  > 1 \}.
			\end{align*} 
			The action given in Equation \ref{eqn: action on sets} implies if $j \leq t_i -1$ then
			\begin{align*}
				r_1 s_j > r_1 &~~~ \text{ if } j \text{ is odd, and}
				\\
				r_1 s_j < r_1 &~~~ \text{ if } j \text{ is even.}
			\end{align*} 
			Hence every term in $\uu_{t_i}^0$ with odd subscript has decoration $\begin{smallmatrix} \ua		\\0		\end{smallmatrix}$, and every term with even subscript has decoration $\begin{smallmatrix} \da		\\0		\end{smallmatrix}$.
			Consequently the subexpression $\ur_1^1 \uu_{t_i}^0$ has defect 0.
			Lemma \ref{lem: subexpression 121 removal} implies $(\ur_1^{1} \uu_{t_i}^{0} \ur_2^{1})_{\bullet} = x\backslash \{t_i  \}$.   
			Since $\ur_1 \ur_2$ is reduced and evaluates to $x \backslash \{t_i \}$, it follows $\ux^E$ has no $\begin{smallmatrix} \da		\\1\end{smallmatrix}$'s. 
			\par 
			Now suppose $E = \{ t_{i_1} > \dots > t_{i_k} \}$ for $k>1$. 
			We can write $\ux$ as the concatenation 
			\begin{align*}
				\ux
				=
				\ur_{1} 
				\uu_{t_{i_1}}
				\dots 
				\uu_{t_{i_{k-1}}}
				\ur_{k}
				\uu_{t_{i_k}}
				\ur_{k+1}
			\end{align*}
			Define $\uy = \ur_1 \uu_{t_1} \dots \ur_k$ and $y = \uy_{\bullet}$. 
			As $\ux$ is a reduced expression for $x$ it follows $\uy$ is a reduced expression for $y$ and $y\in {}^I W$. 
			Moreover, $E \backslash \{ t_{i_k} \} \subset  E(y)$.
			By induction $\uy^{E \backslash\{ t_{i_k}\}}$ is a defect-0 subexpression of $\uy$ that has no $\begin{smallmatrix} \da		\\1		\end{smallmatrix}$'s in its Bruhat stroll and evaluates to $y \backslash (E \backslash \{ t_{i_k} \} )$. 
			Now observe that 
			\begin{align*}
				\ux^E = \uy^{E \backslash \{ t_{i_k} \}} \uu_{t_{i_k}}^0 \ur_{k+1}^1
			\end{align*}
			so we can argue exactly as in the base case to prove the claim. 
		\end{proof}
		
		In Section \ref{Ssec: bs basis} we prove every defect-0 subexpression of $\ux$ has the form $\ux^E$ where $E \subseteq E(x)$. 
		
	\subsection{Entries in the local intersection form}
	\label{Ssec: entries in lif}
		
		We now construct light-leaves associated to the subexpressions $\ux^E$ in the diagrammatic category $\DRA(\h,I)$. 
		Then we determine entries in the local intersection form associated to those light-leaves.
		Recall $\h$ is a Cartan realisation of type $\tB_n$.
		Our labelling conventions imply $\pair{\alphac_0 , \alpha_1} = -2$. 
		\par 
		We call an even-indexed (resp. odd-indexed) simple reflection an \ldef{even reflection (resp. odd reflection)}.
		Given the important role that the parity of a simple reflection plays, we colour our simple reflections as follows:
		
		\[
		\begin{tikzpicture}
			 
			\draw (-1.5,0.05) -- (-1,0.05);
			\draw (-1.5,-0.05) -- (-1,-0.05);
			\draw (-1,0) -- (1.5,0);
					\draw (-1.3,0) -- (-1.1,0.15);
					\draw (-1.3,0) -- (-1.1,-0.15);
			\node[circle,fill,draw,inner sep=0mm,minimum size=2mm,color=red, thick] at (-1.5,0) {}; 
			\node[circle,fill,draw,inner sep=0mm,minimum size=2mm,color=blue, thick] at (-1,0) {};
			\node[circle,fill,draw,inner sep=0mm,minimum size=2mm,color=red!60!yellow!, thick] at (-0.5,0) {};
			\node[circle,fill,draw,inner sep=0mm,minimum size=2mm,color=blue!50!cyan, thick] at (0,0) {}; 
			\node[circle,fill,draw,inner sep=0mm,minimum size=2mm,color=red!20!yellow, thick] at (0.5,0) {}; 
			\node[circle,fill,draw,inner sep=0mm,minimum size=2mm,color=cyan, thick] at (1,0) {};
			\node[circle,fill,draw,inner sep=0mm,minimum size=2mm,color=yellow, thick] at (1.5,0) {}; 
			\filldraw[black] (-1.5,-0.5) node[font = \small]{0}; 
			\filldraw[black] (-1,-0.5) node[font = \small]{1};
			\filldraw[black] (-0.5,-0.5) node[font = \small]{2}; 
			\filldraw[black] (0,-0.5) node[font = \small]{3};
			\filldraw[black] (0.5,-0.5) node[font = \small]{4};
			\filldraw[black] (1,-0.5) node[font = \small]{5};
			\filldraw[black] (1.5,-0.5) node[font = \small]{6};	
			\filldraw[black] (2,0) node[font = \small]{\dots};
		 
		\end{tikzpicture}
		\]

		\par 
		Fix $x$ and $E \subseteq E(x)$.
		By construction $\ur_1$ is formed from alternating decreasing sequences of even and odd reflections. 
		In general $\ur_1$ can end with either a sequence of even or odd reflections. 
		If $\ur_1$ ends with a sequence of odd reflections, they commute with every subsequent reflection in the column reduced expression $\ux$ of $x$. 
		Consequently we only consider the case where every $\ur_1$ ends in a sequence of even reflections. 
		The light-leaf corresponding to $\ur_1^1$ is the identity morphism on $B_{\ur_1}$ as $\ur_1$ is reduced.
		A prototypical light-leaf corresponding to $\ur_1^1$ has the form 
		\[	
		\begin{array}{c}
		\begin{tikzpicture}[scale=0.6]
      		\draw[color=red, thick] (0,0) to (0,3);
      		\draw[color=blue, thick] (0.2,0) to (0.2,3);
      		\draw[color=red!60!yellow, thick] (0.4,0) to (0.4,3);
      		\draw[color=red, thick] (0.6,0) to (0.6,3);
      		\draw[color=blue!50!cyan, thick] (0.8,0) to (0.8,3);
      		\draw[color=blue, thick] (1,0) to (1,3);
      		\draw[color=red!20!yellow, thick] (1.2,0) to (1.2,3);
      		\draw[color=red!60!yellow, thick] (1.4,0) to (1.4,3);
      		\draw[color=red, thick] (1.6,0) to (1.6,3);
    	\end{tikzpicture}
    	\end{array}.
    	\]

    	\par      	
    	Recall that for any $t_i \in E$, the Bruhat stroll $\ux^{E}$ contains the subexpression $\uu_{t_i}^{0}$ which is a sequence of odd reflections all decorated with $\begin{smallmatrix} \ua		\\0		\end{smallmatrix}$ followed by a sequence of even reflections all decorated with $\begin{smallmatrix} \da		\\0		\end{smallmatrix}$. 
    	Consequently a typical light-leaf $LL_{\ur_1 \uu_1 , E}$  corresponding to $ \ur_1^1 \uu_{t_i}^0$  has the form
   
    	\[	
		\begin{array}{c}
		\begin{tikzpicture}[scale=0.6]
      		\draw[color=red, thick] (0,0) to (0,3);
      		\draw[color=blue, thick] (0.2,0) to (0.2,3);
      		\draw[color=red!60!yellow, thick] (0.4,0) to (0.4,3);
      		\draw[color=red, thick] (0.6,0) to (0.6,3);
      		\draw[color=blue!50!cyan, thick] (0.8,0) to (0.8,3);
      		\draw[color=blue, thick] (1,0) to (1,3);
      		\draw[color=red!20!yellow, thick] (1.2,0) to (1.2,3);
      		\draw[color=red!60!yellow, thick] (1.4,0) to (1.4,3);
      		\draw[color=red, thick] (1.6,0) to (1.6,3);
      		\draw[color=cyan, thick] (1.9,0) to (1.9,0.8);
			\node[circle,fill,draw,inner sep=0mm,minimum size=1mm,color=cyan, thick] at (1.9,0.8) {};
      		\draw[color=blue!50!cyan, thick] (2.2,0) to (2.2,0.8);
			\node[circle,fill,draw,inner sep=0mm,minimum size=1mm,color=blue!50!cyan, thick] at (2.2,0.8) {};
			\draw[color=blue, thick] (2.5,0) to (2.5,0.8);
			\node[circle,fill,draw,inner sep=0mm,minimum size=1mm,color=blue, thick] at (2.5,0.8) {};
			\draw[color=red!20!yellow, thick] (2.8,0) arc [start angle=0, end angle = 90, x radius = 1.6cm , y radius = 2cm];
			\draw[color=red!60!yellow, thick] (3,0) arc [start angle=0, end angle = 90, x radius = 1.6cm , y radius = 2.2cm];
			\draw[color=red, thick] (3.2,0) arc [start angle=0, end angle = 90, x radius = 1.6cm , y radius = 2.4cm];
    	\end{tikzpicture}
    	\end{array}.
    	\]
    	That is, it is the identity morphism on $\ur_1$, followed by a sequence of odd-coloured up-dots, followed by a sequence (of the same length) of even-coloured trivalent vertices. 
    	\par
    	Now consider the morphism $LL_{\ur_1 \uu_1, E} \circ \Gamma \Gamma^{\ur_1 \uu_1, E}$, which is the composition of the light leaf $LL_{\ur_1 \uu_1, E}$ with the morphism obtained by reflecting $LL_{\ur_1 \uu_1, E}$ along a horizontal axis. 
    	Note the following rather trivial identity holds for $a<c$ by the twisted Leibniz rule
    	\begin{align}
    	\label{eqn: twisted leibniz}
    		\partial_{s_{\red{2a}}} \left( \prod^c_{b \geq a} \alpha_{s_{\blue{2b+1}}} \right)
    		=
    		\pair{\alphac_{s_{\red{2a}}}, \alpha_{s_{\blue{2a+1}}} }
    		\prod_{b > a}^c \alpha_{s_{\blue{2b+1}}}.
    	\end{align}
    	So, it is easy to describe the action of certain Demazure operators on products of odd-indexed simple roots. 
    	Consequently we have
		\begin{align}
		\label{eqn: repeated barbell forcing}
		LL_{\ur_1 \uu_1, E} \circ \Gamma \Gamma^{\ur_1 \uu_1, E}
		=
		\begin{array}{c}
		\begin{tikzpicture}[scale=0.5]
      		\draw[color=red!20!yellow, thick] (1.2,-3) to (1.2,3);
      		\draw[color=red!60!yellow, thick] (1.4,-3) to (1.4,3);
      		\draw[color=red, thick] (1.6,-3) to (1.6,3);
      		\draw[color=cyan, thick] (1.9,-0.8) to (1.9,0.8);
			\node[circle,fill,draw,inner sep=0mm,minimum size=1mm,color=cyan, thick] at (1.9,0.8) {};
			\node[circle,fill,draw,inner sep=0mm,minimum size=1mm,color=cyan, thick] at (1.9,-0.8) {};
      		\draw[color=blue!50!cyan, thick] (2.2,-0.8) to (2.2,0.8);
			\node[circle,fill,draw,inner sep=0mm,minimum size=1mm,color=blue!50!cyan, thick] at (2.2,0.8) {};
			\node[circle,fill,draw,inner sep=0mm,minimum size=1mm,color=blue!50!cyan, thick] at (2.2,-0.8) {};
			\draw[color=blue, thick] (2.5,-0.8) to (2.5,0.8);
			\node[circle,fill,draw,inner sep=0mm,minimum size=1mm,color=blue, thick] at (2.5,0.8) {};
			\node[circle,fill,draw,inner sep=0mm,minimum size=1mm,color=blue, thick] at (2.5,-0.8) {};
			\draw[color=red!20!yellow, thick] (2.8,0) arc [start angle=0, end angle = 90, x radius = 1.6cm , y radius = 2cm];
			\draw[color=red!20!yellow, thick] (2.8,0) arc [start angle=0, end angle = -90, x radius = 1.6cm , y radius = 2cm];
			\draw[color=red!60!yellow, thick] (3,0) arc [start angle=0, end angle = 90, x radius = 1.6cm , y radius = 2.2cm];
			\draw[color=red!60!yellow, thick] (3,0) arc [start angle=0, end angle = -90, x radius = 1.6cm , y radius = 2.2cm];
			\draw[color=red, thick] (3.2,0) arc [start angle=0, end angle = 90, x radius = 1.6cm , y radius = 2.4cm];
			\draw[color=red, thick] (3.2,0) arc [start angle=0, end angle = -90, x radius = 1.6cm , y radius = 2.4cm];
    	\end{tikzpicture}
    	\end{array}
    	=
    	\begin{array}{c}
		\begin{tikzpicture}[scale=0.5]
      		\draw[color=red!20!yellow, thick] (1.2,-3) to (1.2,3);
      		\draw[color=red!60!yellow, thick] (1.4,-3) to (1.4,3);
      		\draw[color=red, thick] (1.6,-3) to (1.6,3);
      		\draw[color=cyan, thick] (1.9,-0.8) to (1.9,0.8);
			\node[circle,fill,draw,inner sep=0mm,minimum size=1mm,color=cyan, thick] at (1.9,0.8) {};
			\node[circle,fill,draw,inner sep=0mm,minimum size=1mm,color=cyan, thick] at (1.9,-0.8) {};
      		\draw[color=blue!50!cyan, thick] (2.2,-0.8) to (2.2,0.8);
			\node[circle,fill,draw,inner sep=0mm,minimum size=1mm,color=blue!50!cyan, thick] at (2.2,0.8) {};
			\node[circle,fill,draw,inner sep=0mm,minimum size=1mm,color=blue!50!cyan, thick] at (2.2,-0.8) {};
			\draw[color=blue, thick] (2.5,-0.8) to (2.5,0.8);
			\node[circle,fill,draw,inner sep=0mm,minimum size=1mm,color=blue, thick] at (2.5,0.8) {};
			\node[circle,fill,draw,inner sep=0mm,minimum size=1mm,color=blue, thick] at (2.5,-0.8) {};
			\draw[color=red, thick] (2.8,0) arc [start angle=0, end angle = 75, x radius = 1.6cm , y radius = 2cm];
			\draw[color=red, thick] (2.8,0) arc [start angle=0, end angle = -75, x radius = 1.6cm , y radius = 2cm];
			\draw[color=red!60!yellow, thick] (3,0) arc [start angle=0, end angle = 90, x radius = 1.6cm , y radius = 2.2cm];
			\draw[color=red!60!yellow, thick] (3,0) arc [start angle=0, end angle = -90, x radius = 1.6cm , y radius = 2.2cm];
			\draw[color=red!20!yellow, thick] (3.2,0) arc [start angle=0, end angle = 105, x radius = 1.6cm , y radius = 2.4cm];
			\draw[color=red!20!yellow, thick] (3.2,0) arc [start angle=0, end angle = -105, x radius = 1.6cm , y radius = 2.4cm];
    	\end{tikzpicture}
    	\end{array}
    	=
    	\prod_{j=1}^{t_i/2}
    	\pair{\alphac_{s_{2j}},\alpha_{s_{2j+1}}}
    	\begin{array}{c}
		\begin{tikzpicture}[scale=0.5]
      		\draw[color=red!20!yellow, thick] (1.2,-3) to (1.2,3);
      		\draw[color=red!60!yellow, thick] (1.4,-3) to (1.4,3);
      		\draw[color=red, thick] (1.6,-3) to (1.6,3);
    	\end{tikzpicture}
    	\end{array}
    	=
    	2(-1)^{t_i/2}
    	\begin{array}{c}
		\begin{tikzpicture}[scale=0.5]
      		\draw[color=red!20!yellow, thick] (1.2,-3) to (1.2,3);
      		\draw[color=red!60!yellow, thick] (1.4,-3) to (1.4,3);
      		\draw[color=red, thick] (1.6,-3) to (1.6,3);
    	\end{tikzpicture}
    	\end{array}
		\end{align}
		where the second equality follows from the fact that all even reflections commute; the third equality from repeatedly applying the polynomial forcing relation and Equation \ref{eqn: twisted leibniz}; and, the fourth uses that $\pair{\alphac_0 , \alpha_1} = -2$ in type $\tB_n$.  
    	\par 
    	Suppose $E = \{ t_1 , \dots , t_k\}$ where $k>1$. 
		The subexpression corresponding to $E$ is
		\begin{align*}
			\ux^E =  \ur_1^1 \uu_{t_{1}}^0 \ur_2^1 \dots \ur_k^1 \uu_{t_{k}} ^0 \ur_{k+1}^1.
		\end{align*}
		Proposition \ref{prop: bruhat strolls} implies each $\ur_j$ is exclusively decorated with $\begin{smallmatrix} \ua \\1	\end{smallmatrix}$'s, so the corresponding light leaf is the identity on each $B_{\ur_j}$. 
		Again, we can reduce to the case where each $\ur_j$ ends in a sequence of even reflections, as any odd sequence of odd reflections ending $\ur_j$ will commute with every simple reflection to the right of $\ur_j$ in $\ux$. 
		For example, if $\ur_2 \neq \emptyset$, a typical light leaf corresponding to $\ur_1^1 \uu_{t_{i_1}}^0 \ur_2^1$ has the form
		\[	
		\begin{array}{c}
		\begin{tikzpicture}[scale=0.6]
      		\draw[color=red, thick] (0,0) to (0,3);
      		\draw[color=blue, thick] (0.2,0) to (0.2,3);
      		\draw[color=red!60!yellow, thick] (0.4,0) to (0.4,3);
      		\draw[color=red, thick] (0.6,0) to (0.6,3);
      		\draw[color=blue!50!cyan, thick] (0.8,0) to (0.8,3);
      		\draw[color=blue, thick] (1,0) to (1,3);
      		\draw[color=red!20!yellow, thick] (1.2,0) to (1.2,3);
      		\draw[color=red!60!yellow, thick] (1.4,0) to (1.4,3);
      		\draw[color=red, thick] (1.6,0) to (1.6,3);
      		\draw[color=cyan, thick] (1.9,0) to (1.9,0.8);
			\node[circle,fill,draw,inner sep=0mm,minimum size=1mm,color=cyan, thick] at (1.9,0.8) {};
      		\draw[color=blue!50!cyan, thick] (2.2,0) to (2.2,0.8);
			\node[circle,fill,draw,inner sep=0mm,minimum size=1mm,color=blue!50!cyan, thick] at (2.2,0.8) {};
			\draw[color=blue, thick] (2.5,0) to (2.5,0.8);
			\node[circle,fill,draw,inner sep=0mm,minimum size=1mm,color=blue, thick] at (2.5,0.8) {};
			\draw[color=red!20!yellow, thick] (2.8,0) arc [start angle=0, end angle = 90, x radius = 1.6cm , y radius = 2cm];
			\draw[color=red!60!yellow, thick] (3,0) arc [start angle=0, end angle = 90, x radius = 1.6cm , y radius = 2.2cm];
			\draw[color=red, thick] (3.2,0) arc [start angle=0, end angle = 90, x radius = 1.6cm , y radius = 2.4cm];
      		\draw[color=blue!50!cyan, thick] (3.4,0) to (3.4,3);
      		\draw[color=blue, thick] (3.6,0) to (3.6,3);
      		\draw[color=red!60!yellow, thick] (3.8,0) to (3.8,3);
      		\draw[color=red, thick] (4,0) to (4,3);
    	\end{tikzpicture}
    	\end{array}.
    	\]
    	On the other hand, each $\uu_{t_i}^0$ is a sequence of odd reflections each decorated with $\begin{smallmatrix} \ua \\0	\end{smallmatrix}$'s, followed by a sequence of even reflections each decorated with $\begin{smallmatrix} \da \\0	\end{smallmatrix}$'s. 
    	Since both $\ur_j$ and $\uu_{t_{j}}$ end in a sequence of even reflections (with subscripts at least as large as every even reflection appearing to the right in $\ux$), the even reflections in $\uu_{t_{i}}$ must form trivalent vertices with the even reflections in $\ur_{i}$ if $\ur_{i} \neq \emptyset$, or $\uu_{t_{i-1}}$ if $\ur_i = \emptyset$. 
    	Thus, typical light leaves for $\ur_1 \uu_{t_i} \ur_2 \uu_{t_2}$ have the form 
    	\[
    	\begin{array}{c}
		\begin{tikzpicture}[scale=0.6]
      		\draw[color=red, thick] (0,0) to (0,3);
      		\draw[color=blue, thick] (0.2,0) to (0.2,3);
      		\draw[color=red!60!yellow, thick] (0.4,0) to (0.4,3);
      		\draw[color=red, thick] (0.6,0) to (0.6,3);
      		\draw[color=blue!50!cyan, thick] (0.8,0) to (0.8,3);
      		\draw[color=blue, thick] (1,0) to (1,3);
      		\draw[color=red!20!yellow, thick] (1.2,0) to (1.2,3);
      		\draw[color=red!60!yellow, thick] (1.4,0) to (1.4,3);
      		\draw[color=red, thick] (1.6,0) to (1.6,3);
      		\draw[color=cyan, thick] (1.9,0) to (1.9,0.8);
			\node[circle,fill,draw,inner sep=0mm,minimum size=1mm,color=cyan, thick] at (1.9,0.8) {};
      		\draw[color=blue!50!cyan, thick] (2.2,0) to (2.2,0.8);
			\node[circle,fill,draw,inner sep=0mm,minimum size=1mm,color=blue!50!cyan, thick] at (2.2,0.8) {};
			\draw[color=blue, thick] (2.5,0) to (2.5,0.8);
			\node[circle,fill,draw,inner sep=0mm,minimum size=1mm,color=blue, thick] at (2.5,0.8) {};
			\draw[color=red!20!yellow, thick] (2.8,0) arc [start angle=0, end angle = 90, x radius = 1.6cm , y radius = 2cm];
			\draw[color=red!60!yellow, thick] (3,0) arc [start angle=0, end angle = 90, x radius = 1.6cm , y radius = 2.2cm];
			\draw[color=red, thick] (3.2,0) arc [start angle=0, end angle = 90, x radius = 1.6cm , y radius = 2.4cm];
      		\draw[color=blue!50!cyan, thick] (3.4,0) to (3.4,3);
      		\draw[color=blue, thick] (3.6,0) to (3.6,3);
      		\draw[color=red!60!yellow, thick] (3.8,0) to (3.8,3);
      		\draw[color=red, thick] (4,0) to (4,3);
      		\draw[color=blue, thick] (4.3,0) to (4.3,0.8);
			\node[circle,fill,draw,inner sep=0mm,minimum size=1mm,color=blue, thick] at (4.3,0.8) {};
			\draw[color=red, thick] (4.6,0) arc [start angle=0, end angle = 90, x radius = 0.6cm , y radius = 1.5cm];
    	\end{tikzpicture}
    	\end{array}
    	~~~~
    	\text{or}
    	~~~~
    	\begin{array}{c}
		\begin{tikzpicture}[scale=0.6]
      		\draw[color=red, thick] (0,0) to (0,3);
      		\draw[color=blue, thick] (0.2,0) to (0.2,3);
      		\draw[color=red!60!yellow, thick] (0.4,0) to (0.4,3);
      		\draw[color=red, thick] (0.6,0) to (0.6,3);
      		\draw[color=blue!50!cyan, thick] (0.8,0) to (0.8,3);
      		\draw[color=blue, thick] (1,0) to (1,3);
      		\draw[color=red!20!yellow, thick] (1.2,0) to (1.2,3);
      		\draw[color=red!60!yellow, thick] (1.4,0) to (1.4,3);
      		\draw[color=red, thick] (1.6,0) to (1.6,3);
      		\draw[color=cyan, thick] (1.9,0) to (1.9,0.8);
			\node[circle,fill,draw,inner sep=0mm,minimum size=1mm,color=cyan, thick] at (1.9,0.8) {};
      		\draw[color=blue!50!cyan, thick] (2.2,0) to (2.2,0.8);
			\node[circle,fill,draw,inner sep=0mm,minimum size=1mm,color=blue!50!cyan, thick] at (2.2,0.8) {};
			\draw[color=blue, thick] (2.5,0) to (2.5,0.8);
			\node[circle,fill,draw,inner sep=0mm,minimum size=1mm,color=blue, thick] at (2.5,0.8) {};
			\draw[color=red!20!yellow, thick] (2.8,0) arc [start angle=0, end angle = 90, x radius = 1.6cm , y radius = 2cm];
			\draw[color=red!60!yellow, thick] (3,0) arc [start angle=0, end angle = 90, x radius = 1.6cm , y radius = 2.2cm];
			\draw[color=red, thick] (3.2,0) arc [start angle=0, end angle = 90, x radius = 1.6cm , y radius = 2.4cm];	
      		\draw[color=blue!50!cyan, thick] (3.5,0) to (3.5,0.4);
			\node[circle,fill,draw,inner sep=0mm,minimum size=1mm,color=blue!50!cyan, thick] at (3.5,0.4) {};
			\draw[color=blue, thick] (3.8,0) to (3.8,0.4);
			\node[circle,fill,draw,inner sep=0mm,minimum size=1mm,color=blue, thick] at (3.8,0.4) {};
      		\draw[color=red!60!yellow, thick] (4.2,0) arc [start angle=0, end angle = 89, x radius = 1.6cm , y radius = 1.4cm];
			\draw[color=red, thick] (4.4,0) arc [start angle=0, end angle = 90, x radius = 1.6cm , y radius = 1.6cm];
    	\end{tikzpicture}
    	\end{array} 
    	\]
		where in the first example $\ur_2 \neq \emptyset$, and in the second $\ur_2 = \emptyset$.  
		Repeatedly applying Equation \ref{eqn: repeated barbell forcing} to the composition $LL_{\ux, E} \circ \Gamma \Gamma^{\ux, E}$ shows:
		
		\begin{prop}
		\label{prop: local intersection form}
			Let $x \in {}^I W$ and $\ux$ be the column reduced expression. 
			Fix $E = \{t_1 , \dots  ,t_k\} \subseteq E(x)$, and recall the decomposition $\ux = \ur_1 \uu_{t_1} \dots \ur_k \uu_{t_k} \ur_{k+1}$ induced by $E$,  then 
			\begin{align*}
				LL_{\ux, E} \circ \Gamma \Gamma^{\ux, E} 
				=
				2^{|E|} (-1)^{\ell(E)/2} \id_{B_{\ur_1 \ur_2 \ur_3 \dots \ur_{k+1}}} 
			\end{align*}
			where $\id_{B_{\ur_1 \ur_2 \ur_3 \dots \ur_{k+1}}}$ is the identity morphism on the Bott-Samelson object $B_{\ur_1 \ur_2 \ur_3 \dots \ur_{k+1}}$, $|E|=k$ and $\ell(E) = \sum_i t_i$. 
		\end{prop}
		
		We illustrate\footnote{Pun intended.} the Proposition with an aesthetically pleasing example. 
		
		\begin{exmp}
		If $x=\{6,5,4,3,2,1\}$ then $E(x)=\{6,4,2\}$. 
		Set $E=E(x)$. 
		The composition of the light leaves $LL_{\ux,E} \circ \Gamma\Gamma^{\ux,E}$ is
		\[
		\begin{array}{c}
		\begin{tikzpicture}[scale=0.6]
      		\draw[color=red, thick] (0,-3) to (0,3);
      		\draw[color=blue, thick] (0.2,-3) to (0.2,3);
      		\draw[color=red!60!yellow, thick] (0.4,-3) to (0.4,3);
      		\draw[color=red, thick] (0.6,-3) to (0.6,3);
      		\draw[color=blue!50!cyan, thick] (0.8,-3) to (0.8,3);
      		\draw[color=blue, thick] (1,-3) to (1,3);
      		\draw[color=red!20!yellow, thick] (1.2,-3) to (1.2,3);
      		\draw[color=red!60!yellow, thick] (1.4,-3) to (1.4,3);
      		\draw[color=red, thick] (1.6,-3) to (1.6,3);
      		\draw[color=cyan, thick] (1.9,-0.8) to (1.9,0.8);
			\node[circle,fill,draw,inner sep=0mm,minimum size=1mm,color=cyan, thick] at (1.9,0.8) {};
			\node[circle,fill,draw,inner sep=0mm,minimum size=1mm,color=cyan, thick] at (1.9,-0.8) {};
      		\draw[color=blue!50!cyan, thick] (2.2,-0.8) to (2.2,0.8);
			\node[circle,fill,draw,inner sep=0mm,minimum size=1mm,color=blue!50!cyan, thick] at (2.2,0.8) {};
			\node[circle,fill,draw,inner sep=0mm,minimum size=1mm,color=blue!50!cyan, thick] at (2.2,-0.8) {};
			\draw[color=blue, thick] (2.5,-0.8) to (2.5,0.8);
			\node[circle,fill,draw,inner sep=0mm,minimum size=1mm,color=blue, thick] at (2.5,0.8) {};
			\node[circle,fill,draw,inner sep=0mm,minimum size=1mm,color=blue, thick] at (2.5,-0.8) {};
			\draw[color=red!20!yellow, thick] (2.8,0) arc [start angle=0, end angle = 90, x radius = 1.6cm , y radius = 2cm];
			\draw[color=red!20!yellow, thick] (2.8,0) arc [start angle=0, end angle = -90, x radius = 1.6cm , y radius = 2cm];
			\draw[color=red!60!yellow, thick] (3,0) arc [start angle=0, end angle = 90, x radius = 1.6cm , y radius = 2.2cm];
			\draw[color=red!60!yellow, thick] (3,0) arc [start angle=0, end angle = -90, x radius = 1.6cm , y radius = 2.2cm];
			\draw[color=red, thick] (3.2,0) arc [start angle=0, end angle = 90, x radius = 1.6cm , y radius = 2.4cm];	
			\draw[color=red, thick] (3.2,0) arc [start angle=0, end angle = -90, x radius = 1.6cm , y radius = 2.4cm];	
      		\draw[color=blue!50!cyan, thick] (3.5,-0.4) to (3.5,0.4);
			\node[circle,fill,draw,inner sep=0mm,minimum size=1mm,color=blue!50!cyan, thick] at (3.5,0.4) {};
			\node[circle,fill,draw,inner sep=0mm,minimum size=1mm,color=blue!50!cyan, thick] at (3.5,-0.4) {};
			\draw[color=blue, thick] (3.8,-0.4) to (3.8,0.4);
			\node[circle,fill,draw,inner sep=0mm,minimum size=1mm,color=blue, thick] at (3.8,0.4) {};
			\node[circle,fill,draw,inner sep=0mm,minimum size=1mm,color=blue, thick] at (3.8,-0.4) {};
      		\draw[color=red!60!yellow, thick] (4.2,0) arc [start angle=0, end angle = 89, x radius = 1.6cm , y radius = 1.4cm];
      		\draw[color=red!60!yellow, thick] (4.2,0) arc [start angle=0, end angle = -89, x radius = 1.6cm , y radius = 1.4cm];
			\draw[color=red, thick] (4.4,0) arc [start angle=0, end angle = 90, x radius = 1.6cm , y radius = 1.6cm];
			\draw[color=red, thick] (4.4,0) arc [start angle=0, end angle = -90, x radius = 1.6cm , y radius = 1.6cm];
			\draw[color=blue, thick] (4.7,-0.3) to (4.7,0.3);
			\node[circle,fill,draw,inner sep=0mm,minimum size=1mm,color=blue, thick] at (4.7,0.3) {};
			\node[circle,fill,draw,inner sep=0mm,minimum size=1mm,color=blue, thick] at (4.7,-0.3) {};
			\draw[color=red, thick] (5,0) arc [start angle=0, end angle = 90, x radius = 0.8cm , y radius = 0.8cm];
			\draw[color=red, thick] (5,0) arc [start angle=0, end angle = -90, x radius = 0.8cm , y radius = 0.8cm];
    	\end{tikzpicture}
    	\end{array}
    	=
    	2(-1)^{1}
    	\begin{array}{c}
		\begin{tikzpicture}[scale=0.6]
      		\draw[color=red, thick] (0,-3) to (0,3);
      		\draw[color=blue, thick] (0.2,-3) to (0.2,3);
      		\draw[color=red!60!yellow, thick] (0.4,-3) to (0.4,3);
      		\draw[color=red, thick] (0.6,-3) to (0.6,3);
      		\draw[color=blue!50!cyan, thick] (0.8,-3) to (0.8,3);
      		\draw[color=blue, thick] (1,-3) to (1,3);
      		\draw[color=red!20!yellow, thick] (1.2,-3) to (1.2,3);
      		\draw[color=red!60!yellow, thick] (1.4,-3) to (1.4,3);
      		\draw[color=red, thick] (1.6,-3) to (1.6,3);
      		\draw[color=cyan, thick] (1.9,-0.8) to (1.9,0.8);
			\node[circle,fill,draw,inner sep=0mm,minimum size=1mm,color=cyan, thick] at (1.9,0.8) {};
			\node[circle,fill,draw,inner sep=0mm,minimum size=1mm,color=cyan, thick] at (1.9,-0.8) {};
      		\draw[color=blue!50!cyan, thick] (2.2,-0.8) to (2.2,0.8);
			\node[circle,fill,draw,inner sep=0mm,minimum size=1mm,color=blue!50!cyan, thick] at (2.2,0.8) {};
			\node[circle,fill,draw,inner sep=0mm,minimum size=1mm,color=blue!50!cyan, thick] at (2.2,-0.8) {};
			\draw[color=blue, thick] (2.5,-0.8) to (2.5,0.8);
			\node[circle,fill,draw,inner sep=0mm,minimum size=1mm,color=blue, thick] at (2.5,0.8) {};
			\node[circle,fill,draw,inner sep=0mm,minimum size=1mm,color=blue, thick] at (2.5,-0.8) {};
			\draw[color=red!20!yellow, thick] (2.8,0) arc [start angle=0, end angle = 90, x radius = 1.6cm , y radius = 2cm];
			\draw[color=red!20!yellow, thick] (2.8,0) arc [start angle=0, end angle = -90, x radius = 1.6cm , y radius = 2cm];
			\draw[color=red!60!yellow, thick] (3,0) arc [start angle=0, end angle = 90, x radius = 1.6cm , y radius = 2.2cm];
			\draw[color=red!60!yellow, thick] (3,0) arc [start angle=0, end angle = -90, x radius = 1.6cm , y radius = 2.2cm];
			\draw[color=red, thick] (3.2,0) arc [start angle=0, end angle = 90, x radius = 1.6cm , y radius = 2.4cm];	
			\draw[color=red, thick] (3.2,0) arc [start angle=0, end angle = -90, x radius = 1.6cm , y radius = 2.4cm];	
      		\draw[color=blue!50!cyan, thick] (3.5,-0.4) to (3.5,0.4);
			\node[circle,fill,draw,inner sep=0mm,minimum size=1mm,color=blue!50!cyan, thick] at (3.5,0.4) {};
			\node[circle,fill,draw,inner sep=0mm,minimum size=1mm,color=blue!50!cyan, thick] at (3.5,-0.4) {};
			\draw[color=blue, thick] (3.8,-0.4) to (3.8,0.4);
			\node[circle,fill,draw,inner sep=0mm,minimum size=1mm,color=blue, thick] at (3.8,0.4) {};
			\node[circle,fill,draw,inner sep=0mm,minimum size=1mm,color=blue, thick] at (3.8,-0.4) {};
      		\draw[color=red!60!yellow, thick] (4.2,0) arc [start angle=0, end angle = 89, x radius = 1.6cm , y radius = 1.4cm];
      		\draw[color=red!60!yellow, thick] (4.2,0) arc [start angle=0, end angle = -89, x radius = 1.6cm , y radius = 1.4cm];
			\draw[color=red, thick] (4.4,0) arc [start angle=0, end angle = 90, x radius = 1.6cm , y radius = 1.6cm];
			\draw[color=red, thick] (4.4,0) arc [start angle=0, end angle = -90, x radius = 1.6cm , y radius = 1.6cm];
    	\end{tikzpicture}
    	\end{array}
    	=
    	2^2(-1)^{3}
    	\begin{array}{c}
		\begin{tikzpicture}[scale=0.6]
      		\draw[color=red, thick] (0,-3) to (0,3);
      		\draw[color=blue, thick] (0.2,-3) to (0.2,3);
      		\draw[color=red!60!yellow, thick] (0.4,-3) to (0.4,3);
      		\draw[color=red, thick] (0.6,-3) to (0.6,3);
      		\draw[color=blue!50!cyan, thick] (0.8,-3) to (0.8,3);
      		\draw[color=blue, thick] (1,-3) to (1,3);
      		\draw[color=red!20!yellow, thick] (1.2,-3) to (1.2,3);
      		\draw[color=red!60!yellow, thick] (1.4,-3) to (1.4,3);
      		\draw[color=red, thick] (1.6,-3) to (1.6,3);
      		\draw[color=cyan, thick] (1.9,-0.8) to (1.9,0.8);
			\node[circle,fill,draw,inner sep=0mm,minimum size=1mm,color=cyan, thick] at (1.9,0.8) {};
			\node[circle,fill,draw,inner sep=0mm,minimum size=1mm,color=cyan, thick] at (1.9,-0.8) {};
      		\draw[color=blue!50!cyan, thick] (2.2,-0.8) to (2.2,0.8);
			\node[circle,fill,draw,inner sep=0mm,minimum size=1mm,color=blue!50!cyan, thick] at (2.2,0.8) {};
			\node[circle,fill,draw,inner sep=0mm,minimum size=1mm,color=blue!50!cyan, thick] at (2.2,-0.8) {};
			\draw[color=blue, thick] (2.5,-0.8) to (2.5,0.8);
			\node[circle,fill,draw,inner sep=0mm,minimum size=1mm,color=blue, thick] at (2.5,0.8) {};
			\node[circle,fill,draw,inner sep=0mm,minimum size=1mm,color=blue, thick] at (2.5,-0.8) {};
			\draw[color=red!20!yellow, thick] (2.8,0) arc [start angle=0, end angle = 90, x radius = 1.6cm , y radius = 2cm];
			\draw[color=red!20!yellow, thick] (2.8,0) arc [start angle=0, end angle = -90, x radius = 1.6cm , y radius = 2cm];
			\draw[color=red!60!yellow, thick] (3,0) arc [start angle=0, end angle = 90, x radius = 1.6cm , y radius = 2.2cm];
			\draw[color=red!60!yellow, thick] (3,0) arc [start angle=0, end angle = -90, x radius = 1.6cm , y radius = 2.2cm];
			\draw[color=red, thick] (3.2,0) arc [start angle=0, end angle = 90, x radius = 1.6cm , y radius = 2.4cm];	
			\draw[color=red, thick] (3.2,0) arc [start angle=0, end angle = -90, x radius = 1.6cm , y radius = 2.4cm];	
    	\end{tikzpicture}
    	\end{array}
    	=
    	2^3(-1)^6
    	\begin{array}{c}
		\begin{tikzpicture}[scale=0.6]
      		\draw[color=red, thick] (0,-3) to (0,3);
      		\draw[color=blue, thick] (0.2,-3) to (0.2,3);
      		\draw[color=red!60!yellow, thick] (0.4,-3) to (0.4,3);
      		\draw[color=red, thick] (0.6,-3) to (0.6,3);
      		\draw[color=blue!50!cyan, thick] (0.8,-3) to (0.8,3);
      		\draw[color=blue, thick] (1,-3) to (1,3);
      		\draw[color=red!20!yellow, thick] (1.2,-3) to (1.2,3);
      		\draw[color=red!60!yellow, thick] (1.4,-3) to (1.4,3);
      		\draw[color=red, thick] (1.6,-3) to (1.6,3);
    	\end{tikzpicture}
    	\end{array}.
		\]
		For the sake of completeness we emphasise $|E|=3$ and $\ell(E)/2= (6+4+2)/2 = 6$. 
		\end{exmp}
		
		\begin{rem}
			In type $\tC_n$ we have $\pair{\alphac_0 , \alpha_1}= -1$.
			Modifying the preceding argument shows 
			\begin{align*}
				LL_{\ux, E} \circ \Gamma \Gamma^{\ux, E} 
				=
				(-1)^{\ell(E)/2} \id_{B_{\ur_1 \ur_2 \ur_3 \dots \ur_{k+1}}} .
			\end{align*} 	 
		\end{rem}
		\begin{rem}
			Proposition \ref{prop: bruhat strolls} proves the Bruhat strolls associated to $\ux^E$ contain no $\begin{smallmatrix}\da\\1\end{smallmatrix}$'s. 
			Consequently a non-diagrammatic proof of Proposition \ref{prop: local intersection form} can be obtained using the nil-Hecke ring and the `gobbling morphisms' of \cite{HW18}.
		\end{rem}

	\subsection{The classical Kazhdan-Lusztig basis}
	
		We recall the description of the classical Kazhdan-Lusztig basis of ${}^I N$ as determined in \cite{LS13}. 
		They use different combinatorial objects to index their minimal coset representatives, and consider the antispherical module of type  $\tA_n \backslash \tD_{n+1}$. 
		Lemmata \ref{Lem: reduce to simply laced cats} and \ref{Lem: Other Koszul duality Whit version} imply the Kazhdan-Lusztig basis of type  $\tA_n \backslash \tD_{n+1}$ is the Kazhdan-Lusztig basis in type $\tA_{n-1} \backslash \tB_{n}$.
		We also translate their results into our language of sets and partitions.
		\par 
		For this section only, we consider $x \in {}^I W$ as an even cardinality subset of $[\tilde{n}]:= \{ 0, 1, \dots , n \}$, instead of the usual $[n]:= \{ 1, \dots, n \}$. 
		We do this in the obvious way:
		\begin{align*}
		\begin{array}{ccl}
			\cP([n])
		&
			\hookrightarrow
		&
			\cP([\tilde{n}])
		\\
			x
		&
			\mapsto
		&
			\begin{cases}
				x & \text{if } |x| \text{ is even, }
				\\
				x \cup \{ 0 \} & \text{if } |x| \text{ is odd. }
			\end{cases}
		\end{array}
		\end{align*}
		For each $x \in {}^I W$, we construct a relation $\sim_{x}$ on $[\tilde{n}]$ in two steps. 
		If $i,j \in [\tilde{n}]$ and $i \sim_{x} j$ we say $i$ and $j$ are \ldef{$x$-related}. 
		\begin{itemize}
			\item \textit{Step 1}: Say $i \sim_{x} j$ if:  $i>j$, $i \in x$, $j \notin x$, and $i$ is the minimal number such that $|\{ i' \in x ~\vert~ i \geq i' \geq j \}| = |\{ j' \notin x ~\vert~ i\geq j' \geq j \}|$. 
			\item \textit{Step 2}: Enumerate those $t_{1} , \dots ,t_{k} \in x$ that were not joined in step one,  such that $t_{1} < \dots < t_{k}$. 
		 		For any integer $j$, where $1\leq j \leq k/2$, say $t_{{2j-1}} \sim_x t_{2j}$. 
		 		Note that if $k$ is odd, then $t_{k}$ is not $x$-related to anything. 
		\end{itemize}
		
		\begin{exmp}
			Let $x= [4] \subset [\tilde{4}] $. 
			Note that $0 \notin x$. 
			So $1 \sim_x 0 $ in Step 1 and $3 \sim_x 2 $ in Step $2$. 
			The element $4 \in [\tilde{4}]$ is not $x$-related to any element. 
		\end{exmp}
		
		\par
		For each $x$-related pair $i,j \in [\tilde{n}]$ where $i>j$, define maps 
		\begin{align*}
			\phi_{i,j}^x (z) 
		=
			\begin{cases}
				\left( z \backslash \lbrace i \rbrace \right) \cup \lbrace j \rbrace
			&
				\text{if } i,j \text{ are joined in step one, and}
			\\
				z \backslash \lbrace i,j \rbrace 
			&
				\text{if } i,j \text{ are joined in step two.}
			\end{cases}
		\end{align*}
		If there is a sequence of maps $\phi^{x}_{i_1, j_1}, \dots , \phi^{x}_{i_k, j_k}$ satisfying $\phi^{x}_{i_k, j_k} \circ \dots \circ \phi^{x}_{i_1, j_1}(x) =y$, define $\#(y,x)$ to be the minimal length of such a sequence. Otherwise, set $\#(x,y) = \infty$. 
		\par 
		We can now state \cite[\S 2.2]{LS13} by another name.
		
		\begin{prop}
		\label{prop: kl basis}
			Let $x,y \in {}^I W$, then
			\begin{align*}
				n_{y,x}
			=	
				\begin{cases}
					v^{\#(y,x)}
				&
					\text{if } \#(y,x) < \infty,	
				\\
					0
				&
					\text{otherwise.}
				\end{cases}
			\end{align*}
		\end{prop}

		We sketch how to translate our terminology into theirs. 
		For $x \subseteq [\tilde{n}]$, define a sequence of length $n+1$ whose $(i+1)^{\text{th}}$ term is $-$ if $i \in x$ and $+$ if $i \notin x$. 
		This is the $\{ -,+\}$-sequence of \cite[\S 1]{LS13} corresponding to the image of $x$ under the poset isomorphism between $({}^{\tA_{n}} W_{\tD_{n+1}}, \leq )$ and $({}^{\tA_{n-1}} W_{\tC_{n}}, \leq )$. 
		To a $\{ -,+\}$-sequence Lejczyk and Stroppel associate a `cup diagram', see \cite[\S2.1]{LS13}. 
		A tedious definition chase shows that $i \sim_x j$ in:
		\begin{itemize}
			\item step 1 correspond to points in $\{ 1, \dots , n+1 \}$ which are connected by a cup in the cup diagram for the $\{ -,+\}$-sequence associated to $x$, and
			\item step 2 correspond to points in $\{\pm 1 , \dots , \pm (n+1) \}$ which are connected by a cup which crosses 0 in the cup diagram for the $\{ -,+\}$-sequence associated to $x$. 
		\end{itemize} 
		Finally, the maps $\phi^{x}_{i,j}$ correspond to `clockwise-oriented' weights, as we insist $i>j$.

		\begin{rem}
			The maps $\phi^x_{i,j}$ have an intuitive description in terms of partitions.
			Fix $i>j$. 
			If $i \sim_x j$ are linked in step 1 then $\phi^x_{i,j}$ flips row $i$ onto row $j$; recall that we have an empty row when $j=0$.
			If $i \sim_x j$ are linked in step 2 then $\phi^x_{i,j}$ removes both rows $i$ and $j$. 
		\end{rem}
				
		\begin{exmp}
			Let $x = [4] \subset [\tilde{4}]$. Then the maps $\phi_{1,0}^x$ and $\phi_{3,2}^{x}$ are given by 
			\[
			\phi_{1,0}^x 
			\left(
			\begin{array}{c}
			\tableau{
				0&1&2&3 	 		\\
				0&1&2  		\\
				0&1	\\
				0   
			}
			\end{array}
			\right)
			=
			\begin{array}{c}
			\tableau{
				0 & 1 & 2 & 3 	 		\\
				0 & 1 & 2  		\\
				0 & 1	\\
				\hfill   
			}
			\end{array}
			~~~~
			\text{and}
			~~~~
			\phi_{3,2}^x 
			\left(
			\begin{array}{c}
			\tableau{
				0&1&2&3 	 		\\
				0&1&2  		\\
				0&1	\\
				0   
			}
			\end{array}
			\right)
			=
			\begin{array}{c}
			\tableau{
				0&1&2&3 	 		\\
				\hfill & \hfill & \hfill  		\\
				\hfill & \hfill	\\
				0   
			}
			\end{array}.
			\]
			So $d_{\{4,3,2,1\}} = \delta_{\{4,3,2,1\}}^{I} + v \delta_{\{4,3,2\}}^I  + v \delta_{\{4,1\}}^I + v^2 \delta_{\{ 4\}}^I$.  
		\end{exmp}

		Observe $\#(y,x) = 1$ if and only if $y$ is obtained from $x$ by applying exactly one row flip or row removal. 
		Consequently, the symmetric difference of $x$ and $y$ is $\{ i,j\}$ where $i \sim_x j$. 
	
	\subsection{Bott-Samelson basis}
	\label{Ssec: bs basis}

		We now determine the Bott-Samelson basis of ${}^I N$. 
		The inductive construction utilises the \textit{row reduced expression} of $x$. 
		However, since each $x \in {}^I W$ is fully commutative, each result holds independent of the choice of reduced expression.  

		\begin{lem}
		\label{lem: E set}
			If $x = \{t_1 > \dots > t_l > i+1 \} \in {}^I W$, then either $E(x) = E(xs_i)$ or $E(x) = E(xs_i) \cup \{t_j \}$ where $t_{j+ t_j/2} =1$. 
		\end{lem}
		\begin{proof}
			By Equation \ref{eqn: action on sets} we have
			\begin{align*}
				xs_i =
				\begin{cases}
					\{t_1 > \dots > t_l > i \} 
					&\text{ if } i \geq 1,
					\\
					\{t_1 > \dots > t_l  \}
					&\text{ if } i=0.
				\end{cases}
			\end{align*}
			So $x$ and $xs_i$ only differ in their minimal elements. 
			In either case $x>xs_i$. 
			Note that for any $y \in {}^I W$, if $t_j \in E(y)$ then  $t_{j + t_{j}/2} \geq 1$, as row $(j + t_{j}/2)$ in the partition corresponding to $y$ is non-empty. 
			If $i \geq 1$ then any $t \in E(x)$ also satisfies the row-length condition in Equation \ref{eqn: row length condition} to be in $E(xs_i)$ and vice versa. 
			Hence $E(xs_i) = E(x)$. 
			\par 
			If $i=0$, then $x = \{t_1 > \dots > t_l > 1 \}$ and $xs_0 = \{t_1 > \dots > t_l \}$. 
			There are two possibilities to consider. 
			\par  
			Suppose $t_j \in E(x)$ is minimal and $t_{j+ t_j/2} = 1$, then $t_j \notin E(xs_0)$ as $|xs_0| = j+ t_j/2 - 1$, i.e. the partition $xs_0$ has too few parts. 
			Hence $E(x) = E(xs_0) \cup \{ t_j \}$.
			Alternatively, suppose $t_j \in E(x)$ is minimal and $t_{j+ t_j/2} > 1$, then $t_j \in E(xs_0)$ as the row-length condition in Equation \ref{eqn: row length condition} remains satisfied. 
			Hence $E(x) = E(xs_0)$. 	
		\end{proof}

		We can now determine the Bott-Samelson bases for ${}^I N$.
		
		\begin{prop}
		\label{prop: bs decomp}
			Fix $x \in {}^I W$, and let $\ux$ be any reduced expression. Then
			\begin{align*}
				d_{\ux}  = \sum_{E \subseteq E(x)} d_{x \backslash E}.
			\end{align*} 
		\end{prop}

		\begin{proof}
			We prove the claim by induction on $\ell(\ux)$. 
			If $\ux = \emptyset$ then $E(\id)= \emptyset$, and the claim is immediate. 
			\par 
			Now fix $\ux$ where $\ell(\ux)>0$ and assume the claim holds for every reduced expression $\ux'$ where $\ell(\ux')< \ell(\ux)$. 
			Let $x = \{t_1 > \dots > t_l > i+1 \}$, Equation \ref{eqn: action on sets} again implies 
			\begin{align*}
				xs_i =
				\begin{cases}
					\{t_1 > \dots > t_l > i \} 
					&\text{ if } i \geq 1
					\\
					\{t_1 > \dots > t_l  \}
					&\text{ if } i=0.
				\end{cases}
			\end{align*} 
			Clearly $x> xs_i$ as the Bruhat order on ${}^I W$ agrees with the containment order on partitions. 
			Minimality implies $i+1, i \notin E(x)$ and $i+1,i \notin E(xs_i)$. 
			For any $E \subset E(xs_i)$ set $y_E=xs_i \backslash E$. 
			Then we have $y_E s_i = x\backslash E$, as $i$ is in both $y_E$ and $x \backslash E$, and $i+1$ is in neither $y_E$ nor $x \backslash E$.
			The induction hypothesis implies
			\begin{align*}
				d_{\ux} 
				=
				d_{\underline{xs}_i} b_{s_i}
				= 
				\sum_{E \subseteq E(xs_i)} d_{y_E} b_{s_i}
				=
				\sum_{E \subseteq E(xs_i)} d_{x \backslash E}
				+
				\sum_{E \subseteq E(xs_i)}\sum_{zs_i < z} \mu(z ,y_E) d_{z}
			\end{align*} 
			where $\mu(z , y_E)$ is the coefficient of $v$ in $n_{z, y_E}$. 
			By Proposition \ref{prop: kl basis} we know $\mu(z, y_E) = 1$ if $ \#(z, y_E)=1$ and $\mu(z, y_E) = 0$ otherwise.
			Hence
			\begin{align}
			\label{eqn: bs decomp}
				d_{\ux} 
				=
				\sum_{E \subseteq E(xs_i)} d_{x \backslash E}
				+
				\sum_{E \subseteq E(xs_i)}\sum_{\substack{zs_i < z \\ \#(z, y_E)=1 }} d_{z}.
			\end{align}
			The condition $zs_i < z$ is equivalent to either: $i+1 \in z$ and $i \notin z$ when $i>0$; or $1 \in z$ when $i=0$.  
			\par 
			Observe that if $i > 0$ the elements $i+1$ and $i$ are not $y_E$-related as $i+1 \notin y_E$ and it is necessary that the greater of the two elements be in $y_E$. 
			It follows that $\#(z, y_E) \geq 2$ for any $z$ satisfying $zs_i < z$.
			Consequently 
			\begin{align*}
				d_{\ux} 
				=
				\sum_{E \subseteq E(xs_i)} d_{x \backslash E}
				= 
				\sum_{E \subseteq E(x)} d_{x \backslash E}
			\end{align*}
			where the second equality follows from Lemma \ref{lem: E set}. 
			This proves the claim when $i>0$. 
			\par 
			If $i=0$, then $xs_0 = \{t_1 > \dots > t_l  \}$ where $t_l \geq 2$. 
			By Equation \ref{eqn: action on sets}, $zs_0 < z$ is equivalent to $1 \in z$.
			So, it follows that $\#(z,y_E) = 1$ if and only if there exists $t \in y_E$ such that $t \sim_{y_E} 1$ and $\phi_{t,1}^{y_E}(y_E)=z$. 
			Suppose such a $t$ exists. 
			Since $1 \notin y_E$, it must be that $t \sim_{y_E} 1$ in Step 1. 
			By definition $t$ must be the minimal element in $y_E$ such that $|\{ j \in y_E ~\vert~ t \geq j \geq 1 \}| = |\{j \notin y_E ~\vert~ t \geq j \geq 1\} |$. 
			Since $t$ is minimal, $t$ is strictly smaller than the minimal element of $E(xs_0)$ as $t' \in E(xs_0)$ requires $|\{ j \in xs_0 ~\vert~ t' \geq j \geq 1 \}| > |\{j \notin xs_0 ~\vert~ t' \geq j \geq 1\}|$.   
			Hence $t$ is determined by $xs_0$. 
			Consequently if $t$ exists for some $y_E = xs_0 \backslash E$, it exists for every $y_E = xs_0 \backslash E$; that is, it is independent of the subset $E$. 
			Further, $t \in E(x)$ because $x = xs_0 \cup \{ 1 \}$.
			Hence $E(x) = E(xs_0) \cup \{ t \}$ by Lemma \ref{lem: E set}. 
			Finally, note that $z = (y_E \backslash \{ t \}) \cup \{ 1\} = x \backslash (E \cup \{ t\})$ as $t \sim_{y_E} 1$ are related in Step 1.
			Hence
			\begin{align*}
				d_{\ux} 
				= 
				\sum_{E \subseteq E(xs_0)} d_{xs_0 \backslash E} b_{s_0}
				= 
				\sum_{E \subseteq E(xs_0)} (d_{x \backslash E} + d_{x \backslash (E \cup \{ t \})} )
				= 
				\sum_{E \subseteq E(x)} d_{x \backslash E} 
			\end{align*}
			proving the claim.
			\par 
			If no such $t$ exists, then the claim is immediate from Equation \ref{eqn: bs decomp} and Lemma \ref{lem: E set}. 	  
		\end{proof}
		
\subsection{The $2$-Kazhdan-Lusztig basis}
	
		We now determine the $p$-Kazhdan-Lusztig basis using local intersection forms. 
		We quickly recall the construction of local intersection forms from \cite{JW17, GJW23, EMTW20}.  
		\par 
		Let $\sD_{\text{RA}}^{\BS}(\h,I)$ denote the full subcategory of Bott-Samelson objects in $\DRA(\h,I)$. 
		Define  
		\begin{align*}
			\sD^{\not<y}_{\text{RA}} (\h , I) 
			:= 
			\sD_{\text{RA}}^{\BS}(\h,I) ~/~ \langle B_{\uz} \vert z < y \rangle_{(1), \oplus}. 
		\end{align*}
		for each $y \in W$.
		The category is enriched over $\Z$-graded $R$-modules. 
		We denote by  $\Hom_{\not < y}^{\bullet}(B,B')$ the graded Hom spaces and $\Hom_{\not < y}^{i}(B,B')$ the degree-$i$ morphisms. 
		Since  any non-identity endomorphisms of $B_{\uy}$ factor through some $B_{\uz}$ where $z<y$, we have $\Endo^{\bullet}_{\not < y}(B_{\uy}) \cong R$ and $\Endo^{0}_{\not < y}(B_{\uy}) \cong \Bbbk$.
		Hence, for any Bott-Samelson object $B_{\ux}$, we define the \ldef{degree-$i$ local intersection form} $I_{y, \ux}^{i}$ as:
		\[
		\begin{array}{ccccccC}
			I_{y, \ux}^{i}:
			&
			\Hom_{\not< y }^{i}(B_{\ux} , B_{\uy})
			\times
			\Hom^{-i}_{\not< y }(  B_{\uy}, B_{\ux})
			&
			\longrightarrow
			&
			\Endo_{\not< y }^{\bullet}(B_{\uy})
			&
			\longrightarrow
			&
			\Endo_{\not< y }^{0}(B_{\uy}) \cong \Bbbk
			\\\\
			&
			(f,g)
			&
			\longmapsto
			&
			f \circ g
			&
			\longmapsto
			&
			(f \circ g)_0
		\end{array}
		\]
		where $(f \circ g)_0$ denotes the degree-0 component of $f\circ g$. 
		The importance of the local intersection form stems from the following Lemma.
		\begin{lem}
		\label{Lem: IF Multiplicity General}
			Let $\Bbbk$ be a complete local ring, and $\bF$ the residue field.  
			For objects $B_y$ and $B_{\ux}$ in $\DRA(\h,I,\Bbbk)$, the graded multiplicity of $B_y$ in $B_{\ux}$ is 
			\begin{align*}
			\grrk I^{\bullet}_{y, \ux}
			=
			\sum_i  \rank (I_{y,\ux}^i \otimes \bF)  v^i. 
		\end{align*}  
		\end{lem}
		\begin{proof}
			This is \cite[\S 5]{GJW23} for $\DBE(\h)$, the idempotent lifting arguments carry over to $\DRA(\h,I)$.
		\end{proof}

		We can now determine the $p$-Kazhdan-Lusztig basis. 
		
		\begin{thm}
		\label{thm: indecomposable objects}
			Let $\h$ be a Cartan realisation type $\tB_n$ over a field or a complete local ring $\Bbbk$, and $I\subset S$ of type $\tA_{n-1}$. 
			The $p$-Kazhdan-Lusztig basis of $\sD_{\text{RA}}(\h,I, \Bbbk)$ is
			\begin{align*}
				{}^p d_x 
				= 
				\begin{cases}
					d_{\ux} &\text{if } 2 \notin \Bbbk^{\times}, \text{ and }
				\\
					d_x &\text{if } 2 \in \Bbbk^{\times},
				\end{cases}
			\end{align*}
			where $\ux$ is any reduced expression of $x$. 
		\end{thm}
		\begin{proof}
			For any expression $\uz$, not necessarily reduced, Lemma \ref{Lem: IF Multiplicity General} shows that after taking the diagrammatic character (see \cite[\S6]{LW22}) we have 
			\begin{align*}
				d_{\uz} = \sum_{w \in W} \grrk I^{\bullet}_{y, \ux}~ d_w.
			\end{align*}
			Hence it suffices to compute the local intersection forms only at those $y \leq x$ such that $d_y$ has non-zero coefficient in $d_{\ux}$. 
			By Proposition \ref{prop: bs decomp}, we need only check the local intersection forms $I_{x \backslash E, \ux}^{\bullet}$ for each $E \subseteq E(x)$. 
			Since the sum in Proposition \ref{prop: bs decomp} is multiplicity free and concentrated in degree 0, it follows that the $I_{x \backslash E, \ux}^{\bullet}$ is concentrated in degree 0, i.e. $I_{x \backslash E, \ux}^{\bullet} = I^0_{x \backslash E, \ux}$, and determined by $LL_{\ux, E} \circ \Gamma \Gamma^{\ux, E}$. 
			This composition was determined in Proposition \ref{prop: local intersection form}, which implies the local intersection form is the $1 \times 1$ matrix
			\begin{align*}
				I_{x \backslash E, \ux}^0 
				= 
				\left[2^{|E|} (-1)^{\ell(E)/2} \right].
			\end{align*} 
			Hence, if $E \neq \emptyset$ we have
			\begin{align*}
				\grrk \left( I_{x \backslash E, \ux}^{\bullet} \otimes \bF  \right) 
				=
				\begin{cases}
					0 
					&
					\text{ if } 2 \notin \Bbbk, \text{ and }
				\\
					1
					&
					\text{ if } 2 \in \Bbbk.
				\end{cases}
			\end{align*}
			Hence the claim.  
		\end{proof}

		\begin{rem}
			Suppose $\Bbbk$ is a complete local ring such that $2 \notin \Bbbk$ and let $\bK$ denote its fraction field.
			Recall $w_I w_0 = [n]$. 
			It is easily seen that $|E(w_I w_0)| = \lfloor n/2 \rfloor$ and $\ell(w_I w_0) = \binom{n+1}{2}$. 
			If we denote by $\bK(B_{w_I w_0})$ the object obtained by extending scalars, then Theorem \ref{thm: indecomposable objects} implies 
			\begin{align*}
				\bK(B_{w_I w_0}) \cong \bigoplus_{y \in E(w_I w_0)} {}^{\bK} B_{w_I w_0 \backslash E}
			\end{align*}
			which has $2^{\lfloor n/2 \rfloor}$ summands. 
			Varying $n$ produces the first family in finite Hecke categories of indecomposable objects $B$ whose number of summands under extension of scalars grows exponentially relative to their length.\footnote{
			Let $\Bbbk$ be a complete local ring with residue field  of characteristic $p$, and fraction field $\bK$. 
			If $\h$ is a realisation of type $\tI_2(m)$ over $\Bbbk$, then any $x \in W$ has the property $\bK(B_x)$ has $2^k$ indecomposable summands, where $k$ is the number of non-zero digits in the $p$-adic expansion of $\ell(x)$. 
			In particular, the number of summands grows (at most) linearly in the length of $x$. } 
		\end{rem}

		\begin{rem}
		\label{rem: minu p=2 intersection forms}
			The argument in Theorem \ref{thm: indecomposable objects} shows that in type $\tC_n$, for any $x \in {}^I W$ and  $E \subseteq E(x)$, the local intersection form  is 
			\begin{align*}
				I_{x \backslash E, \ux}^0 
				= 
				\left[(-1)^{\ell(E)/2} \right].
			\end{align*}
			Hence the $p$-Kazhdan-Lusztig basis for the odd orthogonal grassmannian is independent of $p$.  
		\end{rem}
		
		\begin{rem}
			The Bott-Samelson basis of ${}^I N$ in type $\tC_n$ was studied in \cite{BDFHN23}, where it was shown that when $d_{\ux}$ is expressed in terms of the standard basis each coefficient is a monic monomial. 
			Obviously their result extends to type $\tB_n$. 
			Combined with Theorem \ref{thm: indecomposable objects}, this shows all $p$-Kazhdan-Lusztig polynomials for $\DRA(\h,I)$ are monoic monomials.  
		\end{rem}

\subsection{Degree-0 endomorphism algebras}
\label{Ssec: Endomorphism algebras}
	
		We conclude by noting that the subalgebra $\Endo^0(B_x) \subset \Endo^{\bullet} (B_x)$ has a beautiful combinatorial description. 
		This allows us to explicitly determine the idempotents which project from $B_{\ux}$ onto $B_x$.
		In this section we fix $x \in {}^I W$ and define $\varphi_{E} := \Gamma \Gamma^{\ux, E} \circ LL_{\ux, E}$ for each $E \subseteq E(x)$. 
		We continue to assume $\Bbbk$ is a field or a complete local ring. 
		
		\begin{lem}
			  If $2 \notin \Bbbk^{\times}$, then the morphisms $\{ \varphi_{E} ~\vert~ E \subseteq E(x) \}$ are a basis of $\Endo^0(B_x)$. 
		\end{lem}
		\begin{proof}
			This is immediate from Theorem \ref{thm: indecomposable objects} and the fact that double-light-leaves are a basis of $B_{\ux}$, see \cite[\S5]{LW22}. 
		\end{proof}
		
		\begin{exmp}
		\label{Exmp: endomorphism basis}
			Let $x = \{ 6,5,4,3,2,1\}$. Then $E(x) = \{ 6,4,2\}$ and $\Endo^0(B_x)$ has a basis given by the following morphisms: 
			\[
			\begin{array}{rrrr}
			\varphi_{\emptyset}=
			\begin{array}{c}
			\begin{tikzpicture}[scale=0.35]
      			\draw[color=red, thick] (0,0) to (0,6);
      			\draw[color=blue, thick] (0.2,0) to (0.2,6);
      			\draw[color=red!60!yellow, thick] (0.4,0) to (0.4,6);
      			\draw[color=red, thick] (0.6,0) to (0.6,6);
      			\draw[color=blue!50!cyan, thick] (0.8,0) to (0.8,6);
      			\draw[color=blue, thick] (1,0) to (1,6);
      			\draw[color=red!20!yellow, thick] (1.2,0) to (1.2,6);
      			\draw[color=red!60!yellow, thick] (1.4,0) to (1.4,6);
      			\draw[color=red, thick] (1.6,0) to (1.6,6);
      			\draw[color=cyan, thick] (1.9,0) to (1.9,6);
      			\draw[color=blue!50!cyan, thick] (2.2,0) to (2.2,6);
      			\draw[color=blue, thick] (2.5,0) to (2.5,6);
				\draw[color=red!20!yellow, thick] (2.8,0) to (2.8,6);
      			\draw[color=red!60!yellow, thick] (3,0) to (3,6);
      			\draw[color=red, thick] (3.2,0) to (3.2,6);	
      			\draw[color=blue!50!cyan, thick] (3.5,0) to (3.5,6);
      			\draw[color=blue, thick] (3.8,0) to (3.8,6);
      			\draw[color=red!60!yellow, thick] (4.2,0) to (4.2,6);
      			\draw[color=red, thick] (4.4,0) to (4.4,6);
				\draw[color=blue, thick] (4.7,0) to (4.7,6);
				\draw[color=red, thick] (5,0) to (5,6);
    		\end{tikzpicture}
    		\end{array},
    		&
			\varphi_{\{2\}}=
			\begin{array}{c}
			\begin{tikzpicture}[scale=0.35]
      			\draw[color=red, thick] (0,0) to (0,6);
      			\draw[color=blue, thick] (0.2,0) to (0.2,6);
      			\draw[color=red!60!yellow, thick] (0.4,0) to (0.4,6);
      			\draw[color=red, thick] (0.6,0) to (0.6,6);
      			\draw[color=blue!50!cyan, thick] (0.8,0) to (0.8,6);
      			\draw[color=blue, thick] (1,0) to (1,6);
      			\draw[color=red!20!yellow, thick] (1.2,0) to (1.2,6);
      			\draw[color=red!60!yellow, thick] (1.4,0) to (1.4,6);
      			\draw[color=red, thick] (1.6,0) to (1.6,6);
      			\draw[color=cyan, thick] (1.9,0) to (1.9,6);
      			\draw[color=blue!50!cyan, thick] (2.2,0) to (2.2,6);
      			\draw[color=blue, thick] (2.5,0) to (2.5,6);
				\draw[color=red!20!yellow, thick] (2.8,0) to (2.8,6);
      			\draw[color=red!60!yellow, thick] (3,0) to (3,6);
      			\draw[color=red, thick] (3.2,0) to (3.2,6);	
      			\draw[color=blue!50!cyan, thick] (3.5,0) to (3.5,6);
      			\draw[color=blue, thick] (3.8,0) to (3.8,6);
      			\draw[color=red!60!yellow, thick] (4.2,0) to (4.2,6);
      			\draw[color=red, thick] (4.4,0) to (4.4,6);
				\draw[color=blue, thick] (4.7,0) to (4.7,0.3);
				\draw[color=blue, thick] (4.7,6) to (4.7,5.7);
				\node[circle,fill,draw,inner sep=0mm,minimum size=1mm,color=blue, thick] at (4.7,0.3) {};
				\node[circle,fill,draw,inner sep=0mm,minimum size=1mm,color=blue, thick] at (4.7,5.7) {};
				\draw[color=red, thick] (5,0) arc [start angle=0, end angle = 90, x radius = 0.6cm , y radius = 0.8cm];
				\draw[color=red, thick] (5,6) arc [start angle=0, end angle = -90, x radius = 0.6cm , y radius = 0.8cm];
    		\end{tikzpicture}
    		\end{array},
    		&
			\varphi_{\{4\}}=
			\begin{array}{c}
			\begin{tikzpicture}[scale=0.35]
      			\draw[color=red, thick] (0,0) to (0,6);
      			\draw[color=blue, thick] (0.2,0) to (0.2,6);
      			\draw[color=red!60!yellow, thick] (0.4,0) to (0.4,6);
      			\draw[color=red, thick] (0.6,0) to (0.6,6);
      			\draw[color=blue!50!cyan, thick] (0.8,0) to (0.8,6);
      			\draw[color=blue, thick] (1,0) to (1,6);
      			\draw[color=red!20!yellow, thick] (1.2,0) to (1.2,6);
      			\draw[color=red!60!yellow, thick] (1.4,0) to (1.4,6);
      			\draw[color=red, thick] (1.6,0) to (1.6,6);
      			\draw[color=cyan, thick] (1.9,0) to (1.9,6);
      			\draw[color=blue!50!cyan, thick] (2.2,0) to (2.2,6);
      			\draw[color=blue, thick] (2.5,0) to (2.5,6);
				\draw[color=red!20!yellow, thick] (2.8,0) to (2.8,6);
      			\draw[color=red!60!yellow, thick] (3,0) to (3,6);
      			\draw[color=red, thick] (3.2,0) to (3.2,6);
      			\draw[color=blue!50!cyan, thick] (3.5,0) to (3.5,0.4);
      			\draw[color=blue, thick] (3.8,0) to (3.8,0.4);
      			\draw[color=blue!50!cyan, thick] (3.5,6) to (3.5,5.6);
      			\draw[color=blue, thick] (3.8,6) to (3.8,5.6);
				\node[circle,fill,draw,inner sep=0mm,minimum size=1mm,color=blue!50!cyan, thick] at (3.5,0.4) {};
				\node[circle,fill,draw,inner sep=0mm,minimum size=1mm,color=blue, thick] at (3.8,0.4) {};
				\node[circle,fill,draw,inner sep=0mm,minimum size=1mm,color=blue!50!cyan, thick] at (3.5,5.6) {};
				\node[circle,fill,draw,inner sep=0mm,minimum size=1mm,color=blue, thick] at (3.8,5.6) {};
      			\draw[color=red!60!yellow, thick] (4.2,0) arc [start angle=0, end angle = 89, x radius = 1.2cm , y radius = 1.4cm];
      			\draw[color=red!60!yellow, thick] (4.2,6) arc [start angle=0, end angle = -89, x radius = 1.2cm , y radius = 1.4cm];
				\draw[color=red, thick] (4.4,0) arc [start angle=0, end angle = 90, x radius = 1.2cm , y radius = 1.6cm];
				\draw[color=red, thick] (4.4,6) arc [start angle=0, end angle = -90, x radius = 1.2cm , y radius = 1.6cm];
				\draw[color=blue, thick] (4.7,0) to (4.7,6);
				\draw[color=red, thick] (5,0) to (5,6);
    		\end{tikzpicture}
    		\end{array},
    		&
			\varphi_{\{4,2\}}=
			\begin{array}{c}
			\begin{tikzpicture}[scale=0.35]
      			\draw[color=red, thick] (0,0) to (0,6);
      			\draw[color=blue, thick] (0.2,0) to (0.2,6);
      			\draw[color=red!60!yellow, thick] (0.4,0) to (0.4,6);
      			\draw[color=red, thick] (0.6,0) to (0.6,6);
      			\draw[color=blue!50!cyan, thick] (0.8,0) to (0.8,6);
      			\draw[color=blue, thick] (1,0) to (1,6);
      			\draw[color=red!20!yellow, thick] (1.2,0) to (1.2,6);
      			\draw[color=red!60!yellow, thick] (1.4,0) to (1.4,6);
      			\draw[color=red, thick] (1.6,0) to (1.6,6);
      			\draw[color=cyan, thick] (1.9,0) to (1.9,6);
      			\draw[color=blue!50!cyan, thick] (2.2,0) to (2.2,6);
      			\draw[color=blue, thick] (2.5,0) to (2.5,6);
				\draw[color=red!20!yellow, thick] (2.8,0) to (2.8,6);
      			\draw[color=red!60!yellow, thick] (3,0) to (3,6);
      			\draw[color=red, thick] (3.2,0) to (3.2,6);
      			\draw[color=blue!50!cyan, thick] (3.5,0) to (3.5,0.4);
      			\draw[color=blue, thick] (3.8,0) to (3.8,0.4);
      			\draw[color=blue!50!cyan, thick] (3.5,6) to (3.5,5.6);
      			\draw[color=blue, thick] (3.8,6) to (3.8,5.6);
				\node[circle,fill,draw,inner sep=0mm,minimum size=1mm,color=blue!50!cyan, thick] at (3.5,0.4) {};
				\node[circle,fill,draw,inner sep=0mm,minimum size=1mm,color=blue, thick] at (3.8,0.4) {};
				\node[circle,fill,draw,inner sep=0mm,minimum size=1mm,color=blue!50!cyan, thick] at (3.5,5.6) {};
				\node[circle,fill,draw,inner sep=0mm,minimum size=1mm,color=blue, thick] at (3.8,5.6) {};
      			\draw[color=red!60!yellow, thick] (4.2,0) arc [start angle=0, end angle = 89, x radius = 1.2cm , y radius = 1.4cm];
      			\draw[color=red!60!yellow, thick] (4.2,6) arc [start angle=0, end angle = -89, x radius = 1.2cm , y radius = 1.4cm];
				\draw[color=red, thick] (4.4,0) arc [start angle=0, end angle = 90, x radius = 1.2cm , y radius = 1.6cm];
				\draw[color=red, thick] (4.4,6) arc [start angle=0, end angle = -90, x radius = 1.2cm , y radius = 1.6cm];
				\draw[color=blue, thick] (4.7,0) to (4.7,0.3);
				\draw[color=blue, thick] (4.7,6) to (4.7,5.7);
				\node[circle,fill,draw,inner sep=0mm,minimum size=1mm,color=blue, thick] at (4.7,0.3) {};
				\node[circle,fill,draw,inner sep=0mm,minimum size=1mm,color=blue, thick] at (4.7,5.7) {};
				\draw[color=red, thick] (5,0) arc [start angle=0, end angle = 90, x radius = 0.8cm , y radius = 0.8cm];
				\draw[color=red, thick] (5,6) arc [start angle=0, end angle = -90, x radius = 0.8cm , y radius = 0.8cm];
    		\end{tikzpicture}
    		\end{array},
    		\\
    		\\
			\varphi_{\{6\}}=
			\begin{array}{c}
			\begin{tikzpicture}[scale=0.35]
      			\draw[color=red, thick] (0,0) to (0,6);
      			\draw[color=blue, thick] (0.2,0) to (0.2,6);
      			\draw[color=red!60!yellow, thick] (0.4,0) to (0.4,6);
      			\draw[color=red, thick] (0.6,0) to (0.6,6);
      			\draw[color=blue!50!cyan, thick] (0.8,0) to (0.8,6);
      			\draw[color=blue, thick] (1,0) to (1,6);
      			\draw[color=red!20!yellow, thick] (1.2,0) to (1.2,6);
      			\draw[color=red!60!yellow, thick] (1.4,0) to (1.4,6);
      			\draw[color=red, thick] (1.6,0) to (1.6,6);
      			\draw[color=cyan, thick] (1.9,0) to (1.9,0.8);
      			\draw[color=blue!50!cyan, thick] (2.2,0) to (2.2,0.8);
      			\draw[color=blue, thick] (2.5,0) to (2.5,0.8);
      			\draw[color=cyan, thick] (1.9,6) to (1.9,5.2);
      			\draw[color=blue!50!cyan, thick] (2.2,6) to (2.2,5.2);
      			\draw[color=blue, thick] (2.5,6) to (2.5,5.2);
				\node[circle,fill,draw,inner sep=0mm,minimum size=1mm,color=cyan, thick] at (1.9,0.8) {};
				\node[circle,fill,draw,inner sep=0mm,minimum size=1mm,color=blue!50!cyan, thick] at (2.2,0.8) {};
				\node[circle,fill,draw,inner sep=0mm,minimum size=1mm,color=blue, thick] at (2.5,0.8) {};
				\node[circle,fill,draw,inner sep=0mm,minimum size=1mm,color=cyan, thick] at (1.9,5.2) {};
				\node[circle,fill,draw,inner sep=0mm,minimum size=1mm,color=blue!50!cyan, thick] at (2.2,5.2) {};
				\node[circle,fill,draw,inner sep=0mm,minimum size=1mm,color=blue, thick] at (2.5,5.2) {};
				\draw[color=red!20!yellow, thick] (2.8,0) arc [start angle=0, end angle = 90, x radius = 1.6cm , y radius = 2cm];
				\draw[color=red!20!yellow, thick] (2.8,6) arc [start angle=0, end angle = -90, x radius = 1.6cm , y radius = 2cm];
				\draw[color=red!60!yellow, thick] (3,0) arc [start angle=0, end angle = 90, x radius = 1.6cm , y radius = 2.2cm];
				\draw[color=red!60!yellow, thick] (3,6) arc [start angle=0, end angle = -90, x radius = 1.6cm , y radius = 2.2cm];
				\draw[color=red, thick] (3.2,0) arc [start angle=0, end angle = 90, x radius = 1.6cm , y radius = 2.4cm];	
				\draw[color=red, thick] (3.2,6) arc [start angle=0, end angle = -90, x radius = 1.6cm , y radius = 2.4cm];	
      			\draw[color=blue!50!cyan, thick] (3.5,0) to (3.5,6);
      			\draw[color=blue, thick] (3.8,0) to (3.8,6);
      			\draw[color=red!60!yellow, thick] (4.2,0) to (4.2,6);
      			\draw[color=red, thick] (4.4,0) to (4.4,6);
				\draw[color=blue, thick] (4.7,0) to (4.7,6);
				\draw[color=red, thick] (5,0) to (5,6);
    		\end{tikzpicture}
    		\end{array},
    		&
			\varphi_{\{6,2\}}=
			\begin{array}{c}
			\begin{tikzpicture}[scale=0.35]
      			\draw[color=red, thick] (0,0) to (0,6);
      			\draw[color=blue, thick] (0.2,0) to (0.2,6);
      			\draw[color=red!60!yellow, thick] (0.4,0) to (0.4,6);
      			\draw[color=red, thick] (0.6,0) to (0.6,6);
      			\draw[color=blue!50!cyan, thick] (0.8,0) to (0.8,6);
      			\draw[color=blue, thick] (1,0) to (1,6);
      			\draw[color=red!20!yellow, thick] (1.2,0) to (1.2,6);
      			\draw[color=red!60!yellow, thick] (1.4,0) to (1.4,6);
      			\draw[color=red, thick] (1.6,0) to (1.6,6);
      			\draw[color=cyan, thick] (1.9,0) to (1.9,0.8);
      			\draw[color=blue!50!cyan, thick] (2.2,0) to (2.2,0.8);
      			\draw[color=blue, thick] (2.5,0) to (2.5,0.8);
      			\draw[color=cyan, thick] (1.9,6) to (1.9,5.2);
      			\draw[color=blue!50!cyan, thick] (2.2,6) to (2.2,5.2);
      			\draw[color=blue, thick] (2.5,6) to (2.5,5.2);
				\node[circle,fill,draw,inner sep=0mm,minimum size=1mm,color=cyan, thick] at (1.9,0.8) {};
				\node[circle,fill,draw,inner sep=0mm,minimum size=1mm,color=blue!50!cyan, thick] at (2.2,0.8) {};
				\node[circle,fill,draw,inner sep=0mm,minimum size=1mm,color=blue, thick] at (2.5,0.8) {};
				\node[circle,fill,draw,inner sep=0mm,minimum size=1mm,color=cyan, thick] at (1.9,5.2) {};
				\node[circle,fill,draw,inner sep=0mm,minimum size=1mm,color=blue!50!cyan, thick] at (2.2,5.2) {};
				\node[circle,fill,draw,inner sep=0mm,minimum size=1mm,color=blue, thick] at (2.5,5.2) {};
				\draw[color=red!20!yellow, thick] (2.8,0) arc [start angle=0, end angle = 90, x radius = 1.6cm , y radius = 2cm];
				\draw[color=red!20!yellow, thick] (2.8,6) arc [start angle=0, end angle = -90, x radius = 1.6cm , y radius = 2cm];
				\draw[color=red!60!yellow, thick] (3,0) arc [start angle=0, end angle = 90, x radius = 1.6cm , y radius = 2.2cm];
				\draw[color=red!60!yellow, thick] (3,6) arc [start angle=0, end angle = -90, x radius = 1.6cm , y radius = 2.2cm];
				\draw[color=red, thick] (3.2,0) arc [start angle=0, end angle = 90, x radius = 1.6cm , y radius = 2.4cm];	
				\draw[color=red, thick] (3.2,6) arc [start angle=0, end angle = -90, x radius = 1.6cm , y radius = 2.4cm];	
      			\draw[color=blue!50!cyan, thick] (3.5,0) to (3.5,6);
      			\draw[color=blue, thick] (3.8,0) to (3.8,6);
      			\draw[color=red!60!yellow, thick] (4.2,0) to (4.2,6);
      			\draw[color=red, thick] (4.4,0) to (4.4,6);
				\draw[color=blue, thick] (4.7,0) to (4.7,0.3);
				\draw[color=blue, thick] (4.7,6) to (4.7,5.7);
				\node[circle,fill,draw,inner sep=0mm,minimum size=1mm,color=blue, thick] at (4.7,0.3) {};
				\node[circle,fill,draw,inner sep=0mm,minimum size=1mm,color=blue, thick] at (4.7,5.7) {};
				\draw[color=red, thick] (5,0) arc [start angle=0, end angle = 90, x radius = 0.6cm , y radius = 0.8cm];
				\draw[color=red, thick] (5,6) arc [start angle=0, end angle = -90, x radius = 0.6cm , y radius = 0.8cm];
    		\end{tikzpicture}
    		\end{array},
    		&
			\varphi_{\{6,4\}}=
			\begin{array}{c}
			\begin{tikzpicture}[scale=0.35]
      			\draw[color=red, thick] (0,0) to (0,6);
      			\draw[color=blue, thick] (0.2,0) to (0.2,6);
      			\draw[color=red!60!yellow, thick] (0.4,0) to (0.4,6);
      			\draw[color=red, thick] (0.6,0) to (0.6,6);
      			\draw[color=blue!50!cyan, thick] (0.8,0) to (0.8,6);
      			\draw[color=blue, thick] (1,0) to (1,6);
      			\draw[color=red!20!yellow, thick] (1.2,0) to (1.2,6);
      			\draw[color=red!60!yellow, thick] (1.4,0) to (1.4,6);
      			\draw[color=red, thick] (1.6,0) to (1.6,6);
      			\draw[color=cyan, thick] (1.9,0) to (1.9,0.8);
      			\draw[color=blue!50!cyan, thick] (2.2,0) to (2.2,0.8);
      			\draw[color=blue, thick] (2.5,0) to (2.5,0.8);
      			\draw[color=cyan, thick] (1.9,6) to (1.9,5.2);
      			\draw[color=blue!50!cyan, thick] (2.2,6) to (2.2,5.2);
      			\draw[color=blue, thick] (2.5,6) to (2.5,5.2);
				\node[circle,fill,draw,inner sep=0mm,minimum size=1mm,color=cyan, thick] at (1.9,0.8) {};
				\node[circle,fill,draw,inner sep=0mm,minimum size=1mm,color=blue!50!cyan, thick] at (2.2,0.8) {};
				\node[circle,fill,draw,inner sep=0mm,minimum size=1mm,color=blue, thick] at (2.5,0.8) {};
				\node[circle,fill,draw,inner sep=0mm,minimum size=1mm,color=cyan, thick] at (1.9,5.2) {};
				\node[circle,fill,draw,inner sep=0mm,minimum size=1mm,color=blue!50!cyan, thick] at (2.2,5.2) {};
				\node[circle,fill,draw,inner sep=0mm,minimum size=1mm,color=blue, thick] at (2.5,5.2) {};
				\draw[color=red!20!yellow, thick] (2.8,0) arc [start angle=0, end angle = 90, x radius = 1.6cm , y radius = 2cm];
				\draw[color=red!20!yellow, thick] (2.8,6) arc [start angle=0, end angle = -90, x radius = 1.6cm , y radius = 2cm];
				\draw[color=red!60!yellow, thick] (3,0) arc [start angle=0, end angle = 90, x radius = 1.6cm , y radius = 2.2cm];
				\draw[color=red!60!yellow, thick] (3,6) arc [start angle=0, end angle = -90, x radius = 1.6cm , y radius = 2.2cm];
				\draw[color=red, thick] (3.2,0) arc [start angle=0, end angle = 90, x radius = 1.6cm , y radius = 2.4cm];	
				\draw[color=red, thick] (3.2,6) arc [start angle=0, end angle = -90, x radius = 1.6cm , y radius = 2.4cm];	
      			\draw[color=blue!50!cyan, thick] (3.5,0) to (3.5,0.4);
      			\draw[color=blue, thick] (3.8,0) to (3.8,0.4);
      			\draw[color=blue!50!cyan, thick] (3.5,6) to (3.5,5.6);
      			\draw[color=blue, thick] (3.8,6) to (3.8,5.6);
				\node[circle,fill,draw,inner sep=0mm,minimum size=1mm,color=blue!50!cyan, thick] at (3.5,0.4) {};
				\node[circle,fill,draw,inner sep=0mm,minimum size=1mm,color=blue, thick] at (3.8,0.4) {};
				\node[circle,fill,draw,inner sep=0mm,minimum size=1mm,color=blue!50!cyan, thick] at (3.5,5.6) {};
				\node[circle,fill,draw,inner sep=0mm,minimum size=1mm,color=blue, thick] at (3.8,5.6) {};
      			\draw[color=red!60!yellow, thick] (4.2,0) arc [start angle=0, end angle = 89, x radius = 1.6cm , y radius = 1.4cm];
      			\draw[color=red!60!yellow, thick] (4.2,6) arc [start angle=0, end angle = -89, x radius = 1.6cm , y radius = 1.4cm];
				\draw[color=red, thick] (4.4,0) arc [start angle=0, end angle = 90, x radius = 1.6cm , y radius = 1.6cm];
				\draw[color=red, thick] (4.4,6) arc [start angle=0, end angle = -90, x radius = 1.6cm , y radius = 1.6cm];
				\draw[color=blue, thick] (4.7,0) to (4.7,6);
				\draw[color=red, thick] (5,0) to (5,6);
    		\end{tikzpicture}
    		\end{array},
    		&
			\varphi_{\{6,4,2\}}=
			\begin{array}{c}
			\begin{tikzpicture}[scale=0.35]
      			\draw[color=red, thick] (0,0) to (0,6);
      			\draw[color=blue, thick] (0.2,0) to (0.2,6);
      			\draw[color=red!60!yellow, thick] (0.4,0) to (0.4,6);
      			\draw[color=red, thick] (0.6,0) to (0.6,6);
      			\draw[color=blue!50!cyan, thick] (0.8,0) to (0.8,6);
      			\draw[color=blue, thick] (1,0) to (1,6);
      			\draw[color=red!20!yellow, thick] (1.2,0) to (1.2,6);
      			\draw[color=red!60!yellow, thick] (1.4,0) to (1.4,6);
      			\draw[color=red, thick] (1.6,0) to (1.6,6);
      			\draw[color=cyan, thick] (1.9,0) to (1.9,0.8);
      			\draw[color=blue!50!cyan, thick] (2.2,0) to (2.2,0.8);
      			\draw[color=blue, thick] (2.5,0) to (2.5,0.8);
      			\draw[color=cyan, thick] (1.9,6) to (1.9,5.2);
      			\draw[color=blue!50!cyan, thick] (2.2,6) to (2.2,5.2);
      			\draw[color=blue, thick] (2.5,6) to (2.5,5.2);
				\node[circle,fill,draw,inner sep=0mm,minimum size=1mm,color=cyan, thick] at (1.9,0.8) {};
				\node[circle,fill,draw,inner sep=0mm,minimum size=1mm,color=blue!50!cyan, thick] at (2.2,0.8) {};
				\node[circle,fill,draw,inner sep=0mm,minimum size=1mm,color=blue, thick] at (2.5,0.8) {};
				\node[circle,fill,draw,inner sep=0mm,minimum size=1mm,color=cyan, thick] at (1.9,5.2) {};
				\node[circle,fill,draw,inner sep=0mm,minimum size=1mm,color=blue!50!cyan, thick] at (2.2,5.2) {};
				\node[circle,fill,draw,inner sep=0mm,minimum size=1mm,color=blue, thick] at (2.5,5.2) {};
				\draw[color=red!20!yellow, thick] (2.8,0) arc [start angle=0, end angle = 90, x radius = 1.6cm , y radius = 2cm];
				\draw[color=red!20!yellow, thick] (2.8,6) arc [start angle=0, end angle = -90, x radius = 1.6cm , y radius = 2cm];
				\draw[color=red!60!yellow, thick] (3,0) arc [start angle=0, end angle = 90, x radius = 1.6cm , y radius = 2.2cm];
				\draw[color=red!60!yellow, thick] (3,6) arc [start angle=0, end angle = -90, x radius = 1.6cm , y radius = 2.2cm];
				\draw[color=red, thick] (3.2,0) arc [start angle=0, end angle = 90, x radius = 1.6cm , y radius = 2.4cm];	
				\draw[color=red, thick] (3.2,6) arc [start angle=0, end angle = -90, x radius = 1.6cm , y radius = 2.4cm];	
      			\draw[color=blue!50!cyan, thick] (3.5,0) to (3.5,0.4);
      			\draw[color=blue, thick] (3.8,0) to (3.8,0.4);
      			\draw[color=blue!50!cyan, thick] (3.5,6) to (3.5,5.6);
      			\draw[color=blue, thick] (3.8,6) to (3.8,5.6);
				\node[circle,fill,draw,inner sep=0mm,minimum size=1mm,color=blue!50!cyan, thick] at (3.5,0.4) {};
				\node[circle,fill,draw,inner sep=0mm,minimum size=1mm,color=blue, thick] at (3.8,0.4) {};
				\node[circle,fill,draw,inner sep=0mm,minimum size=1mm,color=blue!50!cyan, thick] at (3.5,5.6) {};
				\node[circle,fill,draw,inner sep=0mm,minimum size=1mm,color=blue, thick] at (3.8,5.6) {};
      			\draw[color=red!60!yellow, thick] (4.2,0) arc [start angle=0, end angle = 89, x radius = 1.6cm , y radius = 1.4cm];
      			\draw[color=red!60!yellow, thick] (4.2,6) arc [start angle=0, end angle = -89, x radius = 1.6cm , y radius = 1.4cm];
				\draw[color=red, thick] (4.4,0) arc [start angle=0, end angle = 90, x radius = 1.6cm , y radius = 1.6cm];
				\draw[color=red, thick] (4.4,6) arc [start angle=0, end angle = -90, x radius = 1.6cm , y radius = 1.6cm];
				\draw[color=blue, thick] (4.7,0) to (4.7,0.3);
				\draw[color=blue, thick] (4.7,6) to (4.7,5.7);
				\node[circle,fill,draw,inner sep=0mm,minimum size=1mm,color=blue, thick] at (4.7,0.3) {};
				\node[circle,fill,draw,inner sep=0mm,minimum size=1mm,color=blue, thick] at (4.7,5.7) {};
				\draw[color=red, thick] (5,0) arc [start angle=0, end angle = 90, x radius = 0.8cm , y radius = 0.8cm];
				\draw[color=red, thick] (5,6) arc [start angle=0, end angle = -90, x radius = 0.8cm , y radius = 0.8cm];
    		\end{tikzpicture}
    		\end{array}.
    		\end{array}
    		\]
		\end{exmp}

		\begin{lem}
		\label{Lem: Gelfand}
			For any $\Bbbk$ and any $x \in {}^I W$ the algebra $\Endo^0(B_x)$ is commutative. 
		\end{lem}
		The proof is a simple application of Gel'fand's trick. 
		\begin{proof}
			The algebra $\Endo^0(B_x)$ has an antiinvolution $\bD: \Endo^0(B_x) \rightarrow \Endo^0(B_x)$ given by reflecting diagrams about their horizontal midpoint. 
			If $2 \notin \Bbbk^{\times}$, denote the structure constants for the basis $\{ \varphi_{E} \}$ by $\{ c_{E,E'}^{E''} \}$.
			Since each $\varphi_E$ is symmetric about its horizontal midpoint we see
			\begin{align*}
				\varphi_{E'} \varphi_E
				=
				\bD (\varphi_E \varphi_{E'})  
				= 
				\bD \left(\sum_{E''} c^{E''}_{E,E'} \varphi_{E''} \right)
				=
				\sum_{E''} c^{E''}_{E,E'} \bD(\varphi_{E''}) 
				=
				\sum_{E''} c^{E''}_{E,E'} \varphi_{E''}
				=
				\varphi_E \varphi_{E'}.
			\end{align*}
			Hence $\Endo^0(B_x)$ is commutative. 
			If $2 \in \Bbbk^{\times}$, then $\Endo^0(B_x) \cong \Bbbk$ and the claim is trivial. 
		\end{proof}
		
		\begin{rem}
			We note the following generalisations of Lemma \ref{Lem: Gelfand}:
		\begin{enumerate}
			\item[(a)]  	Let $\Bbbk$ be a complete local ring with fraction field $\bK$. 
					The preceding argument shows for any Hecke category with double-light-leaf basis, if $\bK(B)$ has a multiplicity-free decomposition then $\Endo^0(B)$ is commutative.
			\item[(b)] If we fix $x$ and require every $p$-Kazhdan-Lusztig polynomial ${}^p h_{y,x}$, ${}^p m_{y,x}$, or ${}^p n_{y,x}$ be a monic monomial (as we vary $y$), then the graded endomorphism algebra $\Endo^{\bullet}(B_x)$ is commutative. 
		\end{enumerate}
			 
		\end{rem}

		We can now explicitly describe $\Endo^0(B_x)$ for each $x$. 
		
		\begin{prop}
		\label{prop: alg presentation}
			Suppose $2 \notin \Bbbk^{\times}$. The algebra $\Endo^0(B_x)$ has a presentation 
			\begin{align*}
				\Endo^0(B_x) 
				\cong 
				\frac{\Bbbk[\varphi_E \vert E \subseteq E(x)]}
				{\left(\varphi_E \varphi_{E'} - 2^{|E \cap E'|} (-1)^{\ell(E \cap E')/2 } \varphi_{E \cup E'} \right)}.
			\end{align*} 
		\end{prop}
		
		\begin{proof}
			Fix $x \in {}^I W$. 
			If $E(x) = \emptyset$ then $\Endo^0 (B_x) \cong \Bbbk$ so the claim is vacuously true. 
			Assume that $E(x) \neq \emptyset$. 
			Proposition \ref{prop: local intersection form} implies  $\varphi_{\{ t \}}^2 = 2(-1)^{t/2} \varphi_{\{ t \}}$ any $t \in E(x)$.
			Hence the claim follows from proving the identity
			\begin{align*}
				\varphi_E = \prod_{t \in E} \varphi_{\{t \}}
			\end{align*} 
			holds for any $E \subseteq E(x)$. 
			If $E = \emptyset$ this is vacuously true, so assume $E \neq \emptyset$. 
			Let $\ux$ be the column reduced expression of $x$. 
			Recall $E = \{ t_1 > \dots > t_{k-1} > t_k \}$ gives rise to a decomposition
			\begin{align*}
				\ux 
				= 
				\ur_1 \uu_{t_1} \ur_2  \dots \uu_{t_{k-1}} \ur_k \uu_{t_k} \ur_{k+1}. 
			\end{align*}
			Now let $E' = \{ t_1 > \dots > t_{k-1} \}$, i.e. omit the final term. 
			There are two cases to consider. 
			\par 
			\textit{Case 1} ($\ur_k \neq \emptyset$): 
			Note that $\varphi_{E'} = f \otimes \id_{\ur_{k} \uu_{t_k} \ur_{k+1}}$ where $f$ is the double-light-leaf obtained by truncating the subexpression $\ux^E$ to only include the terms $\ur_1 \uu_{t_1} \ur_2  \dots \uu_{t_{k-1}}$. 
			Since $\ur_k \neq \emptyset$, we also have $\varphi_{\{ t_k \}} = \id_{\ur_1 \uu_{t_1} \ur_2  \dots \uu_{t_{k-1}}} \otimes g$ where $g$ is the double light leaf obtained from the subexpression $\ur_k^1 \uu_{t_k}^0 \ur_{k+1}^1$. 
			Note $g$ is well-defined as the trivalent vertices associated to the even reflections in $\uu_{t_k}^0$ are formed with identity strands in $\ur_k^1$. 
			Hence $\varphi_{E'} \varphi_{\{ t_k \}} =  f \otimes g = \varphi_E$. 
			As an illustrative example: 
			\[ 
			\begin{array}{c}
			\begin{tikzpicture}[scale=0.5]
			%
			%
      			\draw[color=red, thick] (0,0) to (0,6);
      			\draw[color=blue, thick] (0.2,0) to (0.2,6);
      			\draw[color=red!60!yellow, thick] (0.4,0) to (0.4,6);
      			\draw[color=red, thick] (0.6,0) to (0.6,6);
      			\draw[color=blue!50!cyan, thick] (0.8,0) to (0.8,6);
      			\draw[color=blue, thick] (1,0) to (1,6);
      			\draw[color=red!20!yellow, thick] (1.2,0) to (1.2,6);
      			\draw[color=red!60!yellow, thick] (1.4,0) to (1.4,6);
      			\draw[color=red, thick] (1.6,0) to (1.6,6);
      			\draw[color=cyan, thick] (1.9,2) to (1.9,2.5);
      			\draw[color=blue!50!cyan, thick] (2.2,2) to (2.2,2.5);
      			\draw[color=blue, thick] (2.5,2) to (2.5,2.5);
      			\draw[color=cyan, thick] (1.9,6) to (1.9,5.5);
      			\draw[color=blue!50!cyan, thick] (2.2,6) to (2.2,5.5);
      			\draw[color=blue, thick] (2.5,6) to (2.5,5.5);
				\node[circle,fill,draw,inner sep=0mm,minimum size=1mm,color=cyan, thick] at (1.9,2.4) {};
				\node[circle,fill,draw,inner sep=0mm,minimum size=1mm,color=blue!50!cyan, thick] at (2.2,2.4) {};
				\node[circle,fill,draw,inner sep=0mm,minimum size=1mm,color=blue, thick] at (2.5,2.4) {};
				\node[circle,fill,draw,inner sep=0mm,minimum size=1mm,color=cyan, thick] at (1.9,5.6) {};
				\node[circle,fill,draw,inner sep=0mm,minimum size=1mm,color=blue!50!cyan, thick] at (2.2,5.6) {};
				\node[circle,fill,draw,inner sep=0mm,minimum size=1mm,color=blue, thick] at (2.5,5.6) {};
				\draw[color=red!20!yellow, thick] (2.8,2) arc [start angle=0, end angle = 90, x radius = 1.6cm , y radius = 1.4cm];
				\draw[color=red!20!yellow, thick] (2.8,6) arc [start angle=0, end angle = -90, x radius = 1.6cm , y radius = 1.4cm];
				\draw[color=red!60!yellow, thick] (3,2) arc [start angle=0, end angle = 90, x radius = 1.6cm , y radius = 1.6cm];
				\draw[color=red!60!yellow, thick] (3,6) arc [start angle=0, end angle = -90, x radius = 1.6cm , y radius = 1.6cm];
				\draw[color=red, thick] (3.2,2) arc [start angle=0, end angle = 90, x radius = 1.6cm , y radius = 1.8cm];	
				\draw[color=red, thick] (3.2,6) arc [start angle=0, end angle = -90, x radius = 1.6cm , y radius = 1.8cm];	
      			\draw[color=blue!50!cyan, thick] (3.5,2) to (3.5,6);
      			\draw[color=blue, thick] (3.8,2) to (3.8,6);
      			\draw[color=red!60!yellow, thick] (4.2,2) to (4.2,6);
      			\draw[color=red, thick] (4.4,2) to (4.4,6);
				\draw[color=blue, thick] (4.7,2) to (4.7,6);
				\draw[color=red, thick] (5,2) to (5,6);
			%
			%
      			\draw[color=cyan, thick] (1.9,0) to (1.9,2);
      			\draw[color=blue!50!cyan, thick] (2.2,0) to (2.2,2);
      			\draw[color=blue, thick] (2.5,0) to (2.5,2);
				\draw[color=red!20!yellow, thick] (2.8,0) to (2.8,2);
      			\draw[color=red!60!yellow, thick] (3,0) to (3,2);
      			\draw[color=red, thick] (3.2,0) to (3.2,2);	
      			\draw[color=blue!50!cyan, thick] (3.5,0) to (3.5,2);
      			\draw[color=blue, thick] (3.8,0) to (3.8,2);
      			\draw[color=red!60!yellow, thick] (4.2,0) to (4.2,2);
      			\draw[color=red, thick] (4.4,0) to (4.4,2);
				\draw[color=blue, thick] (4.7,0) to (4.7,0.3);
				\draw[color=blue, thick] (4.7,2) to (4.7,1.7);
				\node[circle,fill,draw,inner sep=0mm,minimum size=1mm,color=blue, thick] at (4.7,0.3) {};
				\node[circle,fill,draw,inner sep=0mm,minimum size=1mm,color=blue, thick] at (4.7,1.7) {};
				\draw[color=red, thick] (5,0) arc [start angle=0, end angle = 90, x radius = 0.6cm , y radius = 0.8cm];
				\draw[color=red, thick] (5,2) arc [start angle=0, end angle = -90, x radius = 0.6cm , y radius = 0.8cm];
			%
			%
			 	\draw[dotted] (-0.1,0) rectangle (5.1,6);
				\draw[dotted] (3.35,0) to (3.35,6);
				\draw[dotted] (-0.1,2) to (5.1,2);
    		\end{tikzpicture}
    		\end{array}
    		=
    		\begin{array}{c}
			\begin{tikzpicture}[scale=0.5]
      			\draw[color=red, thick] (0,0) to (0,6);
      			\draw[color=blue, thick] (0.2,0) to (0.2,6);
      			\draw[color=red!60!yellow, thick] (0.4,0) to (0.4,6);
      			\draw[color=red, thick] (0.6,0) to (0.6,6);
      			\draw[color=blue!50!cyan, thick] (0.8,0) to (0.8,6);
      			\draw[color=blue, thick] (1,0) to (1,6);
      			\draw[color=red!20!yellow, thick] (1.2,0) to (1.2,6);
      			\draw[color=red!60!yellow, thick] (1.4,0) to (1.4,6);
      			\draw[color=red, thick] (1.6,0) to (1.6,6);
      			\draw[color=cyan, thick] (1.9,0) to (1.9,0.8);
      			\draw[color=blue!50!cyan, thick] (2.2,0) to (2.2,0.8);
      			\draw[color=blue, thick] (2.5,0) to (2.5,0.8);
      			\draw[color=cyan, thick] (1.9,6) to (1.9,5.2);
      			\draw[color=blue!50!cyan, thick] (2.2,6) to (2.2,5.2);
      			\draw[color=blue, thick] (2.5,6) to (2.5,5.2);
				\node[circle,fill,draw,inner sep=0mm,minimum size=1mm,color=cyan, thick] at (1.9,0.8) {};
				\node[circle,fill,draw,inner sep=0mm,minimum size=1mm,color=blue!50!cyan, thick] at (2.2,0.8) {};
				\node[circle,fill,draw,inner sep=0mm,minimum size=1mm,color=blue, thick] at (2.5,0.8) {};
				\node[circle,fill,draw,inner sep=0mm,minimum size=1mm,color=cyan, thick] at (1.9,5.2) {};
				\node[circle,fill,draw,inner sep=0mm,minimum size=1mm,color=blue!50!cyan, thick] at (2.2,5.2) {};
				\node[circle,fill,draw,inner sep=0mm,minimum size=1mm,color=blue, thick] at (2.5,5.2) {};
				\draw[color=red!20!yellow, thick] (2.8,0) arc [start angle=0, end angle = 90, x radius = 1.6cm , y radius = 2cm];
				\draw[color=red!20!yellow, thick] (2.8,6) arc [start angle=0, end angle = -90, x radius = 1.6cm , y radius = 2cm];
				\draw[color=red!60!yellow, thick] (3,0) arc [start angle=0, end angle = 90, x radius = 1.6cm , y radius = 2.2cm];
				\draw[color=red!60!yellow, thick] (3,6) arc [start angle=0, end angle = -90, x radius = 1.6cm , y radius = 2.2cm];
				\draw[color=red, thick] (3.2,0) arc [start angle=0, end angle = 90, x radius = 1.6cm , y radius = 2.4cm];	
				\draw[color=red, thick] (3.2,6) arc [start angle=0, end angle = -90, x radius = 1.6cm , y radius = 2.4cm];	
      			\draw[color=blue!50!cyan, thick] (3.5,0) to (3.5,6);
      			\draw[color=blue, thick] (3.8,0) to (3.8,6);
      			\draw[color=red!60!yellow, thick] (4.2,0) to (4.2,6);
      			\draw[color=red, thick] (4.4,0) to (4.4,6);
				\draw[color=blue, thick] (4.7,0) to (4.7,0.3);
				\draw[color=blue, thick] (4.7,6) to (4.7,5.7);
				\node[circle,fill,draw,inner sep=0mm,minimum size=1mm,color=blue, thick] at (4.7,0.3) {};
				\node[circle,fill,draw,inner sep=0mm,minimum size=1mm,color=blue, thick] at (4.7,5.7) {};
				\draw[color=red, thick] (5,0) arc [start angle=0, end angle = 90, x radius = 0.6cm , y radius = 0.8cm];
				\draw[color=red, thick] (5,6) arc [start angle=0, end angle = -90, x radius = 0.6cm , y radius = 0.8cm];
    		\end{tikzpicture}
    		\end{array}
			\]
			where the top-left rectangle represents $f$ and the bottom-right rectangle represents $g$. 
			\par 
			\textit{Case 2} ($\ur_k = \emptyset$):
			As explained in Section \ref{Ssec: entries in lif}, if $\ur_k = \emptyset$ then the even reflections in $\uu_{t_k}$ form trivalent vertices with the even reflections ending $\uu_{t_{k-1}}$ in the light leaf $LL_{\ux, \{ t_k\}}$. 
			Consequently the left-most even reflections in $\varphi_{\{ t_k \}}$ which form trivalent vertices have a non-trivial interaction with the right-most even reflections forming trivalent vertices in $\varphi_{E'}$. 
			This is clearly the only interaction which is not resolved by isotopy relations. 
			Note that for any  $\red{s}$-reflection, repeatedly applying Frobenius associativity shows
			\[
			\begin{array}{c}
			\begin{tikzpicture}[scale=0.3]
      			\draw[color=red, thick] (0,0) to (0,6);
      			\draw[color=red, thick] (1.5,0) to (1.5,1);
      			\draw[color=red, thick] (1.5,5) to (1.5,6);
      			\draw[color=red, thick] (2.5,2) to (2.5,6);
				\draw[color=red, thick] (1.5,1) arc [start angle=0, end angle = 90, x radius = 1.5cm , y radius = 1cm];	
				\draw[color=red, thick] (1.5,5) arc [start angle=0, end angle = -90, x radius = 1.5cm , y radius = 1cm];
				\draw[color=red, thick] (2.5,0) arc [start angle=0, end angle = 90, x radius = 1cm , y radius = 0.5cm];
				\draw[color=red, thick] (2.5,2) arc [start angle=0, end angle = -90, x radius = 1cm , y radius = 1cm];
    		\end{tikzpicture}
    		\end{array}
    		=
    		\begin{array}{c}
			\begin{tikzpicture}[scale=0.3]
      			\draw[color=red, thick] (0,0) to (0,6);
      			\draw[color=red, thick] (1.5,0) to (1.5,1);
      			\draw[color=red, thick] (1.5,5) to (1.5,6);
      			\draw[color=red, thick] (2.5,4) to (2.5,6);
				\draw[color=red, thick] (1.5,1) arc [start angle=0, end angle = 90, x radius = 1.5cm , y radius = 1cm];	
				\draw[color=red, thick] (1.5,5) arc [start angle=0, end angle = -90, x radius = 1.5cm , y radius = 1cm];
				\draw[color=red, thick] (2.5,0) arc [start angle=0, end angle = 90, x radius = 1cm , y radius = 0.5cm];
				\draw[color=red, thick] (2.5,4) arc [start angle=0, end angle = -90, x radius = 2.5cm , y radius = 1cm];
    		\end{tikzpicture}
    		\end{array}
    		=
    		  \begin{array}{c}
			\begin{tikzpicture}[scale=0.3]
      			\draw[color=red, thick] (0,0) to (0,6);
      			\draw[color=red, thick] (1.5,0) to (1.5,1);
      			\draw[color=red, thick] (1.5,5) to (1.5,6);
				\draw[color=red, thick] (1.5,1) arc [start angle=0, end angle = 90, x radius = 1.5cm , y radius = 1cm];	
				\draw[color=red, thick] (1.5,5) arc [start angle=0, end angle = -90, x radius = 1.5cm , y radius = 1cm];
				\draw[color=red, thick] (2.5,0) arc [start angle=0, end angle = 90, x radius = 1cm , y radius = 0.5cm];
				\draw[color=red, thick] (2.5,6) arc [start angle=0, end angle = -90, x radius = 1cm , y radius = 0.5cm];
    		\end{tikzpicture}
    		\end{array}. 
			\]
			Since all even reflection commute, we can simultaneously apply the above identity to the non-trivially interacting even reflections in $\varphi_{E'} \varphi_{\{ t_k \}}$ to show 
			\[
			\begin{array}{c}
			\begin{tikzpicture}[scale=0.5]
      		%
			%
      			\draw[color=red!20!yellow, thick] (1.2,0) to (1.2,6);
      			\draw[color=red!60!yellow, thick] (1.4,0) to (1.4,6);
      			\draw[color=red, thick] (1.6,0) to (1.6,6);
      			\draw[color=cyan, thick] (1.9,2) to (1.9,2.5);
      			\draw[color=blue!50!cyan, thick] (2.2,2) to (2.2,2.5);
      			\draw[color=blue, thick] (2.5,2) to (2.5,2.5);
      			\draw[color=cyan, thick] (1.9,6) to (1.9,5.5);
      			\draw[color=blue!50!cyan, thick] (2.2,6) to (2.2,5.5);
      			\draw[color=blue, thick] (2.5,6) to (2.5,5.5);
				\node[circle,fill,draw,inner sep=0mm,minimum size=1mm,color=cyan, thick] at (1.9,2.4) {};
				\node[circle,fill,draw,inner sep=0mm,minimum size=1mm,color=blue!50!cyan, thick] at (2.2,2.4) {};
				\node[circle,fill,draw,inner sep=0mm,minimum size=1mm,color=blue, thick] at (2.5,2.4) {};
				\node[circle,fill,draw,inner sep=0mm,minimum size=1mm,color=cyan, thick] at (1.9,5.6) {};
				\node[circle,fill,draw,inner sep=0mm,minimum size=1mm,color=blue!50!cyan, thick] at (2.2,5.6) {};
				\node[circle,fill,draw,inner sep=0mm,minimum size=1mm,color=blue, thick] at (2.5,5.6) {};
				\draw[color=red!20!yellow, thick] (2.8,2) arc [start angle=0, end angle = 90, x radius = 1.6cm , y radius = 1.4cm];
				\draw[color=red!20!yellow, thick] (2.8,6) arc [start angle=0, end angle = -90, x radius = 1.6cm , y radius = 1.4cm];
				\draw[color=red!60!yellow, thick] (3,2) arc [start angle=0, end angle = 90, x radius = 1.6cm , y radius = 1.6cm];
				\draw[color=red!60!yellow, thick] (3,6) arc [start angle=0, end angle = -90, x radius = 1.6cm , y radius = 1.6cm];
				\draw[color=red, thick] (3.2,2) arc [start angle=0, end angle = 90, x radius = 1.6cm , y radius = 1.8cm];	
				\draw[color=red, thick] (3.2,6) arc [start angle=0, end angle = -90, x radius = 1.6cm , y radius = 1.8cm];	
      			\draw[color=blue!50!cyan, thick] (3.5,2) to (3.5,6);
      			\draw[color=blue, thick] (3.8,2) to (3.8,6);
      			\draw[color=red!60!yellow, thick] (4.2,2) to (4.2,6);
      			\draw[color=red, thick] (4.4,2) to (4.4,6);
			%
			%
      			\draw[color=cyan, thick] (1.9,0) to (1.9,2);
      			\draw[color=blue!50!cyan, thick] (2.2,0) to (2.2,2);
      			\draw[color=blue, thick] (2.5,0) to (2.5,2);
				\draw[color=red!20!yellow, thick] (2.8,0) to (2.8,2);
      			\draw[color=red!60!yellow, thick] (3,0) to (3,2);
      			\draw[color=red, thick] (3.2,0) to (3.2,2);
      			\draw[color=blue!50!cyan, thick] (3.5,0) to (3.5,0.4);
      			\draw[color=blue, thick] (3.8,0) to (3.8,0.4);
      			\draw[color=blue!50!cyan, thick] (3.5,2) to (3.5,1.6);
      			\draw[color=blue, thick] (3.8,2) to (3.8,1.6);
				\node[circle,fill,draw,inner sep=0mm,minimum size=1mm,color=blue!50!cyan, thick] at (3.5,0.4) {};
				\node[circle,fill,draw,inner sep=0mm,minimum size=1mm,color=blue, thick] at (3.8,0.4) {};
				\node[circle,fill,draw,inner sep=0mm,minimum size=1mm,color=blue!50!cyan, thick] at (3.5,1.6) {};
				\node[circle,fill,draw,inner sep=0mm,minimum size=1mm,color=blue, thick] at (3.8,1.6) {};
      			\draw[color=red!60!yellow, thick] (4.2,0) arc [start angle=0, end angle = 89, x radius = 1.2cm , y radius = 0.75 cm];
      			\draw[color=red!60!yellow, thick] (4.2,2) arc [start angle=0, end angle = -89, x radius = 1.2cm , y radius = 0.75 cm];
				\draw[color=red, thick] (4.4,0) arc [start angle=0, end angle = 90, x radius = 1.2cm , y radius = 0.9 cm];
				\draw[color=red, thick] (4.4,2) arc [start angle=0, end angle = -90, x radius = 1.2cm , y radius = 0.9 cm];
    		\end{tikzpicture}
    		\end{array}
    		=
    		\begin{array}{c}
			\begin{tikzpicture}[scale=0.5]
      		%
			%
      			\draw[color=red!20!yellow, thick] (1.2,0) to (1.2,6);
      			\draw[color=red!60!yellow, thick] (1.4,0) to (1.4,6);
      			\draw[color=red, thick] (1.6,0) to (1.6,6);
				\draw[color=red!20!yellow, thick] (2.8,2) arc [start angle=0, end angle = 90, x radius = 1.6cm , y radius = 1.4cm];
				\draw[color=red!20!yellow, thick] (2.8,6) arc [start angle=0, end angle = -90, x radius = 1.6cm , y radius = 1.4cm];
				\draw[color=red!60!yellow, thick] (3,2) arc [start angle=0, end angle = 90, x radius = 1.6cm , y radius = 1.6cm];
				\draw[color=red!60!yellow, thick] (3,6) arc [start angle=0, end angle = -90, x radius = 1.6cm , y radius = 1.6cm];
				\draw[color=red, thick] (3.2,2) arc [start angle=0, end angle = 90, x radius = 1.6cm , y radius = 1.8cm];	
				\draw[color=red, thick] (3.2,6) arc [start angle=0, end angle = -90, x radius = 1.6cm , y radius = 1.8cm];	
      			\draw[color=blue!50!cyan, thick] (3.5,0) to (3.5,0.4);
      			\draw[color=blue, thick] (3.8,0) to (3.8,0.4);
      			\draw[color=blue!50!cyan, thick] (3.5,6) to (3.5,5.6);
      			\draw[color=blue, thick] (3.8,6) to (3.8,5.6);
				\node[circle,fill,draw,inner sep=0mm,minimum size=1mm,color=blue!50!cyan, thick] at (3.5,0.4) {};
				\node[circle,fill,draw,inner sep=0mm,minimum size=1mm,color=blue, thick] at (3.8,0.4) {};
				\node[circle,fill,draw,inner sep=0mm,minimum size=1mm,color=blue!50!cyan, thick] at (3.5,5.6) {};
				\node[circle,fill,draw,inner sep=0mm,minimum size=1mm,color=blue, thick] at (3.8,5.6) {};
      			\draw[color=red!60!yellow, thick] (4.2,2) to (4.2,6);
      			\draw[color=red, thick] (4.4,2) to (4.4,6);
			%
			%
      			\draw[color=cyan, thick] (1.9,6) to (1.9,5.5);
      			\draw[color=blue!50!cyan, thick] (2.2,6) to (2.2,5.5);
      			\draw[color=blue, thick] (2.5,6) to (2.5,5.5);
				\node[circle,fill,draw,inner sep=0mm,minimum size=1mm,color=cyan, thick] at (1.9,5.6) {};
				\node[circle,fill,draw,inner sep=0mm,minimum size=1mm,color=blue!50!cyan, thick] at (2.2,5.6) {};
				\node[circle,fill,draw,inner sep=0mm,minimum size=1mm,color=blue, thick] at (2.5,5.6) {};
      			\draw[color=cyan, thick] (1.9,0) to (1.9,0.8);
      			\draw[color=blue!50!cyan, thick] (2.2,0) to (2.2,0.8);
      			\draw[color=blue, thick] (2.5,0) to (2.5,0.8);
				\node[circle,fill,draw,inner sep=0mm,minimum size=1mm,color=cyan, thick] at (1.9,0.8) {};
				\node[circle,fill,draw,inner sep=0mm,minimum size=1mm,color=blue!50!cyan, thick] at (2.2,0.8) {};
				\node[circle,fill,draw,inner sep=0mm,minimum size=1mm,color=blue, thick] at (2.5,0.8) {};
				\draw[color=red!20!yellow, thick] (2.8,0) to (2.8,2);
      			\draw[color=red!60!yellow, thick] (3,0) to (3,2);
      			\draw[color=red, thick] (3.2,0) to (3.2,2);
      			\draw[color=blue!50!cyan, thick] (3.5,0) to (3.5,0.4);
      			\draw[color=blue, thick] (3.8,0) to (3.8,0.4);
				\node[circle,fill,draw,inner sep=0mm,minimum size=1mm,color=blue!50!cyan, thick] at (3.5,0.4) {};
				\node[circle,fill,draw,inner sep=0mm,minimum size=1mm,color=blue, thick] at (3.8,0.4) {};
      			\draw[color=red!60!yellow, thick] (4.2,0) arc [start angle=0, end angle = 89, x radius = 1.2cm , y radius = 0.75 cm];
      			\draw[color=red!60!yellow, thick] (4.2,2) arc [start angle=0, end angle = -89, x radius = 1.2cm , y radius = 0.75 cm];
				\draw[color=red, thick] (4.4,0) arc [start angle=0, end angle = 90, x radius = 1.2cm , y radius = 0.9 cm];
				\draw[color=red, thick] (4.4,2) arc [start angle=0, end angle = -90, x radius = 1.2cm , y radius = 0.9 cm];
    		\end{tikzpicture}
    		\end{array}
    		=
    		\begin{array}{c}
			\begin{tikzpicture}[scale=0.5]
      		%
			%
      			\draw[color=red!20!yellow, thick] (1.2,0) to (1.2,6);
      			\draw[color=red!60!yellow, thick] (1.4,0) to (1.4,6);
      			\draw[color=red, thick] (1.6,0) to (1.6,6);
				\draw[color=red!20!yellow, thick] (2.8,1) arc [start angle=0, end angle = 90, x radius = 1.6cm , y radius = 1.4cm];
				\draw[color=red!20!yellow, thick] (2.8,6) arc [start angle=0, end angle = -90, x radius = 1.6cm , y radius = 1.4cm];
				\draw[color=red!60!yellow, thick] (3,1) arc [start angle=0, end angle = 90, x radius = 1.6cm , y radius = 1.6cm];
				\draw[color=red!60!yellow, thick] (3,6) arc [start angle=0, end angle = -90, x radius = 1.6cm , y radius = 1.6cm];
				\draw[color=red, thick] (3.2,1) arc [start angle=0, end angle = 90, x radius = 1.6cm , y radius = 1.8cm];	
				\draw[color=red, thick] (3.2,6) arc [start angle=0, end angle = -90, x radius = 1.6cm , y radius = 1.8cm];	
      			\draw[color=blue!50!cyan, thick] (3.5,0) to (3.5,0.4);
      			\draw[color=blue, thick] (3.8,0) to (3.8,0.4);
      			\draw[color=blue!50!cyan, thick] (3.5,6) to (3.5,5.6);
      			\draw[color=blue, thick] (3.8,6) to (3.8,5.6);
				\node[circle,fill,draw,inner sep=0mm,minimum size=1mm,color=blue!50!cyan, thick] at (3.5,0.4) {};
				\node[circle,fill,draw,inner sep=0mm,minimum size=1mm,color=blue, thick] at (3.8,0.4) {};
				\node[circle,fill,draw,inner sep=0mm,minimum size=1mm,color=blue!50!cyan, thick] at (3.5,5.6) {};
				\node[circle,fill,draw,inner sep=0mm,minimum size=1mm,color=blue, thick] at (3.8,5.6) {};
      			\draw[color=red!60!yellow, thick] (4.2,4.5) to (4.2,6);
      			\draw[color=red, thick] (4.4,4.5) to (4.4,6);
			%
			%
      			\draw[color=cyan, thick] (1.9,6) to (1.9,5.5);
      			\draw[color=blue!50!cyan, thick] (2.2,6) to (2.2,5.5);
      			\draw[color=blue, thick] (2.5,6) to (2.5,5.5);
				\node[circle,fill,draw,inner sep=0mm,minimum size=1mm,color=cyan, thick] at (1.9,5.6) {};
				\node[circle,fill,draw,inner sep=0mm,minimum size=1mm,color=blue!50!cyan, thick] at (2.2,5.6) {};
				\node[circle,fill,draw,inner sep=0mm,minimum size=1mm,color=blue, thick] at (2.5,5.6) {};
      			\draw[color=cyan, thick] (1.9,0) to (1.9,0.8);
      			\draw[color=blue!50!cyan, thick] (2.2,0) to (2.2,0.8);
      			\draw[color=blue, thick] (2.5,0) to (2.5,0.8);
				\node[circle,fill,draw,inner sep=0mm,minimum size=1mm,color=cyan, thick] at (1.9,0.8) {};
				\node[circle,fill,draw,inner sep=0mm,minimum size=1mm,color=blue!50!cyan, thick] at (2.2,0.8) {};
				\node[circle,fill,draw,inner sep=0mm,minimum size=1mm,color=blue, thick] at (2.5,0.8) {};
				\draw[color=red!20!yellow, thick] (2.8,0) to (2.8,1);
      			\draw[color=red!60!yellow, thick] (3,0) to (3,1);
      			\draw[color=red, thick] (3.2,0) to (3.2,1);
      			\draw[color=blue!50!cyan, thick] (3.5,0) to (3.5,0.4);
      			\draw[color=blue, thick] (3.8,0) to (3.8,0.4);
				\node[circle,fill,draw,inner sep=0mm,minimum size=1mm,color=blue!50!cyan, thick] at (3.5,0.4) {};
				\node[circle,fill,draw,inner sep=0mm,minimum size=1mm,color=blue, thick] at (3.8,0.4) {};
      			\draw[color=red!60!yellow, thick] (4.2,0) arc [start angle=0, end angle = 89, x radius = 1.2cm , y radius = 0.75 cm];
      			\draw[color=red!60!yellow, thick] (4.2,4.5) arc [start angle=0, end angle = -89, x radius = 2.85cm , y radius = 1 cm];
				\draw[color=red, thick] (4.4,0) arc [start angle=0, end angle = 90, x radius = 1.2cm , y radius = 0.9 cm];
				\draw[color=red, thick] (4.4,4.5) arc [start angle=0, end angle = -90, x radius = 2.8cm , y radius = 1.15 cm];
    		\end{tikzpicture}
    		\end{array} 
    		= 
    		\begin{array}{c}
			\begin{tikzpicture}[scale=0.5]
      			\draw[color=red!20!yellow, thick] (1.2,0) to (1.2,6);
      			\draw[color=red!60!yellow, thick] (1.4,0) to (1.4,6);
      			\draw[color=red, thick] (1.6,0) to (1.6,6);
      			\draw[color=cyan, thick] (1.9,0) to (1.9,0.8);
      			\draw[color=blue!50!cyan, thick] (2.2,0) to (2.2,0.8);
      			\draw[color=blue, thick] (2.5,0) to (2.5,0.8);
      			\draw[color=cyan, thick] (1.9,6) to (1.9,5.2);
      			\draw[color=blue!50!cyan, thick] (2.2,6) to (2.2,5.2);
      			\draw[color=blue, thick] (2.5,6) to (2.5,5.2);
				\node[circle,fill,draw,inner sep=0mm,minimum size=1mm,color=cyan, thick] at (1.9,0.8) {};
				\node[circle,fill,draw,inner sep=0mm,minimum size=1mm,color=blue!50!cyan, thick] at (2.2,0.8) {};
				\node[circle,fill,draw,inner sep=0mm,minimum size=1mm,color=blue, thick] at (2.5,0.8) {};
				\node[circle,fill,draw,inner sep=0mm,minimum size=1mm,color=cyan, thick] at (1.9,5.2) {};
				\node[circle,fill,draw,inner sep=0mm,minimum size=1mm,color=blue!50!cyan, thick] at (2.2,5.2) {};
				\node[circle,fill,draw,inner sep=0mm,minimum size=1mm,color=blue, thick] at (2.5,5.2) {};
				\draw[color=red!20!yellow, thick] (2.8,0) arc [start angle=0, end angle = 90, x radius = 1.6cm , y radius = 2cm];
				\draw[color=red!20!yellow, thick] (2.8,6) arc [start angle=0, end angle = -90, x radius = 1.6cm , y radius = 2cm];
				\draw[color=red!60!yellow, thick] (3,0) arc [start angle=0, end angle = 90, x radius = 1.6cm , y radius = 2.2cm];
				\draw[color=red!60!yellow, thick] (3,6) arc [start angle=0, end angle = -90, x radius = 1.6cm , y radius = 2.2cm];
				\draw[color=red, thick] (3.2,0) arc [start angle=0, end angle = 90, x radius = 1.6cm , y radius = 2.4cm];	
				\draw[color=red, thick] (3.2,6) arc [start angle=0, end angle = -90, x radius = 1.6cm , y radius = 2.4cm];	
      			\draw[color=blue!50!cyan, thick] (3.5,0) to (3.5,0.4);
      			\draw[color=blue, thick] (3.8,0) to (3.8,0.4);
      			\draw[color=blue!50!cyan, thick] (3.5,6) to (3.5,5.6);
      			\draw[color=blue, thick] (3.8,6) to (3.8,5.6);
				\node[circle,fill,draw,inner sep=0mm,minimum size=1mm,color=blue!50!cyan, thick] at (3.5,0.4) {};
				\node[circle,fill,draw,inner sep=0mm,minimum size=1mm,color=blue, thick] at (3.8,0.4) {};
				\node[circle,fill,draw,inner sep=0mm,minimum size=1mm,color=blue!50!cyan, thick] at (3.5,5.6) {};
				\node[circle,fill,draw,inner sep=0mm,minimum size=1mm,color=blue, thick] at (3.8,5.6) {};
      			\draw[color=red!60!yellow, thick] (4.2,0) arc [start angle=0, end angle = 89, x radius = 1.6cm , y radius = 1.4cm];
      			\draw[color=red!60!yellow, thick] (4.2,6) arc [start angle=0, end angle = -89, x radius = 1.6cm , y radius = 1.4cm];
				\draw[color=red, thick] (4.4,0) arc [start angle=0, end angle = 90, x radius = 1.6cm , y radius = 1.6cm];
				\draw[color=red, thick] (4.4,6) arc [start angle=0, end angle = -90, x radius = 1.6cm , y radius = 1.6cm];
    		\end{tikzpicture}
    		\end{array}.    		
			\]
			The general case where $\ur_{j} = \emptyset$, for all $l \leq j \leq  k$, follows by repeatedly applying Frobenius associativity in a similar fashion. 	
		\end{proof}

		\begin{prop}
		\label{prop: idempotent}
			Let $2 \in \Bbbk^{\times}$. 
			For any $x \in {}^I W$ the element 
			\begin{align*}
				e_x 
				= 
				\sum_{E \subseteq E(x)} 
				\frac{(-1)^{\ell(E)/2} }{(-2)^{|E|}} \varphi_E
				=
				\prod_{t \in E(x)} \left( 1 - \frac{(-1)^{t/2}}{2} \varphi_{\{ t\}} \right).			\end{align*} 
			is the primitive idempotent in $\Endo^{0}(B_{\ux})$ corresponding to the projection $B_{\ux} \twoheadrightarrow B_x$. 
		\end{prop}
		
		\begin{proof}
			By Theorem \ref{thm: indecomposable objects}, 
			Proposition \ref{prop: alg presentation} also provides a presentation for $\Endo^0(B_{\ux})$ over any complete local ring $\Bbbk$, and for any reduced expression $\ux$ of $x$. 
			The factorised form of $e_x$ shows $\varphi_{\{ t \}} e_x =0$ for every $t \in E(x)$, and $e_x$ is idempotent.
			Moreover, it is evident from the presentation that $\{ \varphi_{\{ t \}} ~\vert~ t \in E(x) \}$ generates the unique maximal ideal of $\Endo^0(B_x)$. 
			Hence $e_x$ annihilates every non-identity morphism.  
			Consequently $e_x$ must be the projection $B_{\ux} \twoheadrightarrow B_x$, as \cite[\S 6]{LW22} implies $\Endo^0(B_x)$ is spanned by the identity morphism. 
		\end{proof}
		
		This gives a third proof that the ${}^p d_x = d_x$ when $p>2$.  
		
		\begin{exmp}
			Continuing from Example \ref{Exmp: endomorphism basis} where $x = \{ 6,5,4,3,2,1\}$, the idempotent is
			\begin{align*}
				e_x &=  \left( 1 + \frac{1}{2}\varphi_{\{ 2\}} \right)
				 	\left( 1 - \frac{1}{2}\varphi_{\{ 4\}} \right)
				  	\left( 1 + \frac{1}{2}\varphi_{\{ 6\}} \right)
				\\&= 
					1 
					+ \frac{1}{2} \varphi_{\{ 2\}} 
					- \frac{1}{2} \varphi_{\{ 4\}} 
					+ \frac{1}{2} \varphi_{\{ 6\}}
					-\frac{1}{4} \varphi_{\{ 4,2\}}
					+\frac{1}{4} \varphi_{\{ 6,2\}}
					-\frac{1}{4} \varphi_{\{ 6,4\}}
					-\frac{1}{8} \varphi_{\{ 6,4,2\}}.
			\end{align*}
		\end{exmp}
		
		\begin{rem}
		\label{rem: minu p=2 idempotent}
			Essentially identical arguments show that when $\h$ is a Cartan realisation of type $\tC_n$ and $\ux$ is any reduced expression of $x \in {}^I W$, the algebra $\Endo^{0}(B_{\ux})$ has a presentation 
			\begin{align*}
				\Endo^0(B_{\ux}) 
				\cong 
				\frac{\Bbbk[\varphi_E \vert E \subseteq E(x)]}
				{\left(\varphi_E \varphi_{E'} - (-1)^{\ell(E \cap E')/2 } \varphi_{E \cup E'} \right)}.
			\end{align*}
			One can then argue as in Proposition \ref{prop: idempotent} to show
			\begin{align*}
				e_x 
				= \sum_{E \subseteq E(x)} (-1)^{\ell(E)/2 + |E|} \varphi_{E} 
				= \prod_{t \in E(x)} \left( 1 - (-1)^{t/2}\varphi_{\{ t\}} \right)
			\end{align*} 
			is the primitive idempotent corresponding to the projection $B_{\ux} \rightarrow B_x$. 
			Since $e_x$ is defined over $\Z$, this a third proof ${}^p d_x = d_x$ for all $p$. 
		\end{rem}
		\begin{exmp}
			If we $x = \{ 6,5,4,3,2,1\}$ when $\h$ is a Cartan realisation of type $\tC_n$, the basis of $\Endo^0(B_{\ux })$ is identical to that in Example \ref{Exmp: endomorphism basis}, but the expression for the idempotent becomes
			\begin{align*}
				e_x &=  \left( 1 - \varphi_{\{ 2\}} \right)
				 	\left( 1 +\varphi_{\{ 4\}} \right)
				  	\left( 1 -\varphi_{\{ 6\}} \right)
				\\&= 
					1 
					-  \varphi_{\{ 2\}} 
					+  \varphi_{\{ 4\}} 
					-  \varphi_{\{ 6\}}
					- \varphi_{\{ 4,2\}}
					+ \varphi_{\{ 6,2\}}
					-\varphi_{\{ 6,4\}}
					+\varphi_{\{ 6,4,2\}}.
			\end{align*}
		\end{exmp}

	\subsection{Antispherical categories for odd dimensional quadric hypersurfaces}
	\label{Ssec: odd quadrics}
		
		We conclude by determining the $p$-Kazhdan-Lusztig bases of antispherical categories for odd dimensional quadric hypersurfaces. 
		Throughout we take $\h$ to be a Cartan realisation of type $\tC_n$, and $I \subseteq S$ will be type $\tC_{n-1}$. 
		We label our simple reflections as follows:
		\[
		\begin{tikzpicture}  
				\draw (-1.5,0.05) -- (-1,0.05);
				\draw (-1.5,-0.05) -- (-1,-0.05);
				\draw (-0.5,0) -- (-1,0);
				\draw (0.5,0) -- (1,0);
				\draw (1,0) -- (1.5,0);
						\draw (-1.3,0) -- (-1.1,0.15);
						\draw (-1.3,0) -- (-1.1,-0.15);
				\filldraw[black] (-1.5,-0.5) node[font = \small]{${}_0$}; 
				\filldraw[black] (-1,-0.5) node[font = \small]{${}_1$};
				\filldraw[black] (-0.5,-0.5) node[font = \small]{${}_2$}; 
					\filldraw[black] (1.5,-0.5) node[font = \small]{${}_{n-1}$};	
				\filldraw[black] (0,0) node[font = \small]{\dots};	
				\node[circle,fill,draw,inner sep=0mm,minimum size=2mm,color=red, thick] at (-1.5,0) {}; 
				\node[circle,fill,draw,inner sep=0mm,minimum size=2mm,color=blue!20!red!80, thick] at (-1,0) {};
				\node[circle,fill,draw,inner sep=0mm,minimum size=2mm,color=red!60!teal!60, thick] at (-0.5,0) {};
				\node[circle,fill,draw,inner sep=0mm,minimum size=2mm,color=blue!60!red!60, thick] at (0.5,0) {}; 
				\node[circle,fill,draw,inner sep=0mm,minimum size=2mm,color=blue!90!red!40, thick] at (1,0) {};
				\node[circle,fill,draw,inner sep=0mm,minimum size=2mm,color=blue, thick] at (1.5,0) {};	 
		\end{tikzpicture}
		\]
		
		There is a unique minimal coset representative $x_i$ of length $i$ in $W^I$, where $0 \leq i < 2n$. 
		Each minimal coset representative $x_i$ has a unique reduced expression $\ux_i$ where
			\begin{align*}
				\ux_i 
			= 
				\begin{cases}
					\underline{\emptyset}
				&
					\text{ if } i=0,
				\\
					s_{n-1} s_{n-2} \dots s_{n-i}
				&
					\text{ if } 0<i \leq n,
				\\
					s_{n-1} s_{n-2} \dots s_1 s_{0} s_1 s_2 \dots s_{i-n} 
				&
					\text{ if } n < i < 2n,
				\end{cases}
			\end{align*}
		A simple induction shows:
		\begin{lem}
		\label{lem: bs basis easy}
			For $0 \leq i < 2n$	the Bott-Samelson element $d_{\ux_i}$ is
			\begin{align*}
				d_{\ux_i} 
			= 
				\begin{cases}
				\begin{array}{lll}
					\delta^I_{\id} 
					& = d_{x_i} 
					& \text{ if } i=0 
					\\ 
					\delta^I_{x_{i}} + v \delta^I_{x_{i-1}}
					&= d_{x_i}
					&\text{ if } 0<i \leq n,
					\\
					\delta^I_{x_{i}} + v \delta^I_{x_{i-1}} + \delta^I_{x_{2n-i}} + v \delta^I_{x_{2n-i-1}} 
					&= d_{x_{i}} + d_{x_{2n-i}} 
					&\text{ if } n < i < 2n.
				\end{array}
				\end{cases}
			\end{align*}
		\end{lem}
		It follows immediately from Lemma \ref{Lem: pKL bound} that ${}^p d_{x_i} = d_{x_i}$ for all $p$, whenever $0 \leq i \leq n$. 
		When $n < i < 2n$, the only local intersection form  that needs to be computed is $I_{  x_{2n-i}, \ux_i}$. 
		Observe that
		\begin{align*}
				\ux_i ^{\ue}
				:= 
				\begin{smallmatrix}
					\ua		&	\ua		&	\dots	&	\ua			&	
					\ua		&	\da		&	\dots	&	\da			& \da
				\\
					1		&	1		&	\dots	&	1			&	
					0		&	1		&	\dots	&	1			& 0
				\\
					s_{n-1}	&	s_{n-2}	&	\dots 	&	s_{1}		&	
					s_{0}	&	s_{1}	&	\dots	&	s_{i-n-1}	&	s_{i-n}
				\end{smallmatrix}
				\subset \ux_i
		\end{align*}
		is an $I$-antispherical subexpression of $\ux_i$ which evaluates to $x_{2n-i}$ and has defect $0$. 
		Note that the corresponding light leaf $LL_{\ux_i , e}$ is 
		\[
		\begin{array}{c}
		\begin{tikzpicture}[scale=0.6]
      		\draw[color=red, thick] (-4,0) to (-4,2.75);
    		\filldraw[black] (-4,-0.5) node{$0$};
      		\draw[color=blue!20!red!80, thick] (-3.5,0) to (-3.5,2.75);
    		\filldraw[black] (-3.5,-0.5) node{$1$};
    		\filldraw[black] (-3,0) node{${}_{\dots}$};
    		\draw[color=blue!60!red!60, thick] (-2.5,0) to (-2.5,2.75);
    		\draw[color=blue!60!red!60, thick] (2.5,0) arc [start angle=0, end angle = 90, x radius = 5cm , y radius = 2.2cm];
    		\filldraw[black] (-2.5,-0.5) node{$i$-$n$};
    		\filldraw[black] (2.5,-0.5) node{$i$-$n$};
      		\draw[color=blue!50!red!50, thick] (2,0) arc [start angle=0, end angle = 180, x radius = 2cm , y radius = 1.5cm];
    		\filldraw[black] (-1.5,0) node{${}_{\dots}$};
    		\filldraw[black] (1.5,0) node{${}_{\dots}$};
			\draw[color=blue!80!red!30, thick] (1,0) arc [start angle=0, end angle = 180, x radius = 1cm , y radius = 1cm];
      		\draw[color=blue, thick] (0,0.5) to (0,0);
			\node[circle,fill,draw,inner sep=0mm,minimum size=1mm,color=blue, thick] at (0,0.5) {};
			\filldraw[black] (0,-0.5) node{$n$-1};
    	\end{tikzpicture}
		\end{array}
		\]
		The local intersection form is entirely determined by the composition $LL_{\ux_i , e} \circ \Gamma \Gamma^{\ux_i , e}$, namely
		\[
		\begin{array}{c}
		\begin{tikzpicture}[scale=0.45]
      		\draw[color=red, thick] (-4,-2.75) to (-4,2.75);
      		\draw[color=blue!20!red!80, thick] (-3.5,-2.75) to (-3.5,2.75);
    		\filldraw[black] (-3,0) node{${}_{\dots}$};
    		\draw[color=blue!60!red!60, thick] (-2.5,-2.75) to (-2.5,2.75);
    		\draw[color=blue!60!red!60, thick] (2.5,0) arc [start angle=0, end angle = 90, x radius = 5cm , y radius = 2.2cm];
    		\draw[color=blue!60!red!60, thick] (2.5,0) arc [start angle=0, end angle = -90, x radius = 5cm , y radius = 2.2cm];
      		\draw[color=blue!50!red!50, thick] (2,0) arc [start angle=0, end angle = 360, x radius = 2cm , y radius = 1.5cm];
    		\filldraw[black] (-1.5,0) node{${}_{\dots}$};
    		\filldraw[black] (1.5,0) node{${}_{\dots}$};
			\draw[color=blue!80!red!30, thick] (1,0) arc [start angle=0, end angle = 360, x radius = 1cm , y radius = 1cm];
      		\draw[color=blue, thick] (0,0.5) to (0,-0.5);
			\node[circle,fill,draw,inner sep=0mm,minimum size=1mm,color=blue, thick] at (0,0.5) {};
			\node[circle,fill,draw,inner sep=0mm,minimum size=1mm,color=blue, thick] at (0,-0.5) {};
    	\end{tikzpicture}
		\end{array}
		= 
		\prod_{k=n-1}^{i-n+1}
		\partial_{s_{k-1}}(\alpha_{s_k})
		\begin{array}{c}
		\begin{tikzpicture}[scale=0.45]
      		\draw[color=red, thick] (-4,-2.75) to (-4,2.75);
      		\draw[color=blue!20!red!80, thick] (-3.5,-2.75) to (-3.5,2.75);
    		\filldraw[black] (-3,0) node{${}_{\dots}$};
    		\draw[color=blue!60!red!60, thick] (-2.5,-2.75) to (-2.5,2.75);
    	\end{tikzpicture}
		\end{array}
		=
		2(-1)^{i-n}
		\begin{array}{c}
		\begin{tikzpicture}[scale=0.45]
      		\draw[color=red, thick] (-4,-2.75) to (-4,2.75);
      		\draw[color=blue!20!red!80, thick] (-3.5,-2.75) to (-3.5,2.75);
    		\filldraw[black] (-3,0) node{${}_{\dots}$};
    		\draw[color=blue!60!red!60, thick] (-2.5,-2.75) to (-2.5,2.75);
    	\end{tikzpicture}
		\end{array}
		\]
		where the first equality follows from repeated applying the polynomial forcing relation, the barbell relation and the needle relation as follows:
		\[
		\begin{array}{c}
		\begin{tikzpicture}[scale=0.8]
       		\draw[dotted] (1,0) arc [start angle=0, end angle = 360, x radius = 1cm , y radius = 1cm];
			\draw[color=teal, thick] (0.7,0) arc [start angle=0, end angle = 360, x radius = 0.7cm , y radius = 0.7cm];
      		\draw[color=blue, thick] (0,0.5) to (0,-0.5);
			\node[circle,fill,draw,inner sep=0mm,minimum size=1mm,color=blue, thick] at (0,0.5) {};
			\node[circle,fill,draw,inner sep=0mm,minimum size=1mm,color=blue, thick] at (0,-0.5) {};
    	\end{tikzpicture}
		\end{array}
		= 
		\partial_{\color{teal}t}(\alpha_{\color{blue}s})
		\begin{array}{c}
		\begin{tikzpicture}[scale=0.8]
       		\draw[dotted] (1,0) arc [start angle=0, end angle = 360, x radius = 1cm , y radius = 1cm];
      		\draw[color=teal, thick] (0,0.5) to (0,-0.5);
			\node[circle,fill,draw,inner sep=0mm,minimum size=1mm,color=teal, thick] at (0,0.5) {};
			\node[circle,fill,draw,inner sep=0mm,minimum size=1mm,color=teal, thick] at (0,-0.5) {};
    	\end{tikzpicture}
		\end{array}
		+
		{\color{teal}t}(\alpha_{\color{blue}s})
		\begin{array}{c}
		\begin{tikzpicture}[scale=0.8]
       		\draw[dotted] (1,0) arc [start angle=0, end angle = 360, x radius = 1cm , y radius = 1cm];
			\draw[color=teal, thick] (0.7,0) arc [start angle=0, end angle = 360, x radius = 0.7cm , y radius = 0.7cm];
    	\end{tikzpicture}
		\end{array}
		=
		\begin{array}{c}
		\begin{tikzpicture}[scale=0.8]
       		\draw[dotted] (1,0) arc [start angle=0, end angle = 360, x radius = 1cm , y radius = 1cm];
			\filldraw[black] (0,0) node{$\partial_{\color{teal}t}(\alpha_{\blue{s}}) \alpha_{\color{teal}t}$};
    	\end{tikzpicture}
		\end{array}
		.
		\]
		Hence, the local intersection form is the $1 \times 1$ matrix
		\begin{align*}
			I_{ x_{n-i}, \ux_i} = \left[ 2(-1)^{n-i} \right]. 
		\end{align*} 
		The following proposition follows immediately.
		
		\begin{prop}
			Let $\h$ be a Cartan realisation of type $\tC_n$ and $I \subset S$ of type $\tC_{n-1}$. 
			The $p$-Kazhdan-Lusztig basis of $\DRA(\h,I, \Bbbk)$ is
			\begin{align*}
				{}^p d_x 
				= 
				\begin{cases}
					d_{\ux} &\text{if } 2 \notin \Bbbk^{\times}, \text{ and }
				\\
					d_x &\text{if } 2 \in \Bbbk^{\times},
				\end{cases}
			\end{align*}
			where $\ux$ is the unique reduced expression of $x$. 
			Moreover, each $p$-Kazhdan-Lusztig polynomial is a monic monomial. 
		\end{prop}

		\begin{rem}
			We can again obtain a presentation for the degree-0 endomorphism algebra of Bott-Samelson objects. 
			If $0 \leq i \leq n$ then $\Endo^0(B_{\ux_i}) \cong \Bbbk$. 
			If $n < i < 2n$ then 
			\begin{align*}
				\Endo^0 (B_{\ux_i})
				\cong
				\frac{\Bbbk[\varphi]}{\varphi^2 - 2(-1)^{n-i} \varphi}.
			\end{align*}
			If $2 \in \Bbbk^{\times}$ the idempotent $e_x$ corresponding to the projection $B_{\ux_i} \twoheadrightarrow B_{x_i}$ is $e_{x_i} = 1 - (-1)^{n-i} \varphi/2$, giving a third proof ${}^p d_x = d_x$ for all $x \in {}^I W$ if $p>2$. 
		\end{rem}
		
		\begin{rem}
			Modifying the preceding arguments shows if $\h$ is type $\tB_n$ and $I$ type $\tB_{n-1}$ then for  $n < i < 2n$ we have
			\begin{align*}
				\Endo^0 (B_{\ux_i})
				\cong
				\frac{\Bbbk[\varphi]}{\varphi^2 - (-1)^{n-i} \varphi}.
			\end{align*}
			The idempotent $e_x$ corresponding to the projection $B_{\ux_i} \twoheadrightarrow B_{x_i}$ is $e_{x_i} = 1 - (-1)^{n-i} \varphi$ which is defined over $\Z$, giving a third proof ${}^p d_x = d_x$ for all $p$ and all $x \in {}^I W$. 
		\end{rem}

\section{Graded decomposition numbers}
\label{Sec: Decomposition numbers}

	We now relate the graded decomposition numbers for mixed tilting, projective, parity and intersection cohomology sheaves under the (derived) extension of scalars and (derived) modular reduction functors. 
	This generalises \cite[\S 2.7]{AR16a} to the mixed and parabolic settings.
	The results of this Section will be used in Section \ref{Sec: Spherical} to determine the $p$-Kazhdan-Lusztig basis of spherical (co)minuscule Hecke categories. 

\subsection{Standard and Costandard objects}
	
	We recall basic properties of (co)standard objects following \cite{AR16a,AR16b,AMRW19}. 
	\\
	\par 
	For a Schubert cell $B \backslash B x \cdot B$, let $i_x : B \backslash B x \cdot B \hookrightarrow B \backslash G$ denote the locally closed inclusion of strata. 
	The functors $i_{x*}$ and $i_{x!}$ restrict to functors of the subcategories of parity sheaves, so induce functors on the relevant mixed categories. 
	We define $\Delta_x^B$ and $\nabla_x^B$ in $D^{\mix}(B \backslash G / B )$ by
	\begin{align*}
		\Delta_x^B := i_{x!} \underline{\Bbbk}_{X_x} (\ell(x))
		&&
		\nabla_x^B := i_{x*} \underline{\Bbbk}_{X_x} (\ell(x))
	\end{align*}
	and denote by $\Delta_x^U$ and $\nabla_x^U$ (resp. $\Delta_x^P$ and $\nabla_x^P$) the analogously defined sheaves in $D^{\mix}(B \backslash G / U )$ (resp. $D^{\mix}(U \backslash G / P )$).
	There is an analogous construction of $\Delta_x$ and $\nabla_x$ in $D^{\mix}_{Wh, I}(B \backslash G )$ where the Artin-Schreier sheaf $\mathcal{L}_{\psi}$ is used in lieu of $\underline{\Bbbk}$, see \cite[\S 6.2]{AMRW19}.
	The objects $\Delta$ (resp. $\nabla$) are standard (resp. costandard) objects in the perverse $t$-structures on each of the preceding categories \cite[\S4.3]{AR16b}.
	\par 	
	The categories $D^{\mix}(B \backslash G / B )$, $D^{\mix}(B \backslash G / U )$ and $D^{\mix}_{Wh,I}(B \backslash G)$ are related by the functors
	\begin{align*}
		\For: D^{\mix}(B \backslash G / B ) \longrightarrow D^{\mix}(B \backslash G / U ),
		&&
		Av: D^{\mix}(B \backslash G / U ) \longrightarrow D^{\mix}_{Wh,I}(B \backslash G),
	\end{align*}
	where $\For$ is the functor forgetting $B$-equivariance and $Av$ was introduced in Section \ref{Ssec: Geo Quotient Construction}. 
	The former is $t$-exact and satisfies  $\For(\Delta_x^{B}) \cong \Delta_x^U$ for all $x \in W$ \cite[\S 3.5]{AR16b}. 
	The latter satisfies $Av(\Delta_x^U) \cong \Delta_x$ for all $x \in W^I$ by \cite[\S6.2]{AMRW19}. 
	Analogous statements hold for each $\nabla_x^B$, $\nabla_x^U$ and $\nabla_x$. 
	Moreover, $\For$ and $Av$ each commute with the action (induced by convolution) of $D^{\mix}(B \backslash G / B )$ on $D^{\mix}(B \backslash G / U )$ and $D^{\mix}_{Wh,I}(B \backslash G)$ respectively, see \cite{AR16b} and \cite[\S6.2]{AMRW19}. 
	We denote this action by $\ustar$. 
	\par 
	We note the following property of $\nabla_x^{U}$. 
	\begin{lem}
	\label{Lem: costandard image}
		For any $x \in W^I$, we have $Av(\nabla_{xw_I}^U) \cong \nabla_x \langle \ell(w_I) \rangle $.
	\end{lem}
	\begin{proof}
		For any $s \in S$, $\nabla_s^B$ is represented by a complex with $\cE_{\id}(-1)$ in cohomological degree $-1$, and $\cE_s$ in cohomological degree $0$. 
		Recall $\nabla_{ys}^B \cong \nabla_{y}^B \ustar \nabla_{s}^B  $ whenever $xs>x$ \cite[\S 4.3]{AR16b}.
		Basic Kazhdan-Lusztig cell theory implies $\langle \cE_y \vert y \neq \id \rangle_{\oplus, (1)}$ forms a tensor ideal in $\Parity (B \backslash G/B)$.
		Hence $\nabla_{w_I}^B$ is represented by a complex with a unique summand $\cE_{\id}$ in cohomological degree $-\ell(w_I)$ and internal degree $-\ell(w_I)$.
		The claim follows from $\For(\nabla_x^B) \cong \nabla_x^U$ as in \cite[\S3.5]{AR16b}, $Av$ satisfying $Av (\cE_y) \cong 0$ if $y \notin w^I$ \cite[\S11.2]{RW18} and $\ustar$ commuting with $\For$ and $Av$ \cite[\S 6.2]{AMRW19}. 
	\end{proof}

\subsection{Mixed parabolic Ringel duality}
\label{Ssec: Ringel duality}

	We now construct mixed, parabolic Ringel duality functors.
	Throughout we require $\Bbbk$ is a field. 
	\par 
	Recall Ringel duality for $D^{\mix}(B \backslash G /U )$ is an autoequivalence $\tR$ satisfying $\tR(T_{x}) \cong P_{ w_0 x}$. 
	In \cite[\S 4.4]{AR16b} it is defined as $\tR(A) \cong \Delta_{w_0}^B \ustar A$.
	An essential property of their construction is $\tR(\nabla_x^U) \cong \Delta_{w_0x}^U$, which categorifies $\delta_{w_0} \overline{\delta_x} = \delta_{w_0 x}$ in the Hecke algebra. 
	At the level of Bruhat orders, this descends to the antiinvolution $x \mapsto w_0 x$ on $W$.  
	The parabolic analogue is the antiinvolution $x \mapsto w_0 x w_I$ on $W^I$, and for any $x \in W^I$ we have the identity in $N^I$
	\begin{align*}
		(-v^{-1})^{\ell(w_I)} \, \delta_{w_0} \cdot (\overline{\delta_x} \otimes 1)
		=
		\delta_{w_0} \overline{\delta_x}  \otimes  (-v^{-1})^{\ell(w_I)} 
		=
		\delta_{w_0} \overline{\delta_x} \,  \overline{\delta_{w_I}}  \otimes 1
		=
		\delta_{w_0 x w_I} \otimes 1.
	\end{align*}
	Motivated by this identity we define 
	\begin{align*}
		\tR_A : D^{\mix}_{Wh,I}(B \backslash G) \longrightarrow D^{\mix}_{Wh,I}(B \backslash G)
		~~\text{ where }~~
		C \mapsto \Delta_{w_0}^B \langle \ell(w_I) \rangle \ustar C.
	\end{align*}	
		
	\begin{lem}
	\label{Lem: Ringel on costandard}
		For any $x \in W^I$ we have $\tR_{A}(\nabla_x) \cong \Delta_{w_0 x w_I}$. 
	\end{lem}
	\begin{proof}
		The claim follows from Lemma \ref{Lem: costandard image} and the fact $Av$ and $\For$ commute with $\ustar$, see \cite[\S 6.2]{AMRW19}. 
		In particular, for any $x \in W^I$ we have
		\begin{align*}
			\Delta_{w_0}^B \langle \ell(w_I) \rangle \ustar  \nabla_x
			&\cong
			\Delta_{w_0}^B \ustar  \nabla_x\langle \ell(w_I) \rangle
			\\&\cong 
			\Delta_{w_0}^B \ustar  \left(Av \circ \For ( \nabla_{xw_I}^B) \right)
			\\&\cong
			Av \circ \For (\Delta_{w_0}^B \ustar \nabla_{xw_I}^B)
			\\&\cong
			Av \circ \For (\Delta_{w_0 x w_I}^B)
			\\&\cong
			\Delta_{w_0 x w_I},
		\end{align*}
		which proves the claim.
	\end{proof}
	
	\begin{prop}
	\label{Prop: Antispherical Ringel}
		The functor $\tR_A : D^{\mix}_{Wh,I}(B \backslash G) \rightarrow D^{\mix}_{Wh,I}(B \backslash G)$ is an equivalence of categories with quasi-inverse $\tR_A^{-1}(C) = \nabla_{w_0}^B\langle  - \ell(w_I) \rangle \ustar C$.
		Moreover, $\tR_A$ satisfies $\tR_A(T_{x}) \cong P_{w_0 x w_I}$ for each $x \in W^I $.
	\end{prop}
	
	Our proof follows the arguments of \cite[\S2.3]{BBM04}, \cite[\S4.2.1]{Yun09} and \cite[\S10.2]{ARV20}.
	
	\begin{proof}
		That $\tR_A$ is an equivalence with quasi-inverse $\tR_A^{-1}$ is immediate from $\Delta_x^B \ustar \nabla_x^B \cong \underline{\Bbbk}_{X_{\id}}$ \cite[\S4.3]{AR16a} and the fact that $\ustar$ commutes with $\For$ and $Av$. 
		It remains to show $\tR_A$ has the desired effect on tilting objects. 
		\par 
		Fix $x \in W^I$, let $T_x$ denote the corresponding indecomposable tilting and set $P = \tR_A(T_x)$. 
		Lemma \ref{Lem: Ringel on costandard} and the fact $T_x$ is tilting imply $P$ is perverse and has a $\Delta$-filtration.  
		It is shown in \cite[\S6.2]{AMRW19} that 
		\begin{align*}
			\Hom_{D^{\mix}_{Wh,I}(B \backslash G)}(\Delta_x , \nabla_y \langle n \rangle [m]) 
			= 
			\begin{cases}
				\Bbbk & \text{if } x=y \text{ and } n=m=0,
				\\
				0 & \text{otherwise.}
			\end{cases}
		\end{align*}
		Hence, 
		\begin{align*}
			\Hom_{D^{\mix}_{Wh,I}(B \backslash G)}(P, \Delta_y \langle n \rangle[m]) 
			\cong
			\Hom_{D^{\mix}_{Wh,I}(B \backslash G)}(T_x , \nabla_{w_0y w_I} \langle n \rangle[m]) 
		\end{align*}
		can only be non-zero when $m=0$.
		Since $D^{\mix}_{Wh,I}( B \backslash G)^{\leq 0}$ is generated under extensions by $\Delta_x \langle m \rangle [n]$, where $x \in W$, $m \in \Z$ and $n \in \Z_{\geq 0}$, it follows for any $C \in D^{\mix}_{Wh,I}( B \backslash G)^{\leq 0}$ that
		\begin{align*}
			\Hom_{D^{\mix}_{Wh,I}(B \backslash G)}(P, C) = 0 
		\end{align*}
		and consequently, for any $C'$ in $\Pmixw(B \backslash G)$ we have 
		\begin{align*}
			\text{Ext}^1_{\Pmixw(B \backslash G)} (P, C') =0.
		\end{align*}
		Hence $P$ is projective. 
		Since $\tR_A$ is an equivalence and $T_x$ indecomposable, $\Endo(P) \cong \Endo(T_x)$ is a local ring, so $P$ is indecomposable. 
		Finally, the kernel of $T_x \twoheadrightarrow \nabla_x$ admits a $\nabla$-filtration. 
		Hence we obtain a surjections $P \twoheadrightarrow \nabla_{w_0 x w_I} \twoheadrightarrow L_{w_0 x w_I}$. 
		Thus, $P \cong P_{w_0 x w_I}$. 
	\end{proof}

	We now construct mixed Ringel duality for $D^{\mix}(U \backslash G / P)$. 
	Since $D^{\mix}(U \backslash G / P)$ is not a module category for $D^{\mix}(B \backslash G / B)$ our construction is less direct than desired. 
	
	\begin{prop}
		Suppose $2 \in \Bbbk^{\times}$. Define $\tR_S: D^{\mix}(U \backslash G / P) \rightarrow D^{\mix}(U \backslash G / P)$ as the composition $R_S = \kappa \circ \tR_A \circ \kappa$. 
		Then $\tR_S (T_{x}) \cong P_{w_0 x w_I}$ for any $x \in W^I$, and $\tR_S$ has quasi-inverse $\tR_S^{-1} = \kappa \circ \tR_A^{-1} \circ \kappa$.
	\end{prop}
	\begin{proof}
		Parabolic Koszul duality and Lemma \ref{Lem: Ringel on costandard} imply $\tR_S(\nabla_x^P) \cong \Delta_{w_0 x w_I}^P$ for all $x \in W^I$. 
		The claim then follows from arguments similar to that in Proposition \ref{Prop: Antispherical Ringel}. 
	\end{proof}
	
	\begin{rem}
		A more direct construction using the `free monodromic category' of \cite[\S2.3]{AMRW19} may be possible. 
		However, it is currently unknown if the free monodromic analogue of $\Delta_{w_0}$ is convolutive. 
	\end{rem}

\subsection{Extension of scalars and modular reduction}
\label{Ssec: Extension of scalars}
	
	We now recall collect facts regarding the behaviour of various sheaves under extension of scalars and modular reduction. 
	These results can all be found in either \cite{AR16b} or \cite{AMRW19}. 
	\par 
	Fix a $p$-modular system $(\bO, \bK, \bF)$.
	For $\Bbbk \in \{\bO, \bK, \bF\}$ we let ${}^{\Bbbk} \check{\cE}_x$ denote the indecomposable parity sheaf in $D^{\mix}_{Wh,I} (\check{B} \backslash \check{G}, \Bbbk)$, we take ${}^{\Bbbk}\cE_x , {}^{\Bbbk}T_{x}, {}^{\Bbbk}P_x, {}^{\Bbbk}L_x, {}^{\Bbbk}\Delta_x, {}^{\Bbbk}\nabla_x$ to be the usual objects in $D^{\mix}(U \backslash G / P, \Bbbk)$.
	Note that the existence of these objects when $\Bbbk= \bO$ is shown in \cite{AR16b}. 
	In Section \ref{Ssec: Decomposition numbers} we will also consider the non-mixed indecomposable tilting ${}^{\Bbbk}\cT_x$, projective ${}^{\Bbbk}\cP_x$, and simple intersection cohomology $\mathcal{IC}_x$ sheaves in $D^b(U \backslash G / P, \Bbbk)$, and the corresponding sheaves ${}^{\Bbbk}\cT_x^B$, ${}^{\Bbbk}\cP_x^B$, and $\mathcal{IC}_x^B$ in $D^b(U \backslash G / B, \Bbbk)$.
	\par 
	For mixed sheaves with coefficients in $\bO$, we consider the functors 
	\begin{align*}
		&\bK (-) : D_{Wh,I}^{\mix} (B \backslash G , \bO) \rightarrow D_{Wh,I}^{\mix} (B \backslash G , \bK)
		&&
		\bF (-) : D_{Wh,I}^{\mix} (B \backslash G , \bO) \rightarrow D_{Wh,I}^{\mix} (B \backslash G , \bF)
		\\
		&\bK (-) : D^{\mix} (U \backslash G / P , \bO) \rightarrow D^{\mix} (U \backslash G / P , \bK)
		&&
		\bF (-) : D^{\mix} (U \backslash G /P , \bO) \rightarrow D^{\mix} (U \backslash G / P , \bF)
	\end{align*}
	each of which is induced by the derived extension of scalars functors.
	These functors commute with Koszul duality by \cite[\S5.3]{AR16b}.
	\par 
	Our interest lies in how various sheaves behave under extension of scalars $\bK(-)$ and modular reduction $\bF(-)$. 
	Standard and costandard objects are well-behaved under both functors
	\begin{align*}
		\bF({}^{\bO} \Delta_x) \cong {}^{\bF} \Delta_x
		&&
		\bF({}^{\bO} \nabla_x) \cong {}^{\bF} \nabla_x
		&&
		\bK({}^{\bO} \Delta_x) \cong {}^{\bK} \Delta_x
		&&
		\bK({}^{\bO} \nabla_x) \cong {}^{\bK} \nabla_x
	\end{align*}
	by \cite[\S3.1]{AR16b}. 
	Parity, tilting and projective sheaves remain indecomposable under modular reduction
	\begin{align*}
		\bF({}^{\bO}\check{\cE}_x) \cong {}^{\bF}\check{\cE}_x
		&&
		\bF({}^{\bO} \cE_x) \cong {}^{\bF}\cE_x
		&&
		\bF({}^{\bO} T_x) \cong {}^{\bF}T_x
		&&
		\bF({}^{\bO} P_x) \cong {}^{\bF} P_x
	\end{align*}
	by \cite[\S3.1]{AR16b}. 
	While extension of scalars $\bK(-)$ takes parity sheaves to parity sheaves \cite[\S2.5]{JMW14}, mixed tilting sheaves to mixed tilting sheave, and mixed projective sheaves to mixed projective sheaves \cite[\S3.3]{AR16b}.
	In particular, we define integers $\check{e}_{y,x}^i$, $e_{y,x}^i$, $t_{y,x}^i$ and $p_{y,x}^i$ such that 
	\begin{align*}
		\bK({}^{\bO}\check{\cE}_x) 
		&\cong 
		\bigoplus_{i \in \Z, y \in W^I} {}^{\bK}\check{\cE}_x(i)^{\oplus \check{e}_{y,x}^i},
		&
		\bK({}^{\bO} \cE_x) 
		&\cong 
		\bigoplus_{i \in \Z, y \in W^I} {}^{\bK}\cE_x(i)^{\oplus e_{y,x}^i},
		\\
		\bK({}^{\bO} T_x) 
		&\cong 
		\bigoplus_{i \in \Z, y \in W^I} {}^{\bK}T_x \langle i \rangle^{\oplus t_{y,x}^i},
		&
		\bK({}^{\bO} P_x) 
		&\cong 
		\bigoplus_{i \in \Z, y \in W^I} {}^{\bK} P_x \langle i \rangle^{\oplus p_{y,x}^i}.
	\end{align*}
	Finally, in \cite[\S3.1]{AR16b} it is shown that $\bK({}^{\bO}L_x) \cong {}^{\bK}L_{x}$ and $\bF({}^{\bO}L_x)$ is perverse. 
	Since $P^{\mix}(U \backslash G/P)$ is finite-length, the Jordan-H\"{o}lder multiplicities $l_{y,x}^i := [\bF({}^{\bO}L_x) : {}^{\bF}L_y \langle i \rangle]$ are well defined.

\subsection{Graded decomposition numbers}
\label{Ssec: Decomposition numbers}
	
	In Section \ref{Ssec: Extension of scalars} we introduced the graded decomposition numbers $\check{e}_{y,x}^i$, $t_{y,x}^i$, $p_{y,x}^i$ and $l_{y,x}^i$. 
	We now relate these decomposition numbers. 
	Define Laurent polynomials 
	\begin{align*}
		\check{e}_{y,x}^{\mix} = \sum_{i \in \Z} \check{e}_{y,x}^i v^i ,
		&&
		t_{y,x}^{\mix} = \sum_{i \in \Z} t_{y,x}^i v^i,
		&&
		p_{y,x}^{\mix} = \sum_{i \in \Z} p_{y,x}^i v^i,
		&&
		l_{y,x}^{\mix} = \sum_{i \in \Z} l_{y,x}^i v^i,
	\end{align*}
	and consider the matrices
	\begin{align*}
		\check{\mathfrak{E}}_{I}^{\mix} = (\check{e}_{y,x}^{\mix} ),
		&&
		\mathfrak{T}_{I}^{\mix} = (t_{y,x}^{\mix}),
		&&
		\mathfrak{P}_{I}^{\mix} = ( p_{y,x}^{\mix}),
		&&
		\mathfrak{L}_{I}^{\mix} = ( l_{y,x}^{\mix}),
	\end{align*}
	where $y,x \in W^I$. 
	Given any matrix $\mathfrak{M}= (m_{y,x})$ where $y,x \in W^I$ we define $\tR\mathfrak{M}$ to be the matrix $\tR\mathfrak{M} = (m_{w_0 y w_I, w_0 x w_I})$, and $\mathfrak{M}^t$ to be the transpose matrix.
	\par 
	The following relationship between the decomposition numbers is a direct generalisation of \cite[\S2.7]{AR16a} to the mixed and parabolic settings. 
	\begin{prop}
\label{Prop: mixed decomp numbers}
		For any $I$ and $p>2$ we have
		\begin{align*}
			\check{\mathfrak{E}}_{I}^{\mix}
			=
			\mathfrak{T}_{I}^{\mix}
			=
			\tR \mathfrak{P}_{I}^{\mix}
			=
			\tR \mathfrak{L}_{I}^{\mix,t}.
		\end{align*}
	\end{prop}
	
	We prove the Proposition in the following lemmata. 
	
	\begin{lem}
	\label{Lem: parity tilt decomp}
		For any $I \subseteq S$ and $p>2$, we have $\check{\mathfrak{E}}_I^{\mix} =  \mathfrak{T}_I^{\mix}$.
	\end{lem}
	\begin{proof}
		Extension of scalars gives rise to the identity
		\begin{align*}
			\ch [{}^{\bO} T_x] = \sum_{i \in \Z, y \in W^I} t_{y,x}^i [{}^{\bK} T_y] (-v^{-1})^i.
		\end{align*}
		On the other hand, parabolic Koszul duality and extension of scalars shows
		\begin{align*}
			\ch [{}^{\bO} T_x]
			=
			\kappa ( \ch [{}^{\bO} \check{\cE}_x])
			=
			\kappa \left( \sum_{i \in \Z, y \in W^I} \check{e}_{y,x}^i \ch [{}^{\bK} \check{\cE}_y] v^i \right)
			=
			 \sum_{i \in \Z, y \in W^I} \check{e}_{y,x}^i \ch [{}^{\bK} T_y] (-v^{-1})^i.
		\end{align*}
		Equating coefficients gives $t_{y,x}^i = \check{e}_{y,x}^i$ which proves the claim. 
	\end{proof}
	
	\begin{lem}
	\label{Lem: tilt proj decomp}
		For any $I \subseteq S$ and $p>2$, we have $\mathfrak{T}_I^{\mix} = \tR\mathfrak{P}_I^{\mix}$. 
	\end{lem}
	\begin{proof}
		First, note that mixed spherical Ringel duality $\tR_S$ implies
		\begin{align*}
			\ch[{}^{\bF}P_{w_0 y w_I}] = v^{\ell(w_I)} \delta_{w_0} \ch[{}^{\bF}T_{x}],
			&&
			\ch[{}^{\bK}P_{w_0 y w_I}] = v^{\ell(w_I)} \delta_{w_0} \ch[{}^{\bK}T_{x}].
		\end{align*}
		Modular reduction gives $\ch[{}^{\bF} T_x] = \ch [{}^{\bO} T_x]$ and $\ch[{}^{\bF} P_x] = \ch [{}^{\bO} P_x]$.
		Combining these facts, we obtain 
		\begin{align*}
			\ch[{}^{\bO} P_{w_0 x w_I} ]
			&=
			(-v^{-1})^{\ell(w_I)} \delta_{w_0} \ch[{}^{\bO}T_{x}]
			\\&=
			\sum_{i \in \Z, y \in W^I} t_{y,x}^i(-v^{-1})^{\ell(w_I)} \delta_{w_0} \ch [{}^{\bK}T_y]
			\\&=
			\sum_{i \in \Z, y \in W^I} t_{y,x}^i  \ch [{}^{\bK}P_{w_0y w_I}].
		\end{align*} 
		Comparing coefficients with $\ch[\bK({}^{\bO}P_{w_0 x w_I})]$ proves the claim. 
	\end{proof}
	
	\begin{lem}
		For any $I \subseteq S$, we have $\mathfrak{R}_I^{\mix} = \mathfrak{L}_I^{\mix,t}$. 
	\end{lem}
	
	This argument is a modification of that in \cite[\S 2.7]{AR16a}.
	
	\begin{proof}
		The arguments of \cite[\S2.4]{RSW14} and \cite[\S5.2]{AR16a} appply in the mixed setting.
		In particular, for any $x,y \in W^I$ and $i \in \Z$ the space $\Hom_{D^{\mix}(U \backslash G/ P, \bO)}({}^{\bO} P_{x}\langle i \rangle , {}^{\bO} L_y)$ is a free $\bO$-module, and we have isomorphisms of $\bF$-vector spaces 
		\begin{align*}
			\bF \otimes_{\bO} \Hom_{D^{\mix}(U \backslash G/ P, \bO)} ({}^{\bO} P_{x}\langle i \rangle , {}^{\bO} L_y) 
			\cong 
			 \Hom_{D^{\mix}(U \backslash G/ P, \bF)} ({}^{\bF} P_{x}\langle i \rangle , \bF( {}^{\bO} L_y)).
		\end{align*}
		Hence $ \rank_{\bO} \Hom_{D^{\mix}(U \backslash G/ P, \bO)} ({}^{\bO} P_{x}\langle i \rangle , {}^{\bO} L_y)  = [\bF( {}^{\bO} L_y) :  {}^{\bF} L_x \langle i \rangle] $.
		On the other hand, we have isomorphisms 
		\begin{align*}
			\bK \otimes_{\bO} \Hom_{D^{\mix}(U \backslash G/ P, \bO)} ({}^{\bO} P_{x}\langle i \rangle , {}^{\bO} L_y)
			&\cong
			\Hom_{D^{\mix}(U \backslash G/ P, \bK)} (\bK ({}^{\bO} P_{x}) , {}^{\bK} L_y \langle -i \rangle)
			\\
			&\cong 
			\Hom_{D^{\mix}(U \backslash G/ P, \bK)} \left(\bigoplus_{j \in \Z, z \in W^I} {}^{\bK} P_{z} \langle j \rangle^{\oplus p_{z,x}^j} , {}^{\bK} L_y \langle -i \rangle \right)
			\\
			&\cong
			\Hom_{D^{\mix}(U \backslash G/ P, \bK)} ({}^{\bK} P_{y} \langle -i \rangle^{\oplus p_{y,x}^{-i}} , {}^{\bK} L_y \langle -i \rangle).
		\end{align*}
		Hence $\rank_{\bO} \Hom_{D^{\mix}(U \backslash G/ P, \bO)} ({}^{\bO} P_{x}\langle i \rangle , {}^{\bO} L_y) = p_{y,x}^{-i}$. 
		The self-duality of parity sheaves \cite[\S2.2]{JMW14} together with Lemmata \ref{Lem: parity tilt decomp} and \ref{Lem: tilt proj decomp} imply $p_{y,x}^{-i} = p_{y,x}^i$. 
		Hence the claim.  
	\end{proof}

	Let $\mathfrak{T}_{I} =(t_{y,x})$, $\mathfrak{P}_{I}=(p_{y,x})$ and $\mathfrak{L}_{I} = (l_{y,x})$ denote the corresponding matrices of decomposition numbers for the (unmixed) sheaves ${}^{\bO} \cT_x$, ${}^{\bO}\cP_x$ and ${}^{\bO}\mathcal{IC}_x$ in $D^b(U \backslash G / P)$ respectively. 
	We have the following (unmixed) analogue of \cite[\S2.7]{AR16a}.
	
	\begin{lem}
\label{Lem: ungraded decomp numbers}
		For any $I \subseteq S$ and any $p \geq 0$, we have
		\begin{align*}
			\mathfrak{T}_{I}
			=
			\tR \mathfrak{P}_{I}
			=
			\tR \mathfrak{L}_{I}^{t}.
		\end{align*}
	\end{lem}
	
	The following sketches how the argument in \cite[\S2.7]{AR16a} can be extended to the parabolic setting. 
	
	\begin{proof}
		In \cite[\S2.3]{AR16a} (unmixed) Ringel duality is constructed using the Radon transform for flag varieties. 
		In \cite[\S5.4]{Yun09} a Radon transform is constructed for partial flag varieties, and the arguments of \cite[\S B.5]{AR16a} continue to hold in this setting. 
		So $\mathfrak{T}_{I}=\tR \mathfrak{P}_{I}$ follows from the geometric construction of Ringel duality and its compatibility with extension of scalars.
		Finally, $\mathfrak{P}_{I}= \mathfrak{L}_{I}^t$ by arguments identical to those in \cite[\S2.7]{AR16a}.
	\end{proof}
	
	Given some matrix $\mathfrak{M}$ with entries in $\Z[v,v^{-1}]$, we denote by $\mathfrak{M}|_{v=1}$ the matrix obtained by the evaluation $v=1$. 
	
	\begin{lem}
\label{Lem: Degrading decomp numbers}
		For any $I\subseteq S$ and any good prime $p$, we have 
		\begin{align*}
			\mathfrak{T}_{I}^{\mix}|_{v=1}
			=
			\mathfrak{T}_{I}
			=
			\tR \mathfrak{P}_{I}
			=
			\tR \mathfrak{L}_{I}^{t}.
		\end{align*}
	\end{lem}
	\begin{proof}
		Denote by $t_{y,x}^{i, B}$ the graded decomposition numbers for $\bK ({}^{\bO}T_{x}^{B})$ and $t_{y,x}^{B}$ the (ungraded) decomposition numbers for $\bK ({}^{\bO}\cT_{x}^{B})$.
		Since $p$ is a good prime for $G$, we can apply the degrading functor of \cite[\S 5.3]{AR16b}. 
		Namely, there is a $t$-exact functor $\mu: D^{\mix}(U \backslash G / B) \rightarrow D^{b}(U \backslash G / B)$ satisfying (among other things) $\mu \circ \langle 1 \rangle \cong \mu$ and $\mu (T_x^B) \cong \cT_{x}^B$ for each $x \in W$. 
		Hence, we obtain 
		\begin{align*}
			t_{y,x}^{B} = \sum_i t_{y,x}^{i, B}.
		\end{align*}
		The natural map $p: G/B \rightarrow G/P$ induces functors on $p_* : D^{\mix}(U \backslash G / B) \rightarrow D^{\mix}(U \backslash G / P)$ and $p_* : D^{b}(U \backslash G / B) \rightarrow D^{b}(U \backslash G / P)$ satisfying
		\begin{align*}
			p_* (T_x^B)
			\cong 
			\begin{cases}
				T_x &\text{if } x \in W^I,
			\\
				0 & \text{otherwise},
			\end{cases}
			&&
			p_* (\cT_x^B)
			\cong 
			\begin{cases}
				\cT_x &\text{if } x \in W^I,
			\\
				0 & \text{otherwise},
			\end{cases}
		\end{align*}
		respectively, by \cite[\S6.3]{AMRW19} and \cite[\S3.4]{Yun09}.
		These imply that for any $x,y \in W^I$ we have 
		\begin{align*}
			t_{y,x} = t_{y,x}^B = \sum_{i \in \Z} t_{y,x}^{i, B} = \sum_{i \in \Z} t_{y,x}^{i} = t_{y,x}^{\mix}|_{v=1}
		\end{align*}
		where $t_{y,x}$ denotes the relevant decomposition number for $\bK({}^{\bO} \cT_{x})$.
		Hence $\mathfrak{T}_{I}^{\mix}|_{v=1}=\mathfrak{T}_{I}$. 
		The remaining equalities follow from Lemma \ref{Lem: ungraded decomp numbers}.
	\end{proof}
	
	Recall that $e_{y,x}^i$ denotes the relevant decomposition number for $\bK({}^{\bO}\cE_x)$ in $D^{\mix}(U \backslash G/P, \bK)$. 
	Define the Laurent polynomial $e^{\mix}_{y,x}$ and the matrix $\mathfrak{E}$ by 
	\begin{align*}
		e^{\mix}_{y,x} = \sum_{i} e_{y,x}^i v^i,
		&&
		\mathfrak{E}_I = (e^{\mix}_{y,x})
	\end{align*}
	respectively. 
	Note these should not to be confused with $\check{e}^{\mix}_{y,x}$ and $\check{\mathfrak{E}}_I$ which encode the decomposition numbers for parity objects in $D^{\mix}_{Wh,I}(B \backslash G)$.
	\par 
	We will need the following parabolic analogue of \cite[\S2.2]{Wil15}. 
	
	\begin{lem}
\label{Lem: IC parity decomp}
		Fix $x \in W^I$ and suppose $y<x$  is maximal in the Bruhat order such that $e_{y,x} \neq 0$. If $e_{y,x} \in \Z$ then $l_{y,x}=e_{x,y}$.
	\end{lem}
	\begin{proof}
		The arguments in \cite[\S2.2]{Wil15} hold if one instead uses $X = \bigsqcup_{z \geq x} B\cdot xP/P$, $Z = B \cdot xB$ and $U = X \backslash Z$.\footnote{There is a typographic error in \cite[\S2.2]{Wil15}: $U = X \backslash U$ should be $U = X \backslash Z$.} 
	\end{proof}
	
	\begin{prop}
\label{Prop: Parity parity decomp}
		Let $I \subseteq S$ be arbitrary and $p>2$ be a good prime for $G$.
		If $\check{\mathfrak{E}}$ is the identity matrix, then $\check{\mathfrak{E}} = \mathfrak{E}$.
	\end{prop}
	\begin{proof}
		By Proposition \ref{Prop: mixed decomp numbers}, Lemma \ref{Lem: Degrading decomp numbers} and the assumptions on $p$, we have $\check{\mathfrak{E}}_{I}^{\mix}|_{v=1} = \mathfrak{L}_{I}$. 
		Hence $\mathfrak{L}_{I}$ is the identity matrix.
		Now, suppose for a contradiction that $x \in W^I$ is minimal with respect to the Bruhat order such that $\bK({}^{\bO}\cE_x) \not\cong {}^{\bK}\cE_x$.  
		Denote by ${}^{\bO, B} \cE_x$ the corresponding indecomposable sheaf in $D^{\mix}(B \backslash G / P,\bO)$ and note that ${}^{\bO}\cE_{x}$ is the image of ${}^{\bO, B} \cE_x$ under the functor forgetting $B$-equivariance, \cite[\S4.1]{JMW14}.
		Then, for any $s \in S$ such that $sx<x$, we have
		\begin{align*}
			\bK({}^{\bO}\cE_s^B \star {}^{\bO, B} \cE_{sx})
			\cong
			{}^{\bK}\cE_s^B \star {}^{\bK, B} \cE_{sx}
			\cong
			{}^{\bK, B} \cE_{x}
			\oplus
			\bigoplus_{sy\leq y <x} {}^{\bK, B} \cE_{y}^{\oplus \mu_{y,x}^{m}}
		\end{align*}
		where $sy\leq y$ means either $sy<y$ or $sy \notin W^I$, and $\mu_{y,x}^{m}$ denotes the coefficient of $v$ in the spherical Kazhdan-Lusztig polynomial $m_{y,x}$. 
		In particular, all summands in the decomposition are in degree 0. 
		Consequently every summand of $\bK({}^{\bO}\cE_x)$ must also be in degree 0, i.e. $e_{y,x} \in \Z$. 
		Hence we can apply Lemma \ref{Lem: IC parity decomp} to obtain a non-zero $l_{y,x}$, contradicting the fact $\mathfrak{L}$ is the identity matrix. 
	\end{proof}
	
	\begin{rem}
		If $\check{\mathfrak{E}}_I$ is not the identity matrix then the relationship between $\check{\mathfrak{E}}_I$ and $\mathfrak{E}_I$ is highly opaque. 
		For example, if $G=SO_{5}$, $P$ is a parabolic of type $\tA_{2}$, and $p=2$ then $\mathfrak{E}$ has entries in $\{ 0,1\}$, while $\check{\mathfrak{E}}$ has entries in $\Z[v,v^{-1}]$.
	\end{rem}

\section{Spherical (co)minuscule Hecke categories}
\label{Sec: Spherical}

	We now deduce the $p$-Kazhdan-Lusztig basis of the spherical (co)minuscule Hecke categories from our knowledge of the antispherical categories and decomposition numbers. 

	\begin{thm}
		Suppose $\Bbbk$ is a field or a complete local ring in a $p$-modular system, $D^{\mix}(U \backslash G / P, \Bbbk)$ is (co)minuscule, and all primes which are not good for $G$ are invertible in $\Bbbk$. 
		The $p$-Kazhdan-Lusztig basis of $D^{\mix}(U \backslash G / P)$  satisfies ${}^p c_x = c_x$ for all $x \in W^I$. 
	\end{thm}
	
	\begin{proof}
		If $p>2$ then the claim is immediate from Theorems \ref{Thm: antispherical cominuscule pKL basis} and \ref{Thm: antispherical minuscule pKL basis}, and Proposition \ref{Prop: Parity parity decomp}.
		The only remaining case is $p=2$ in type $\tA$.
		However, (co)minuscule flag varieties in type $\tA$ are all Grassmannians, for which it is known that each Schubert variety $X \subseteq G/P$ admits a small resolution $\pi: \tilde{X} \rightarrow X$, see \cite{Zel83}. 
		In particular $\pi_* \underline{\Bbbk}_{\tilde{X}}$ is a simple intersection cohomology sheaf, so it is also an indecomposable parity sheaf. 
		Hence ${}^p c_x = c_x$ for all $x \in W^I$.  
	\end{proof}

	\par 
	When all primes which are not good for $G$ are not invertible in $\Bbbk$ we may have ${}^p c_x \neq c_x$. 
	The most interesting behaviour arises from the Lagrangian Grassmannian.
	The elements in $W^I$ can be identified with subsets of $\{ 1 , \dots , n \}$ using the map in Section \ref{Ssec: min coset reps LG} and inversion in $W$. 
	For each $x \in {}^I W$, define the set
		\begin{align*}
			F(x) 
			:=
			\{
			y \subseteq x
			~\vert~
			\text{if } i,i-1 \in x 
			\text{ and } i \notin y
			\text{ then } i-1 \notin y 
			\}.
		\end{align*}
	Further, for each $s \in S$ and $x \in W^I$ define ${}^ p \mu_{s,x}^y \in \Z[v,v^{-1}]$ by the equality 
	\begin{align*}
		b_s ~{}^p c_x = \sum_{y \in W} {}^p \mu_{s,x}^y ~ {}^p c_y
	\end{align*}
	in $M^I$. 
	We propose the following combinatorial algorithm for determining ${}^2 c_x$ for the Lagrangian Grassmannian; i.e. when $\h$ is a Cartan realisation of type $\tC_n$ and $I\subseteq S$ is type $\tA_{n-1}$. 
				
	\begin{conj}
	\label{conj: spherical p2} 
		Fix $x \in W$ and let $i> 0$.  If $s_i x < x$, then  
		\begin{align*}
			{}^2 c_{x}
			=
			b_{s_i} \, {}^2 c_{s_i x}  
			-
			\sum_{ y \notin  F(x)} {}^2\mu^{y}_{s_i, s_i x} ~ {}^2c_y. 
		\end{align*}
	\end{conj}
		
	\begin{rem}
		If the left descent set of $x$ is $\{ s_0 \}$ then Equation \ref{eqn: action on sets} implies  $x = \{ k, \dots , 1 \}$ for some $k\leq n$, and the corresponding Schubert variety is smooth, by \cite[\S1]{KP18}.
		Hence ${}^2 c_{\{ k, \dots , 1 \}} = c_{\{ k, \dots , 1 \}}$.
		For all other $x \in W^I$ the element ${}^2 c_x$ can be computed using Conjecture \ref{conj: spherical p2}. 
	\end{rem}
	\begin{rem}
		Conjecture \ref{conj: spherical p2} has been checked up to rank 7 using the IHecke code of \cite{GJW23}. 
	\end{rem}
	\begin{rem}
		In general, one does not expect a combinatorial algorithm will compute $p$-Kazhdan-Lusztig bases. 
		However, it is clear (only after the fact) from Theorems \ref{Thm: antispherical cominuscule pKL basis} and \ref{Thm: antispherical minuscule pKL basis} that $p$-Kazhdan-Lusztig bases for antispherical (co)minuscule Hecke categories can be computed using a combinatorial algorithm; they either arise from the Kazhdan-Lusztig algorithm, or they are monomials in the Kazhdan-Lusztig generators. 
		Thus, one might hope that a combinatorial algorithm also computes the $p$-Kazhdan-Lusztig bases for spherical (co)minuscule Hecke categories. 
	\end{rem}

\section{$p$-tightness and $p$-small resolutions}
\label{Sec: p-tightness}

	We conclude by introducing the notions of $p$-tight elements for a realisation  and $p$-small resolutions of singularities.
	This is done with a view towards providing a geometric explanation for the characteristic-$2$ phenomena in antispherical cominuscule Hecke categories. 

\subsection{Definitions and basic properties}
\label{Ssec: p-tightness and p-small}

	Let $\Bbbk$ a field of characteristic $p$, or a complete local ring with residue field of characteristic $p$.
	\par 
	Fix a realisation $\h$ of a Coxeter system $(W,S)$, and denote by $\sD_{\text{BE}}(\h, \Bbbk)$ the corresponding Elias-Williamson diagrammatic category. 
	We define an element $x \in W$ to be \ldef{$p$-tight} if ${}^p b_{x} = b_{\ux}$ for some reduced expression $\ux$ of $x$. 
	\begin{rem}
		An element $x \in W$ is called tight if $b_x = b_{\ux}$ for some reduced expression $\ux$ of $x$. 
		These have been studied extensively in \cite{Deo90,Lus93, BW01, BJ07}. 
		When $\Bbbk$ is a field of characteristic 0, $p$-tightness reduces to the classical notion of tightness. 
		Note that unlike tightness, $p$-tightness depends on the realisation rather than the Coxeter system. 
	\end{rem}
	
	When $W$ is the Weyl group of a reductive algebraic group $G$, tight elements have a natural geometric interpretation. 
	An element $x \in W$ is tight if and only if there is some reduced expression $\ux$ of $x$ such that the corresponding Bott-Samelson resolution $\pi_{\ux} :\BS(\ux) \rightarrow X_{x}$ of the Schubert variety $X_x$ is a small resolution \cite[\S 2]{BJ07}.
	\par 
	Consider a stratified variety $Y = \sqcup_{s} Y_s$ satisfying the conditions of \cite[\S2.1]{JMW14} so that the category of parity sheaves on $Y$ is defined. 
	Let $\pi: \tilde{Y}_s \rightarrow \overline{Y_s}$ be a resolution of singularities that is even in the sense of \cite[\S2.4]{JMW14}, then the decomposition theorem for parity sheaves implies $\pi_* \underline{\Bbbk}_{\tilde{Y}_s}$ decomposes as a direct sum of (shifts of) indecomposable parity sheaves. 
	We say the morphism $\pi: \tilde{Y}_s \rightarrow \overline{Y_s}$ is a \ldef{$p$-small resolution of singularities} if $\pi_* \underline{\Bbbk}_{\tilde{Y}_s}$ is indecomposable. 
	If $\Bbbk$ is a complete local ring, then \cite[\S2.5]{JMW14} implies $\pi_* \underline{\Bbbk}_{\tilde{Y}_s}$ remains indecomposable under modular reduction. 
	In particular, a $p$-small resolution $\pi: \tilde{Y}_s \rightarrow \overline{Y_s}$ depends on $p$ rather than the choice of $\Bbbk$.
	
	\begin{rem}
		If $\Bbbk$ is a field of characteristic $0$, then the classical decomposition theorem implies  $\pi_* \underline{\Bbbk}_{\tilde{Y}_s}$ decomposes as a direct sum of (shifts of) simple intersection cohomology sheaves. 
		In particular, if $\pi_* \underline{\Bbbk}_{\tilde{Y}_s}$ is indecomposable then $\pi: \tilde{Y}_s \rightarrow \overline{Y_s}$ is a small resolution. 
	\end{rem}
	\begin{rem}
		In \cite[\S5]{HW23} it is shown that if $Y$ is any irreducible variety and  $\pi : \tilde{Y} \rightarrow Y$  is any resolution of singularities, then $\pi_* \underline{\Bbbk}_{\tilde{Y}}$ contains a canonical indecomposable direct summand $\cE(Y)$ with dense support, which is independent of the choice of resolution. 
		One could more generally define a resolution  $\pi : \tilde{Y} \rightarrow Y$ to be $\Bbbk$-small if $\pi_{*} \underline{\Bbbk}_{\tilde{Y}}$ is indecomposable. 
	\end{rem}

\subsection{Elementary properties of $p$-tight elements and $p$-small resolutions}
	We now note elementary properties of $p$-tight elements and $p$-small resolutions. 
	\par 
	Recall the notions of Soergel bimodules \cite{Soe07} and their generalisations due to Abe \cite{Abe21}. 
	For any Soergel bimodule $B$ we denote by $\grrk B$ its graded rank as an $R$-module. 
	\begin{prop}
		\label{prop: p-tight equiv}
			Let $\Bbbk$ be a field of characteristic $p$ or a complete local ring with residue field of characteristic $p$.
			Let $\h$ be a realisation over $\Bbbk$ of an arbitrary Coxeter system $(W,S)$. 
			The following are equivalent: 
			\begin{enumerate}
				\item[(1)] The element $x \in W$ is $p$-tight for the realisation $\h$;
				\item[(2)] The indecomposable Soergel bimodule $B_x$ has graded rank $\grrk B_x = (v+v^{-1})^{\ell(x)}$; 
				\item[(3)] For every reduced expression $\ux$ of $x$ there is an isomorphism $B_{\ux} \cong B_x$ in $\DBE (\h)$; 
				\item[(4)] For every reduced expression $\ux$ of $x$ we have ${}^p h_{\id , x}  = \sum_{ \ux^e_{\bullet} = \id} v^{\text{df}(e)} $.
			\end{enumerate}
			If $W$ is the Weyl group of a reductive algebraic group $G$ and $\h$ a Cartan realisation, then (1)-(4) are also equivalent to: 
			\begin{enumerate}
				\item[(5)] For any reduced expression $\ux$ of $x$, the Bott-Samelson resolution $\pi_{\ux} : \BS(\ux) \rightarrow X_{x}$ is a $p$-small resolution;
				\item[(6)] For any reduced expression $\chi(\cE_{x}\vert_{\id}) = \chi(\pi_{\ux,*} (\underline{\Bbbk}_{\BS(\ux)}) \vert_{\id})$, where $\chi(\mathcal{F}\vert_z)$ denotes the Euler characteristic of the stalk of $\mathcal{F}$ at $zB/B$. 
			\end{enumerate}
		\end{prop}
	 
	 	\begin{proof}
	 		$(1) \implies (2)$: This is immediate from the fact that $\grrk B_s = v+ v^{-1}$. 
	 		\par 
	 		$(2)$ $\implies$ $(1)$:
	 		Since $B_x$ occurs as a multiplicity-1 direct summand of $B_{\ux}$ for any reduced expression $\ux$ of $x$, we have a coefficient-wise inequality $\grrk B_x \preceq \grrk B_{\ux} = (v+v^{-1})^{\ell(x)}$.
	 		This inequality is strict if and only if $B_{\ux}$ contains a non-trivial direct summand that is not $B_x$.  
	 		\par 
	 		$(2)$ $\implies$ $(3)$:
	 		Suppose, for a contradiction, there exists some reduced expression $\ux$ of $x$ such that $B_{\ux} \not\cong B_x$. 
	 		Then $B_{\ux}$ contains a non-trivial direct summand $B$ where $B \not\cong B_x$. 
	 		As $B$ is non-trivial we know $\grrk B \neq 0$, which would contradict $(2)$. 
	 		\par 
	 		$(3)$ $\implies$ $(1)$:  This is clear.
	 		\par 
	 		$(1) \iff (4)$:  Deodhar's defect formula \cite[\S4]{Deo90} states $b_{\ux} = \sum_{y} \sum_{\ux^e_{\bullet} = y} v^{\text{df}(e)} \delta_y$.
	 		Hence ${}^p h_{y,x} = \sum_{\ux^e_{\bullet} = y} v^{\text{df}(e)}$ if and only if $b_{\ux} = {}^p b_x$. 
	 		By $(2)$ this is independent of the choice of reduced expression. 
	 		\par 
	 		$(1) \iff (5)$: For any expression $\ux =(s_1 , \dots , s_k)$ we have by \cite[\S10.2]{RW18} that $\pi_{\ux,*} \underline{\Bbbk}_{\BS(\ux)} \cong \cE_{s_1} \star \dots \star \cE_{s_k}$. 
	 		Hence $\ch [\pi_{\ux,*} \underline{\Bbbk}_{\BS(\ux)}] = b_{\ux}$. 
	 		The equivalence follows immediately. 
	 		\par
	 		$(4) \iff (6)$:  This follows from the parity-vanishing property of parity sheaves and the definition of the character map in Section \ref{Ssec: parity sheaves} which imply $\chi(\cE_{x} \vert_y) = (-1)^{\ell(x)}{}^p h_{y,x}(1)$. 
	 	\end{proof}
	 
	 	\begin{rem}
	 		Computing local intersection forms is, in general, extremely computationally demanding; 
	 		see \cite[\S 4.9]{GJW23} for a discussion of some issues arising from the implementation of their algorithm. 
	 		To naively check $x$ is $p$-tight all intersection forms must be computed.   
	 		Statements (4) and (6) in Proposition \ref{prop: p-tight equiv} reduces this problem to computing \textit{one} $p$-Kazhdan-Lusztig polynomial or Euler characteristic, which may be more tractable. 
	 	\end{rem}
	 
		The elements $x \in W$ which can be $p$-tight are highly constrained.
	\begin{lem}
	\label{Lem: p-tight Bruhat}
		Suppose $x$ is $p$-tight for a realisation $\h$. 
		If $y<x$ then $y$ is $p$-tight.
	\end{lem}
	\begin{proof}
		Suppose $x$ is $p$-tight with reduced expression $\ux$. 
		Write $\ux$ as the concatenation $\ux = \uy \uw \uz$ for some expressions $\uy, \uw, \uz$. 
		Then $B_{\uy}$, $B_{\uw}$, and $B_{\uz}$ must each be $p$-tight, as any non-trivial summand in $B_{\uy}$, $B_{\uw}$, or $B_{\uz}$ would give rise to a non-trivial summand in $B_{\ux} \cong B_{\uy}B_{\uw}B_{\uz}$.
		The claim follows from the case when $\uw= s$.
	\end{proof}
	
	\begin{prop}
	\label{Prop: p-tight is fc}
			If $x$ is $p$-tight for some realisation $\h$ and prime $p$, then $x$ is fully commutative.  
	\end{prop}
		\begin{proof}
			Suppose $x$ is $p$-tight, but not fully-commutative. 
			Then some reduced expression $\ux$ can be expressed as the concatenation $\uy \uw_I \uz$ where $\uw_I$ is a reduced expression for the longest word in a Coxeter system $(W_I, I)$ of type $\textbf{I}_2(m)$ where $m> 2$.
			But $B_{\uw_I}$ is never $p$-tight when $m>2$, as the indecomposable Soergel (or Abe) bimodule $B_{w_I} \cong R \otimes_{R^I} R$ has (ungraded) rank $|W_I| < 2^{\ell(w_I)}$, which contradicts the argument in Lemma \ref{Lem: p-tight Bruhat}. 
	\end{proof}
	
	\begin{rem}
		The proposition implies, for a fixed realisation of any Coxeter system $W$, a sequence of inclusions:
		\begin{align*}
			\left\{
			\begin{array}{c}
				\text{tight}
				\\
				\text{elements}
			\end{array}
			\right\}
			\subseteq 
			\left\{
			\begin{array}{c}
				\text{$p$-tight}
				\\
				\text{elements}
			\end{array}
			\right\}
			\subseteq
			\left\{
			\begin{array}{c}
				\text{fully commutative}
				\\
				\text{elements}
			\end{array}
			\right\}. 
		\end{align*}
	\end{rem}
	\begin{rem}
		If $W_{pt}$ denotes the $p$-tight elements of $W$, the proof of Lemma \ref{Lem: p-tight Bruhat} shows $\langle B_x ~|~ x \notin W_{pt}\rangle_{\oplus,(1)}$ forms a tensor ideal in $\sD_{\text{BE}}(\h, \Bbbk)$. 
	\end{rem}
	
	We conclude by noting some examples of $p$-tightness in Cartan realisations of classical groups. 
	
	\begin{exmp}[The Hexagon permuation]
		Let $(W,S)$ be type $\tA_7$ with the usual ordering of simple reflections. 
		Define 
		\begin{align*}
			\underline{\text{hex}} = (s_5, s_6, s_7, s_3, s_4, s_5, s_6, s_2, s_3, s_4, s_5, s_1, s_2, s_3),
			&&
			\uw = (s_2, s_3, s_2, s_5, s_6, s_5).
		\end{align*} 
		Both $\underline{\text{hex}}$ and $\uw$ are reduced expressions; we denote by $\text{hex}$ and $w$ their respective images in $W$.
		One easily checks $b_{\underline{\text{hex}}} = b_{\text{hex}} + b_w$.
		Moreover, $\text{hex}$ is the Bruhat minimal element in type $\tA$ that is fully-commutative, but not tight \cite[\S1]{BW01}. 
		In \cite{WB12} it was shown ${}^2 b_{\text{hex}} = b_{\text{hex}} + b_w$ and $\text{hex}$ is the Bruhat minimal element such that ${}^p b_x \neq b_x$ for some $p$. 
		In particular, $\text{hex}$ is 2-tight, and $2$-tightness appears at the first possible opportunity in type $\tA$. 
	\end{exmp}
	
	\begin{exmp}
		Let $\h$ be a Cartan realisation of type $\tG_2$, and let $s_0$ be the simple reflection corresponding to the short root. 
		Then ${}^3 b_{s_0 s_1 s_0} = b_{\underline{s_0 s_1 s_0}}$, so $s_0 s_1 s_0$ is $3$-tight, but neither tight nor $2$-tight.
	\end{exmp}

\subsection{2-tightness in type $\tB_n$}
\label{Ssec: 2-tightness}

	We conjecture that $2$-tight elements are abundant in type $\tB_n$. 
		
	\begin{conj}
	\label{Conj: p-tight}
		Let $(W,S)$ be a Coxeter system of type $\tB_n$, $\h$ a Cartan realisation of the same type, $I\subset S$ a parabolic of type $\tA_{n-1}$.  
		Any $x \in  W^I$, considered as an element of $W$, is $2$-tight.  
	\end{conj}
	
	The validity of Conjecture \ref{Conj: p-tight}, together with the quotient construction of the antispherical category in Section \ref{Ssec: Geo Quotient Construction}, would imply a geometric proof of the characteristic-2 phenomena in Theorem \ref{Thm: antispherical cominuscule pKL basis}. 
	However, it is not clear from this Conjecture that every indecomposable parity sheaf in $\Parity_{Wh,I} (\check{B} \backslash \check{G})$ is perverse. 
	
	\begin{rem}
		The conjecture has been checked for all elements $x \in  W^I \subseteq W$ up to rank $7$ using the ASLoc and IHecke software of Gibson, Jensen, and Williamson \cite{GJW23}. 
	\end{rem}
	\begin{rem}
		By Lemma \ref{Lem: p-tight Bruhat} the conjecture holds if and only if $w_0 w_I$ is $2$-tight for the Cartan realisation of type $\tB_n$. 
	\end{rem}

\section{Acknowledgements and figures}
\subsection{Acknowledgements}

	This work appeared in the author's PhD thesis which was completed at the University of Sydney under the supervision of Geordie Williamson. 
	I thank Geordie and Noriyuki Abe for insightful comments on preliminary versions of this work. 
	It is a pleasure to dedicate this work to Janice and Tony. 
	The author was supported by the award of a Research Training Program scholarship.

\subsection{Figures}
\label{Ssec: Figures}

	We conclude by explaining how to interpret Figures \ref{fig:LG-AS-p0} and \ref{fig:LG-AS-p2}. 
	Recall the identification of elements in ${}^I W$ with subsets of $\{ 1, \dots , n \}$ given in Section \ref{Ssec: min coset reps LG}. 
	The map $\zeta :{}^I W \rightarrow \Z$ defined by $\zeta(x) := \sum_{t \in x} 2^t$ induces a linear order on ${}^I W$ which refines the Bruhat order.
	The antispherical $p$-Kazhdan-Lusztig polynomial ${}^p n_{y,x}$ is the $(\zeta(y),\zeta(x))$-entry of each matrix. 
	By Theorem \ref{Thm: Monomial property} each ${}^p n_{y,x}$ is a monomial; in these examples the monomials appearing are $v^0 , \dots , v^5$. 
	The shade of blue corresponds to the exponent of the monomial; deeper shades of blue indicate higher exponents. 




\end{document}